\documentclass[10pt, reqno]{amsart}
\usepackage{graphicx, amssymb, amsmath, amsthm}
\numberwithin{equation}{section}

\usepackage{xcolor}
\usepackage{todonotes}
\usepackage{mathtools}
\usepackage[inline]{enumitem}
\usepackage{mathrsfs} 
\usepackage{booktabs}
\usepackage[singlelinecheck=false]{caption}

\usepackage{tikz}

\usepackage[%
]{hyperref}

\allowdisplaybreaks
\newcommand{\dd}{\mathop{}\!\mathrm{d}}%
\let\Re\relax\DeclareMathOperator*{\Re}{Re}

\newcommand{\fd}{\mathfrak D} 

\newcommand{\myC}{\mathcal C}

\newcommand{\myE}{\mathcal E}
\newcommand{\myF}{\mathcal F}
\newcommand{\myG}{\mathcal G}
\newcommand{\myT}{\mathcal T}
\newcommand{\sh}{\mathscr H} 

\newcommand{\myH}{\mathbb H} 
\newcommand{\bessel}{\mathscr B} 
\newcommand{\hankel}{\mathscr H} 
\newcommand{\proj}{\mathbb P} 
\newcommand{\coloneq}{\coloneqq }
\newcommand{\R}{\mathbb{R}}
\newcommand{\C}{\mathbb{C}}
\newcommand{\la}{\mathcal{L}_a}
\DeclareMathOperator*{\slim}{s-lim}
\newcommand{\CZ}{Calder\'on--Zygmund}
\newcommand{\FDB}{Fa\'a di Bruno formula}
\newcommand{\acoloneq }[0]{\mathrlap\vcentcolon=}
\newcommand{\myone}{\textup{I}}
\newcommand{\mytwo}{\textup{II}}
\newcommand{\mythree}{\textup{III}}

\newtheorem{theorem}{Theorem}[section]
\newtheorem{prop}[theorem]{Proposition}
\newtheorem{lemma}[theorem]{Lemma}

\newtheorem{corollary}[theorem]{Corollary}
\newtheorem{conjecture}[theorem]{Conjecture}

\theoremstyle{definition}
\newtheorem{definition}[theorem]{Definition}
\newtheorem{remark}[theorem]{Remark}

\makeatletter
\newcommand{\Extend}[5]{\ext@arrow0099{\arrowfill@#1#2#3}{#4}{#5}}
\makeatother

\begin{document}
\title[$W^{s,p}$-boundedness of Wave operators]{The $W^{s,p}$-boundedness of stationary wave operators for the Schr\"odinger operator with inverse-square potential }

\author[C. Miao]{Changxing Miao}
\address{Institute of Applied Physics and Computational Mathematics\\ Beijing 100088\\ China}
\email{miao\_changxing@iapcm.ac.com}

\author[X. Su]{Xiaoyan Su}
\address{Laboratory of Mathematics and Complex Systems (Ministry of Education) \\ School of Mathematical Sciences\\
 Beijing Normal University, Beijing 100875,  China}
\email{suxiaoyan0427@qq.com}
\author[J. Zheng]{Jiqiang Zheng}
\address{Institute of Applied Physics and Computational Mathematics\\ Beijing 100088\\ China}
\email{zheng\_jiqiang@iapcm.ac.cn, zhengjiqiang@gmail.com}

\begin{abstract}
In this paper, we investigate the $W^{s,p}$-boundedness for stationary wave operators of the Schr\"odinger operator with inverse-square potential $$\mathcal L_a=-\Delta+\tfrac{a}{|x|^2}, \quad a\geq  -\tfrac{(d-2)^2}{4},$$
in dimension $d\geq 2$. We construct the stationary wave operators in terms of integrals of Bessel functions and  spherical harmonics, and  prove that they are $W^{s,p}$-bounded for certain $p$ and $s$ which depend on $a$.
As corollaries,  we solve some open problems associated with the operator $\mathcal L_a$, which include the dispersive estimates and the local smoothing estimates in dimension $d\geq 2$. We also generalize some known results such as the  uniform Sobolev inequalities, the equivalence of Sobolev norms and the Mikhlin multiplier theorem, to a larger range of indices. These results are important in the description of linear and nonlinear dynamics for dispersive equations with inverse-square potential.
\end{abstract}

 \maketitle

\begin{center}
 \begin{minipage}{100mm}
 \end{minipage}
 \end{center}

\setcounter{tocdepth}{1}
\tableofcontents
\section{Main results and applications}
\label{s:intro}
\subsection{Background}
For a real-valued potential $V$, the time-dependent wave operators for the pair of self-adjoint operators $H_0=-\Delta$ and $H\coloneq  H_0+V$ are defined as
\begin{align}
    W_{\pm} &= \slim_{t\to {\pm \infty}}e^{it H}e^{-itH_0}, \label{wpm-time-dep}\\
      W_{\pm}^* &= \slim_{t\to {\pm \infty}}e^{it H_0}e^{-itH} \label{wpm*-time-dep},
\end{align}
where  $\slim$ indicates the strong limit in $L^2(\mathbb R^d)$. In \cite{KatoKuroda}, Kato and Kuroda derived the following stationary formula for the wave operators,
\begin{align}\label{stationary wave operators}
  W_{\pm}=\int_{-\infty}^\infty \frac{\dd E(\lambda)}{\dd \lambda}(H-\lambda) R_0(\lambda\pm i0) \dd \lambda,
\shortintertext{
and
}\label{stationary adjoint wave operators}
  W_{\pm}^*=\int_{-\infty}^\infty \frac{\dd E_0(\lambda)}{\dd \lambda}(H_0-\lambda) R(\lambda\pm i0) \dd \lambda,
\end{align}
where $\dd E(\lambda), \dd E_0(\lambda),  R(\lambda\pm i0), R_0(\lambda\pm i0) $ are the spectral measures and resolvents associated to $H$ and $H_0$, respectively. The authors in \cite{KatoKuroda} proved that under some restrictive assumptions, the wave operators defined in time-dependent form \eqref{wpm-time-dep},\eqref{wpm*-time-dep} and stationary form \eqref{stationary wave operators}, \eqref{stationary adjoint wave operators} are identical. 

For suitable potentials $V$ e.g. short range potentials, one can prove that the time-dependent $W_{\pm}$ and $W_{\pm}^*$ exist and are complete, see  Agmon \cite{Ag}. This implies that they are isometries from $L^2(\mathbb R^d)$ to $L^2(\mathbb R^d)$, and there is no singular continuous spectrum. It is natural to ask: for what $p\neq 2$ are the wave operators $W_{\pm}, W_{\pm}^* $  bounded in  $L^{p}(\mathbb R^d)$, or  bounded in $W^{s,p}(\mathbb R^d)$? This is precisely the question that this paper tries to answer.

The $W^{s,p}$-boundedness of the wave operators was first investigated by Yajima \cite{Yajima,Yajima 1, Yajima 2, Yajima 3, Yajima 4, Yajima 5} under the assumption that zero energy is neither a eigenvalue nor a resonance. These results are established by using the perturbation method and the resolvent identity. However, the above results only apply to a restricted class of potentials. To relax the conditions, Beceanu and Schlag recently \cite{BS} established a structure formula for wave operators using an abstract version of Wiener's theorem, which implies the $L^p$-boundedness for more general potentials.


In this paper, we study the $W^{s,p}$-boundedness of the stationary wave operators associated to the operator $\la=-\Delta+\tfrac{a}{|x|^2}$ for $a\geq -\smash{\tfrac{(d-2)^2}4}$ in dimensions $d\geq 2$, for suitable $p\in[1,\infty]$ and $s\in\mathbb R$. We define $\la$ as the Friedrichs extension of the operator $-\Delta+\tfrac{a}{|x|^2}$ defined on  $C_c^\infty(\R^d\backslash\{0\})$. We refer the readers to \cite[Section X.3]{RS} for the general theory of such extensions. The restriction $a\geq -{\frac{(d-2)^2}4}$ guarantees positivity of $\la$; indeed, the associated form is
\begin{align}
  Q(f) = \int_{\R^d} \Big|\nabla f + \frac{\sigma x}{|x|^2}f\Big|^2 \,dx, \quad\text{where}\quad \sigma \coloneq  \frac{d-2}2-\Big[\Big(\frac{d-2}2\Big)^2+a\Big]^{\frac12}. \label{equ:qfdefsigma}
\end{align}
%
%

The 
elliptic operator $\la$ has appeared frequently in the literature. For example, the heat and Schr\"odinger flows associated with the elliptic operator $\la$ have been studied in the theory of combustion (see \cite{LZ}), and in quantum mechanics (see \cite{KSWW}). It is known \cite{BKM,Mizutani} that the time-dependent wave operators of $-\Delta$ and $\la$ exist and are complete for $a>-\tfrac{(d-2)^2}4$ in dimensions $d\geq3$.  For the case $d=2$ and $d\geq 3$ with $a=-\tfrac{(d-2)^2}4$, there is no scattering theory for $\la$. However, we can still define the stationary wave operators using the formula \eqref{stationary wave operators} and \eqref{stationary adjoint wave operators}, provided we have the spectral measures. 


One of the main motivations for investigating the $W^{s,p}(\mathbb R^d)$-boundedness for the stationary wave operators is to derive the Strichartz and dispersive estimates.
 Burq, Planchon, Stalker, and Tahvildar-Zadeh \cite{BPSS, BPST} proved  the  Strichartz estimates for the Schr\"odinger and wave equations with  inverse-square potentials, in space dimensions $d\geq2$ and $a>-\tfrac{(d-2)^2}4$. The endpoint case $a=-\tfrac{(d-2)^2}4$ was proven by Mizutani \cite{Miz-17}. Later, Fanelli, Felli, Fontelos, and Primo  proved  \cite{FFFP1, FFFP} the dispersive estimates for certain Schr\"odinger evolution, including the inverse-square potential in dimensions $d=2$ and $3$. The dispersive estimates and Strichartz estimates play an important role in the description of well-posedness and long-time behavior for the nonlinear dispersive equation with potentials, see \cite{KMVZZ-Sobolev,KMVZZ-NLS,KMVZ,LMM,MMZ,ZhaZhe,Z18}.
However, to the best of our knowledge, the dispersive estimates for the Schr\"odinger evolution $e^{it\la}$ in higher dimensions $d\geq4$, and wave propagator $\frac{\sin(t\sqrt{\la})}{\sqrt{\la}}$ in all dimensions  are still open.  Once we have the $W^{s,p}$-boundedness, by the intertwining property%
, we may transfer the dispersive estimates and Strichartz estimates of the operator  $-\Delta$ to the operator $\la$.

Another motivation for our work is to obtain Sogge's local smoothing estimates for  $e^{it\sqrt{\la}}$. 
This states that  the space regularity of half-wave propagator will be improved if we  average in time.
 Using the $W^{s,p}$-boundedness, we will  establish the corresponding estimates for $e^{it\sqrt{\la}}$, under the assumption that local smoothing estimates have been proven for $e^{it\sqrt{-\Delta}}$.

Our results also imply the uniform Sobolev inequality for $\la$, the equivalence of Sobolev norms adapted to $\la$, and a Mikhlin-type multiplier theorem  for $\la$. These results improve the corresponding results in \cite{MZZ, KMVZZ-Sobolev} to a more general range of indices. See Subsection \ref{ss:application} for more details.

One thing we want to mention is  that a number of other estimates related to $\la$ also can be directly deduced from the $W^{s,p}$-boundedness of the stationary wave operators, but to keep our paper at a reasonable length, we  omit the details.

%

\subsection{Main results}

\subsection*{Main result 1.}
First, we give the $L^p$-boundedness of the stationary wave operators, which is also the first step  to establish the $W^{s,p}$-boundedness.
\begin{theorem}[$L^p$-boundedness]\label{thm:main} Let $d\geq 2$ and $a\geq -\tfrac{(d-2)^2}4$. Then, the stationary wave operators $W_{\pm}$  and $W_{\pm}^*$ are bounded in $L^p(\R^d)$ with
\begin{equation}\label{equ:pcondlp}
  \max\big\{0,\tfrac{\sigma}{d}\big\}<\tfrac1p<\min\big\{1,\tfrac{d-\sigma}d\big\},
\end{equation}
 where $\sigma$ is as in \eqref{equ:qfdefsigma}.
\end{theorem}

\begin{remark}
The condition \eqref{equ:pcondlp} is equivalent to
\begin{equation}\label{equ:pthmconde}
  \begin{cases}
    \phantom{_0}1<p<+\infty&\text{if}\quad a>0,\\
    p_0'<p<p_0& \text{if}\quad a<0,
  \end{cases}
\end{equation}
where  $p_0=\frac{d}{\sigma}$.	
\end{remark}

\subsection*{Proof sketch for Theorem \ref{thm:main}} To prove the above results, we cannot use the perturbative method as Yajima and Beceanu did, due to the strong singularity at zero. However, notice that the potential term $V=\smash{\frac{a}{|x|^2}}$ is homogeneous of degree $-2$, which has the same scaling as $-\Delta$.  This key property has been employed in \cite{BPSS,BPST,PSS} to prove the dispersive and Strichartz estimates for the solutions of Schr\"odinger equation and wave equation with inverse square potentials. 
%
%
In particular, \cite{PSS} proved Strichartz estimates for the wave equation by constructing some intertwining operators which enabled them to transfer the results for $-\Delta$ to $\la$, but under the radial assumption. Our contribution is to remove this assumption, as we now explain.

Using the special structure of our operator, on each subspace of spherical harmonics, the operator $\la$ is equivalent to a Bessel operator which can be diagonalized  using the eigenfunction expansion in terms of spherical waves.  This eigenfunction expansion serves as the ``distorted Fourier transform'' for $\la$.  It is known \cite{Kuroda} that once the distorted Fourier transforms $ \mathscr F_{\pm}$ are well-defined, the wave operators can be given by $W_{\pm}=\mathscr F_{\pm}^* \mathscr F$. This formula inspires us to construct the stationary wave operators directly.

Hence, we construct two intertwining operators  $W$ and $W^*$ in $L^2(\mathbb R^d)$ for $-\Delta$ and $\la$ using the decomposition into spherical harmonics,
\begin{align}
      Wf(r, \omega)=\int_0^\infty \sum_{k=0}^\infty K_k(r,s ) \proj_k f(r, \omega)\frac{\dd r}{r},\label{introduction $W$} \\
       W^*f(r, \omega)=\int_0^\infty \sum_{k=0}^\infty K_k^*(r,s ) \proj_k f(r, \omega)\frac{\dd r}{r},
  \end{align}
where $K_k(r,s ),$ and  $K_k^*(r,s )$ are given by  integrals involving two Bessel functions, and the projection operator $\proj_k$ is defined as in \eqref{equ:projkfdef} below.
%
  It is easy to check that $W$  and $W^*$ are unitary in $L^2(\mathbb R^d)$. In addition, we are able to prove that $W=W_{-}$ and $W^*=W_{-}^*$ defined by \eqref{stationary wave operators} and \eqref{stationary adjoint wave operators}.

To verify that  $W$  and $W^*$ are bounded in $L^p(\mathbb R^d)$, 
it suffices to prove two modified operators $\widetilde W$ and $\widetilde W^*$ are bounded in $L^p(\frac{\dd r}{r}\cdot \dd \omega)$.
One of the key ingredients in proving the $L^p$-boundedness is a discrete multiplier theorem (Theorem \ref{multiplier theorem for spherical harmonic decomposition} below) for  spherical harmonics.
If for any fixed $r$ and $s$, the coefficients $\widetilde K_k(r,s)$ 
satisfy the assumptions of this theorem, then it follows that
 \begin{align*}
    \|\widetilde Wf(r, \omega)\|_{L^p(\dd \omega)} \leq M(r,s)\Big\|\sum_{k=0}^\infty \proj_k f(r, \omega)\Big\|_{L^p(\dd \omega)},
\end{align*}
where $M(r, s)$ is the upper bound for $ \widetilde K_k(r,s)$ uniformly in $k$. If $M(r,s) $ is bounded in $L^1(\frac{\dd r}{r})$ on the diagonal $r=s$, then by Young's inequality, we would have that $\widetilde W$ is bounded in $L^p\big(\frac{\dd r}{r}\cdot \dd \omega\big)$.
 However, this is \emph{not} true. In fact, as we will see from \eqref{singularity of kernels} below, for each fixed $k$, $\widetilde K_k(r,s)$ has a singularity of the type
\begin{align*}
    c(r-s)^{-1}
\end{align*}
as $r$ approaches $s$, which cannot be controlled by a $L^1(\frac{\dd r}{r})$ function on the diagonal. This fact also suggests that the intertwining operators $\widetilde W$ and $\widetilde W^*$ are \CZ{} singular integral operators on the radial part in each subspace. Hence, our second ingredient in proving the $L^p$-boundedness is  the \CZ{} theory.

If we directly estimate each $k$th part in $r$ using  \CZ{} theory and Minkowski's inequality, we would have
\begin{align*}
   \| \widetilde Wf(r, \omega)\|_{L^p(\frac{\dd r}{r})} \leq \sum_{k=0}^\infty \big\|\proj_k f(r, \omega)\big \|_{L^p(\frac{\dd r}{r})}.
\end{align*}
That is, the $L^p$ norm would go inside of the sum, making it impossible to recover the $L^p$ norm of $f$.   However, if we estimate the angular part first, the resulting bound is expressed in terms of $1/|r-s|$ instead of $1/(r-s)$ and we will not be able to take advantage of the cancellation in $\widetilde K_k(r,s)$.

One special case which can be easily proved is when the kernels are of the variable-separated form $\widetilde{K}_k(r, s)=F(r, s) M(k)$, where $F(r, s)$ defines a \CZ{} operator and $\{M(k)\}_{k\ge 0}$ satisfies the condition of Theorem \ref{multiplier theorem for spherical harmonic decomposition}. In this situation, the operator $\widetilde{W}$   can be written as
\begin{align*}
     \widetilde Wf(r, \omega)=\int_0^\infty F(r, s) \sum_{k=0}^\infty M(k) \proj_k f(r, \omega)\frac{\dd r}{r}.
\end{align*}
Using  \CZ{}  theory, we have
\begin{align*}
   \| \widetilde Wf(r, \omega)\|_{L^p(\frac{\dd r}{r})}  \leq \Big\|\sum_{k=0}^\infty M(k) \proj_k f(r, \omega)\Big\|_{L^p(\frac{\dd r}{r})}.
\end{align*}
Furthermore, by Theorem \ref{multiplier theorem for spherical harmonic decomposition}, we would get the boundedness in $L^p(\frac{\dd r}{r}\cdot \dd \omega)$. This easy case motivates us to approximate the $\widetilde{K}_k(r,s)$ with variable-separated kernels to capture the singularity in the radial part. The remainder term has an $L^1(\frac{\dd r}{r})$ kernel which guarantees that we can estimate the angular part first.

 Having the previous ideas in mind, the main difficulty in proving Theorem \ref{thm:main} is in checking that the conditions for the discrete multiplier theorem  are valid. To obtain this, we need to use the power series expansion for hypergeometric functions, and the asymptotic expansions for Gamma functions and its derivatives. This step forms the main part of the proof.

\subsection*{Main result 2.} Now we give our second theorem, which improves the previous theorem to the scale of Sobolev spaces. Define
\begin{align} \label{df of nu zero}
    \nu_0=\sqrt{\big(\tfrac{d-2}2\big)^2+a}.
\end{align}

\begin{theorem}[$W^{s,p}$-boundedness]
\label{thm:main 2} Assume that  $d\geq 2$ and $a\geq -\tfrac{(d-2)^2}{4}$.
\begin{enumerate}
\item Let $-d<\alpha <2+2\nu_0$.  The stationary wave operators $W_{\pm}$ are bounded in ${W}^{\alpha,p}(\mathbb R^d)$, if $p$ satisfies
    \begin{equation}\label{equ:psig}
     \max\Big\{0, \frac{\sigma}{d} , \frac{\sigma+\alpha}{d}\Big\}<\frac{1}{p}<\min\Big\{1, \frac{d-\sigma}{d}, \frac{d+\alpha}{d}\Big\}.
    \end{equation}
    \item
Let $-2\nu_0-2<\beta <d$.  The stationary wave operators $W_{\pm}^*$ are bounded in  ${W}^{\beta,p}(\mathbb R^d)$, if $p$ satisfies
 \begin{equation}\label{equ:rangpbeta}
   \max\Big\{0, \frac{\sigma}{d}, \frac{\beta}{d} \Big\}<\frac{1}{p}<\min\Big\{1, \frac{d-\sigma}{d}, \frac{d-\sigma+\beta}{d}\Big\}.
 \end{equation}
\end{enumerate}
\end{theorem}

\subsection*{Proof sketch for Theorem \ref{thm:main 2}}
To prove the $W^{s,p}$-boundedness of the wave operators, we introduce the following Riesz-type operators $ \mathcal R^{\alpha}$  formally given by $$\mathcal R^\alpha =  (-\Delta)^{\alpha/2} \la ^{-\alpha/2}, \quad -d<\alpha<2\nu_0+2, \ \alpha \neq 0,$$
and their inverse operators
$$\mathcal R^{-\beta}=  \la ^{\beta/2}(-\Delta)^{-\beta/2}, \quad -2-2\nu_0<\beta<d, \ \beta \neq 0.$$
The precise definitions are given in Section \ref{s:proof-of-wsp-bddness}.

 For these operators, we have the following $L^p$-boundedness results.
\begin{prop}\label{bdd for Riesz transform}Let $d\geq 2$ and $a\geq -\tfrac{(d-2)^2}{4}$.
\begin{enumerate}
    \item Let $-d<\alpha <2+2\nu_0$ and $\alpha \neq 0$. The operator $\mathcal R^\alpha$ given by \eqref{df of RT} is bounded  in $L^p(\mathbb R^d)$ with
\begin{equation}\label{equ:lprals}
 \max \Big\{0, \frac{\sigma+\alpha}{d}\Big\}<\frac{1}{p}<\min \Big\{1, \frac{d-\sigma}{d},  \frac{d+\alpha}{d}\Big\}.
\end{equation}
\item\label{Main theorem 2}Let  $-2\nu_0-2<\beta <d $ and $\beta\neq 0$. The operator $\mathcal R^{-\beta}$ given by \eqref{df of IRT} is bounded  in $L^p(\mathbb R^d)$ with
    \begin{equation}\label{equ:pbedbe}
      \max\Big\{0, \frac{\beta}{d}, \frac{\sigma}{d}\Big \}<\frac{1}{p}<\min \Big\{1, \frac{d-\sigma+\beta}{d} \Big\}.
    \end{equation}
\end{enumerate}
\end{prop}

Once we prove Proposition \ref{bdd for Riesz transform}, combining it with the $L^p$-boundedness of wave operators immediately gives the $ W^{s, p}$-boundedness for $W_{\pm}$ and $W^*_{\pm}$: for $s\neq 0$  satisfying $-d<s<2\nu_0+2$,  and
$p$ as in \eqref{equ:psig},  we have
\begin{gather}
    \begin{alignedat}{100}\|W_{\pm}f\|_{\dot W^{s, p}(\mathbb R^d)} &=\big\||\nabla |^s W_{\pm}|\nabla |^{-s} |\nabla |^s f \big\|_{L^p(\mathbb R^d)}\\
    &=\big\||\nabla |^s \la ^{-s/2}W_{\pm} |\nabla |^s f \big\|_{L^p(\mathbb R^d)}\\
    &\lesssim \|W_{\pm}|\nabla |^s f\|_{L^p(\mathbb R^d)}\\
    &\lesssim \|f\|_{\dot W^{s, p}(\mathbb R^d)}.
    \end{alignedat} \label{reduction}
\end{gather}
This proves the result, because boundedness in $L^p
 (\mathbb R^d)$   and $\dot W^{s, p}(\mathbb R^d)$ implies  boundedness in $W^{s,p}(\mathbb R^d)$. Hence, to prove Theorem \ref{thm:main 2}, we only need to prove Proposition \ref{bdd for Riesz transform}.

The operators $\mathcal R^\alpha $ and $\mathcal R^{-\beta}$ are essentially Mellin multipliers on each subspace of spherical harmonics. The kernel functions of $\mathcal R^{\alpha}$ and $\mathcal R^{-\beta}$ are given by Fox $H$-functions $H_{4,4}^{2,2}(k;r,s)$, which plays the same role as the hypergeometric functions in the study of wave operators. Using the residue theorem, $H_{4,4}^{2,2}$ can be written as power series when $0<\frac rs<1$ and when $\frac rs>1$. At the point $r=s$, we can prove that they have a singularity of type
\begin{align*}
    c(r-s)^{-1},
\end{align*}
see Lemma \ref{sin at $r=s$.} below. These facts allow us to essentially re-use the proof of the $L^p$-boundedness of $W_{\pm}$ and $W^*_{\pm}$ to show the boundedness of $\mathcal R^\alpha$ and $\mathcal R^{-\beta}$. The details are given in Section \ref{s:proof-of-wsp-bddness}.

The $L^p$-boundedness of $\mathcal R^{\alpha}$  and $\mathcal R^{-\beta}$ is also closely related to the equivalence of Sobolev norms as studied in \cite{KMVZZ-Sobolev}, where the authors utilized generalized Gaussian kernel estimates  for the operator $e^{-t \la}$ and a perturbative method to obtain the equivalence in a restricted range $\alpha\in(0, 2)$. In this paper, we directly handle these operators using spherical harmonics and Fox $H$-functions, which enable us to extend the range of the fractional powers to  $-2-2\nu_0<\beta<d$ and $-d<\alpha<2+2\nu_0$. These bounds on $\alpha$ and $\beta$ arise naturally  in order to guarantee that the Fox $H$-functions are well-defined. In particular,  the equivalence of Sobolev norms hold for a larger range of $\alpha$. To be precise, we have the following:

\begin{corollary}[Equivalence of Sobolev norms]
\label{c:equiv-of-sob-norm} Let $d\geq 2$ and $a\ge -(\frac{d-2}2)^2$.
     \begin{enumerate}
    \item
 Let $-d<\alpha <2+2\nu_0$ and $\alpha \neq 0$. Further assume that $p$ satisfies
\begin{equation*}
 \max \Big\{0, \frac{\sigma+\alpha}{d}\Big\}<\frac{1}{p}<\min \Big\{1, \frac{d-\sigma}{d},  \frac{d+\alpha}{d}\Big\}.
\end{equation*}
Then for all $f \in C_c^\infty(\mathbb R^d \setminus \{0\})$,
\begin{align}
    \|(-\Delta)^{\alpha/2}f\|_{L^p(\mathbb R^d)} \lesssim  \|\la ^{\alpha/2}f\|_{L^p(\mathbb R^d)}. 
\end{align}
\item Let $-2\nu_0-2<\beta <d $ and $\beta \neq 0$. Further assume that p satisfies
  \begin{equation*}
      \max\Big\{0, \frac{\beta}{d}, \frac{\sigma}{d}\Big \}<\frac{1}{p}<\min \Big\{1, \frac{d-\sigma+\beta}{d} \Big\}.
    \end{equation*}
Then for all $f \in C_c^\infty(\mathbb R^d \setminus \{0\})$,
    \begin{align}\label{sobolev norm equiv}
    \|\la ^{\beta/2}f\|_{L^p(\mathbb R^d)} \lesssim  \|(-\Delta) ^{\beta/2}f\|_{L^p(\mathbb R^d)}.
\end{align}
\end{enumerate}

\end{corollary}
\begin{remark}
 We want to emphasize  part \emph{(2)} of Corollary \ref{c:equiv-of-sob-norm} in dimension $d=2$: in this case, $a\ge 0$, and our result simply says that $\mathcal R^{-1}$ is  $L^p$-bounded for $1<p<2$. In particular, this implies that for $1<p<2$
\begin{align}\label{failure of hardy's inequality}
    \big\| |x|^{-1}f\big\|_{L^p(\mathbb R^2)} \lesssim \|\nabla f\|_{L^p(\mathbb R^2)}.
\end{align}
This is consistent with the fact that Hardy's inequality, which is \eqref{failure of hardy's inequality} at the $p=2$ endpoint, is false. In this sense, the range for which we proved the inequality \eqref{sobolev norm equiv} is sharp in dimension two.
\end{remark}

%
%


%

\subsection{Applications}\label{ss:application}

Since $\la$ is a positive self-adjoint  operator on $L^2(\mathbb R^d)$. By the spectral theorem, for any bounded Borel function $F$,  we can define the operator $F(\la)$ by the formula
\begin{align*}
   F(\la)=\int_0^\infty F(\lambda)\dd E_{\la}(\lambda).
\end{align*}
By the intertwining property $W_{\pm} (-\Delta)= \la W_{\pm}$  and the spectral theorem, we have $W\dd E_{-\Delta}(\lambda)= \dd E_{\la}(\lambda) W $. Hence,
\begin{equation}\label{equ:intpro}
  F(\la)= W_{\pm}F(H_0)W_{\pm}^{*}.
\end{equation}

\subsection*{Application 1: Dispersive estimates.} \label{app-1}

 Let $e^{it \la}f$ be the solution of the free Schr\"odinger equation with inverse-square potential and initial data $f$.
Using \eqref{equ:intpro} and the  $L^{p}$-$L^{p'}$ estimates for the free Schr\"odinger operator $e^{-it\Delta}$,
\begin{equation*}
  \big\|e^{-it\Delta}f\big\|_{L^{p'}(\R^d)}\lesssim |t|^{-d(\frac1p-\frac12)}\|f\|_{L^{p}(\R^d)},\quad \forall\;t\neq0
\end{equation*}
with $1\leq p \leq 2,$  and by Theorem \ref{thm:main}, we obtain the dispersive estimates for the Schr\"odinger operator $e^{-it\la}$   as follows:
\begin{corollary}[Dispersive estimate for Schr\"odinger equations]\label{cor:disest of sch}
Let $d\geq 2$, $a\geq-\tfrac{(d-2)^2}{4}$ and $\tfrac12\leq\tfrac1p<\min\big\{1,\tfrac{d-\sigma}d\big\}$. Then, there holds
\begin{align}\label{equ:dispes}
   \big\|e^{-it\la}f\big\|_{L^{p'}(\R^d)}\lesssim & |t|^{-d(\frac1p-\frac12)}\|f\|_{L^{p}(\R^d)},\quad \forall\;t\neq0.
\end{align}
\end{corollary}

%
%

For the wave propagator, the following dispersive estimate
\begin{equation}
  \Big\|\tfrac{\sin(t\sqrt{-\Delta})}{\sqrt{-\Delta}} f\Big\|_{L^{p'}(\R^d)}\lesssim |t|^{-(d-1)(\frac1{p}-\frac{1}{2})}\|f\|_{\dot{W}^{\theta_{p, d},p}(\R^d)},\quad \forall\;t\neq0,
\end{equation}
 is valid, where
\begin{align}\label{theta p d}
   \theta_{p,d}=\tfrac{d+1}{p}-\tfrac{d+3}{2}.
\end{align}
%
%
This inequality together with Theorem \ref{thm:main 2}  yields the dispersive estimates for the wave propagator $\tfrac{\sin(t\sqrt{\la})}{\sqrt{\la}}$.
\begin{corollary}[Dispersive estimate for wave equations]\label{cor:disestwave}
Let $d\geq 2$,  $a\geq-\tfrac{(d-2)^2}{4}$  and $\tfrac12\leq\tfrac1p<\min\Big\{1,\tfrac{d-\sigma}d\Big\}$. If further assume that
$\frac{1}{p}>\frac{3}{2}-\frac{d}{2}+\sigma$,  then  there holds
\begin{align}\label{equ:dispwava-1}
   \Big\|\tfrac{\sin(t\sqrt{\la})}{\sqrt{\la}}f\Big\|_{L^{p'}(\R^d)}\lesssim& |t|^{-(d-1)(\frac1p-\frac12)}\big\|f\big\|_{\dot W^{\theta_{p, d} , p}},\quad \forall\;t\neq0.
\end{align}
\end{corollary}
\begin{proof} By \eqref{equ:intpro}, we have
$\tfrac{\sin(t\sqrt{\la})}{\sqrt{\la}}=W_{\pm} \tfrac{\sin(t\sqrt{-\Delta})}{\sqrt{-\Delta}} W^*_{\pm}$.
Hence,
   \begin{align}
      \Big \|\tfrac{\sin(t\sqrt{\la})}{\sqrt{\la}}f \Big\|_{L^{p'}(\mathbb R^d)}&= \Big \|W_{\pm} \tfrac{\sin(t\sqrt{-\Delta})}{\sqrt{-\Delta}}W_{\pm}^*f \Big \|_{L^{p'}(\mathbb R^d)} \notag  \\
       &\lesssim \big\| \tfrac{\sin(t\sqrt{-\Delta})}{\sqrt{-\Delta}}W_{\pm}^*f\big\|_{L^{p'}(\mathbb R^d)} \label{pf-cor1.9-1} \\
       &\lesssim  |t|^{-(d-1)(\frac1p-\frac12)} \|W_{\pm}^*f\|_{\dot W^{\theta_{p, d},p}(\mathbb R^d)} \notag \\
       &\lesssim  |t|^{-(d-1)(\frac1p-\frac12)}\|f\|_{\dot {W}^{\theta_{p, d},p}(\mathbb R^d)},
        \label{pf-cor1.9-2}
       \end{align}
       where in \eqref{pf-cor1.9-1}, we used that  $\max \Big\{0, \frac{\sigma}{d}\Big\}<\frac{1}{p'}<\min \Big\{1, \frac{d-\sigma}{d}\Big\}$, and in \eqref{pf-cor1.9-2}, we used that
       \[  -2\nu_0-2<\theta_{p,d }<d, \max\Big\{0, \tfrac{\theta_{p,d}}{d}, \tfrac{\sigma}{d}\Big\}<\frac{1}{p}<\min \Big\{1, \tfrac{d-\sigma}{d},  \tfrac{d-\sigma+\theta_{p, d}}{d}\Big\}. \qedhere\]        
  \end{proof}


\begin{remark} As we  mentioned before, Fanelli, Felli, Fontelos, and Primo have proven \eqref{equ:dispes} with $d\in\{2,3\}$ in \cite{FFFP1,FFFP}, while we prove the dispersive estimates for $e^{it\la}$ in all  dimensions $d\geq4$.  We stress that the dispersive estimates for the  propagator $\tfrac{\sin(t\sqrt{\la})}{\sqrt{\la}}$ are new for all dimensions $d\geq 2$.
\end{remark}

\subsection*{Application 2: Strichartz estimates.}\label{app-2}

Using the classical Strichartz estimates for the free Schr\"odinger operator $e^{-it\Delta}$ (see Keel--Tao \cite{KeelTao} and Strichartz \cite{St})
\begin{align*}
  \|e^{-it\Delta}f\|_{L_t^q L_x^r(\R\times\R^d)}\lesssim\|f\|_{L^2(\R^d)},
\end{align*}
with $(q,r)\in\Lambda_0\coloneq \big\{(q,r):\;q,r\geq2,\;\tfrac2q=d\big(\tfrac12-\tfrac1r\big),\;(q,r,d)\neq (2,\infty,2)\big\}$, we recover the Strichartz estimates for $e^{it \la}$. 


\begin{corollary}[Strichartz estimates for Schr\"odinger equations] Let $d\geq 2$ and $a\geq -\tfrac{(d-2)^2}4$. Then for any $(q,r)\in\Lambda_0$,  we have
	\begin{align}
		\|e^{it \la}f\|_{L^q_t L^r_x(\R\times\R^d)}\lesssim  \|f\|_{L^2(\R^d)}.
	\end{align}
\end{corollary}
\begin{proof} Using the condition $ \frac{1}{r} \in \left(\frac{\sigma}{d}, \frac{d-\sigma}{d}\right)$, the spectral theorem gives that
\begin{align*}
    \|e^{it \la}f\|_{L^q_t L^r_x(\R\times\R^d)}& =\|We^{-it \Delta }W^*f\|_{L^q_t L^r_x(\R\times\R^d)}\\
    &\lesssim \|e^{-it \Delta }W^*f\|_{L^q_t L^r_x(\R\times\R^d)} \quad \\
  & \lesssim \|W^*f\|_{L^2(\mathbb R^d)}
  \\  & \lesssim \|f\|_{L^2(\mathbb R^d)}.\qedhere
\end{align*}
\end{proof}

Let $I$ be a time interval and let $u: I\times \mathbb R^d\to \mathbb R$ be a Schwartz solution to the homogeneous wave equation $\partial_{tt} u-\Delta u=0$ with the initial data $(f, g)$. Then we have the estimates
\begin{align*}
    \|u\|_{L^q_t  L_x^r(I \times \mathbb R^d)}\lesssim \|f\|_{\dot W^{s, 2}(\mathbb R^d)}+\|g\|_{\dot W^{s-1, 2}(\mathbb R^d)},
\end{align*}
whenever $s\geq 0$, $2\leq  q\leq \infty$ and $2\leq r<\infty$ obey the scaling condition
\begin{align*}
    \frac{1}{q}+\frac{d}{r}=\frac{d}{2}-s,
\end{align*}and the wave admissibility conditions
\begin{align*}
    \frac{1}{q}+\frac{d-1}{2r} \leq \frac{d-1}{4}.
\end{align*}
One could also obtain the following  Strichartz estimates for the wave equation with inverse-square potential using Theorem \ref{thm:main 2}.

\begin{corollary}[Strichartz estimates for wave equation]    Let $d\geq 2$, $a\geq -\frac{(d-2)^2} {4}$, and $u$ be the solution of wave equation $\partial_{tt}u-\la u=0$ with Cauchy data $(f, g) \in \dot W^{s, 2}(\mathbb R^d) \times \dot W^{s-1, 2}(\mathbb R^d)$ with $0<s<\frac{d}{2}$.  Let $2\leq q\leq \infty$ and $2\leq r<\infty$      such that $\frac{2}{q}+\frac{d-1}{r}\leq \frac{d-1}{2}\  (q>2, \ \text{if} \ d=3 \  \text{and} \ q>4, \ \text {if }\ d=2)$.  Then,
\begin{align}
 \|u\|_{L^q_t  L_x^r(I \times \mathbb R^d)}\lesssim \|f\|_{\dot W^{s, 2}(\mathbb R^d)}+\|g\|_{\dot W^{s-1, 2}(\mathbb R^d)},
   \end{align}
where $s=\frac{d}{2}-\frac{1}{q}-\frac{d}{r}$.

\end{corollary}

\begin{proof}  Let $\tilde u$ be the solution of the wave equation without potential with Cauchy data $(W^*f, W^*g)$, then  $(W^*f, W^*g) \in \dot W^{s, 2} \times \dot W^{s-1, 2}$. By $W^{s,p}$-boundedness of $W$ and $W^*$, since  we have
\[ \max\Big\{0, \frac{\sigma}{d} , \frac{s}{d}\Big\}<\frac{1}{r}<\min\Big\{1, \frac{d-\sigma}{d},  \frac{d-\sigma+s-1}{d}\Big\}, \]
it follows that
  \begin{align*}
      \|u\|_{L^{q}_t L_x^{r}}& = \|W\tilde u\|_{L^q_t  L_x^r(I \times \mathbb R^d)}\leq  \|\tilde u\|_{L^q_t  L_x^r(I \times \mathbb R^d)}
      \\
      &\leq  C(\|W^*f\|_{\dot W^{s, 2}(\mathbb R^d)}+\|W^*g\|_{\dot W^{s-1, 2}(\mathbb R^d)})\\
      &\leq  C(\|f\|_{\dot W^{s, 2}(\mathbb R^d)}+\|g\|_{\dot W^{s-1, 2}(\mathbb R^d)}).\qedhere
  \end{align*}
\end{proof}

\subsection*{Application 3: Sogge's local smoothing estimates.}
By the $W^{s,p}$-boundedness of wave operators,   we can extend some results of conjectures in harmonic analysis associated to  $-\Delta$, such as the local smoothing conjecture, Bochner--Riesz conjecture,  and restriction conjecture,  to the operator $\mathcal{L}_a$.  For simplicity,  we take Sogge's local smoothing conjecture  for example.

\begin{conjecture}[Sogge's Local smoothing conjecture, \cite{Sog91}]  For $d\geq 2$ and $a\geq -\frac{(d-2)^2}{4}$, the inequality
\begin{align}\label{equ:sogglosme}
 \left (  \int_{0}^1 \|e^{it \sqrt{-\Delta}}f\|_{L^p(\mathbb R^d)}^p \dd t\right)^{\frac{1}{p}} \lesssim \|f\|_{W^{s_p -\varepsilon , p}(\mathbb R^d)},
\end{align}
holds for all
\begin{align*}
   \varepsilon <
    \begin{cases} s_p, & \text{if} \quad 2<p\leq \frac{2d}{d-1},\\
     \frac{1}{p}, & \text{if} \quad \frac{2d}{d-1} <p<\infty.
    \end{cases}
\end{align*}
\end{conjecture}

For the history and progress on the local smoothing conjecture, we refer the readers to \cite{BHS,Gao}.
%
By  the intertwining property \eqref{equ:intpro}  and Theorem \ref{thm:main 2}, we establish the local smoothing estimates for the propagator $e^{it \sqrt{\la}}$ as follows:
\begin{corollary}[Local smoothing estimates for $e^{it \sqrt{\la}} $]   Let $d\geq 2$, and $a\geq-\tfrac{(d-2)^2}4.$ Suppose that \eqref{equ:sogglosme} holds for $p\geq p_{\textup{known}}$.  Then, for $p\geq p_{\textup{known}}$ and
$\max\big\{\tfrac{s_p-\varepsilon}{d}, \tfrac{\sigma}{d}\big\}<\frac{1}{p}\leq\frac12$, the following estimate is valid:
    \begin{align}
       \left (  \int_{0}^1 \|e^{it \sqrt{\la}} f\|_{L^p(\mathbb R^d)}^p \dd t\right)^{\frac{1}{p}} \lesssim \|f\|_{{W}^{s_p -\varepsilon , p}(\mathbb R^d)}.
    \end{align}
\end{corollary}

\begin{proof} Under the assumption for $p$, we know that $W$ is $L^p$-bounded  and $W^*$ is $W^{s_p-\varepsilon,p}$-bounded. Hence,
  \begin{align*}
      \left (  \int_{0}^1 \|e^{it \sqrt{\la}} f\|_{L^p(\mathbb R^d)}^p \dd t\right)^{\frac{1}{p}}&=\left (  \int_{0}^1 \|We^{it \sqrt{-\Delta}}W^*f\|_{L^p(\mathbb R^d)}^p \dd t\right)^{\frac{1}{p}}\\
     & \lesssim \left (  \int_{0}^1 \|e^{it \sqrt{-\Delta}}W^*f\|_{L^p(\mathbb R^d)}^p \dd t\right)^{\frac{1}{p}}\\
     &\lesssim \|W^*f\|_{{W}^{s_p -\varepsilon , p}(\mathbb R^d)}\\
     &\lesssim \|f\|_{{W}^{s_p -\varepsilon , p}(\mathbb R^d)}.\qedhere
  \end{align*}
\end{proof}

\subsection*{Application 4: Uniform resolvent inequality.}
 For the Laplacian  $-\Delta$, the uniform Sobolev inequality  has been proved by Kenig--Ruiz--Sogge \cite{KRS} which is given as follows: for $z\in\C\setminus \R^+$ and $f\in C_0^\infty(\R^d)$,
\begin{equation}
\label{uniform_Sobolev}
\left\|(-\Delta-z)^{-1}f\right\|_{L^{q} (\R^d)}\leq
C|z|^{\frac d2(\frac1p-\frac1q)-1}\|f\|_{L^{p}(\R^d)},
\end{equation}
where $d\geq3$ and $(p,q)$ satisfies the conditions
\begin{align}
\label{p_q_0}
\frac{2}{d+1}\le \frac1p-\frac1q\le \frac2d,\quad  \frac{2d}{d+3}<p<\frac{2d}{d+1},\text{\ \ and\ \  }  \frac{2d}{d-1}<q<\frac{2d}{d-3}.
\end{align}
See also the paper of Guti\'errez \cite{Gut}, which proved that the condition \eqref{p_q_0} is  sharp. More recently,  Ev\'equoz \cite{Ev} applied the method  of \cite{Gut}  to show the uniform resolvent estimate \eqref{uniform_Sobolev} in dimension $d=2$ provided that $(1/p,1/q)$ is contained in the pentagon
\begin{equation}\label{pq2}
	\big\{\big(\tfrac1p,\tfrac1q\big): \tfrac23\le \tfrac1p-\tfrac1q <1, \, \tfrac34<\tfrac1p\le1, \, 0\le \tfrac1q<\tfrac14 \big\}.	
\end{equation}

As a consequence of Theorem \ref{thm:main}, one can extend \eqref{uniform_Sobolev} to the operator $\la$. To state our result, with $\nu_0$  as in \eqref{df of nu zero}, we write
$$
\mu_a\coloneq \min\Big\{\frac12 ,\nu_0\Big\}.
$$
\begin{corollary} [Uniform Sobolev inequality]\label{thm:unisobine} For $d\geq3$, we  suppose  \begin{equation}\label{p_q}
\frac{2}{d+1}\le \frac1p-\frac1q\le \frac2d,\quad  \frac{2d}{d+2(1+\mu_a)}<p<\frac{2d}{d+1},\quad  \frac{2d}{d-1}<q<\frac{2d}{d-2(1+\mu_a)}.
\end{equation}
While for $d=2$, we assume that $(1/p,1/q)$ is as in \eqref{pq2}.
Then there exists a positive constant $C$ such that for $ z\in\C\setminus \R^+,\ f\in C_0^\infty(\R^d),$
\begin{equation}\label{unf-sob}
\left\|(\la-z)^{-1}f\right\|_{L^{q} (\R^d)}\leq
C|z|^{\frac d2(\frac1p-\frac1q)-1}\|f\|_{L^{p}(\R^d)}.
\end{equation}
\end{corollary}

\begin{remark}
The result of Corollary \ref{thm:unisobine} with $d=2$ is new.
For $d\geq3$,  the estimates for \eqref{unf-sob} were first proved by Bouclet and Mizutani \cite{BM1} and Mizutani \cite{Miz1} under \eqref{p_q} and an extra assumption $\tfrac1p+\tfrac1q=1$. Mizutani--Zhang--Zheng \cite{MZZ} proved
Corollary \ref{thm:unisobine}  for
$$\mu_a=
\begin{cases}
\frac12,&\nu_0\ge\frac12;\\
\tilde \nu_0,&0<\nu_0<\frac12,
\end{cases}
$$
with $\tilde\nu_0=\tfrac{\nu_0^2}{1-2\nu_0^2}$. Our result extends their results to $\mu_a=\min\big\{\tfrac12,\nu_0\big\}$, since $\tilde\nu_0<\nu_0$ for $\nu_0<\tfrac12$.

See Figure \ref{figure_1} for an illustration  of condition \eqref{p_q}:
\begin{itemize}
    \item When $a\geq-\tfrac{(d-2)^2}4+\tfrac14$, or equivalently  $\nu_0\geq\tfrac12$, \eqref{p_q} coincides with \eqref{p_q_0} and corresponds to the interior of the trapezium $ABB'A'$ unioned with the two open line segments ${AA'}$, ${BB'}$  (see Figure \ref{figure_1}), in which case Corollary \ref{thm:unisobine} gives the full range of uniform Sobolev inequalities for $\la$.
    \item When $-\tfrac{(d-2)^2}4\leq a<-\tfrac{(d-2)^2}4+\tfrac14,$  the condition on $(p,q)$ in the paper of Mizutani--Zhang--Zheng \cite{MZZ}  corresponds to the interior of the shaded  hexagon $CDBB'D'C'$, unioned with the two open line segments $BB'$ and $CC'$.     From our work, we extend the region of validity of the uniform Sobolev inequality to the interior of the larger shaded hexagon  $EFBB'F'E'$ unioned with the two open line segments $BB'$ and $EE'$, which corresponds to the condition \eqref{p_q} of Corollary \ref{thm:unisobine}.
 \end{itemize}
\end{remark}

\begin{figure}[htbp]
\begin{center}
\definecolor{my1}{rgb}{0,0.4,1}
\definecolor{my2}{rgb}{0.9,0.5,1}
\begin{tikzpicture}[scale=1]

\draw (0,0) rectangle (6.75,6.75);
\draw[->]  (0,0) -- (0,7.5);
\draw[->]  (0,0) -- (7.5,0);
\filldraw[fill=my1!40](4.3875,1.0875)--(3.75,1.0875)--(3.75,1.5)--(5.25,3)--(5.6625,3)--(5.6625,2.3625); 
\filldraw[fill=my2](4.3875,1.0875)--(3.75,1.0875)--(3.75,0.8)--(4.1,0.8); 
\filldraw[fill=my2](5.6625,2.3625)--(5.6625,3)--(6,3)--(6,2.68); 
\draw (6.75,0) node[below] {$\ 1$};
\draw (0,6.75) node[left] {$1$};
\draw (0,0) node[below, left] {$0$};
\draw (7.5,0) node[above] {$1/p$};
\draw (0,7.5) node[right] {$1/q$};
\draw[dotted] (6.75,0) -- (0,6.75);
\draw[dotted] (6.3,3) -- (7.8,4.5);
\draw (7,3.5) node[right] {$\frac1p-\frac1q=\frac 2d$};
\draw[dotted] (5.25,3) -- (7.8,5.55);
\draw (7.9,5.4) node[above] {$\frac1p-\frac1q=\frac{2}{d+1}$};
\draw[dotted] (3.75,1.5) -- (3.75,3.0) -- (5.25,3.0);
\draw[dotted] (3.75,0.45) -- (0,0.45);
\draw[dotted] (6.3,3) -- (6.3,0);
\filldraw (3.75,0.45) circle (1pt) node[right] {$A$};    
\filldraw (6.3,3)  circle (1pt) node[right] {$\!A'$}; 
\filldraw (3.75,1.5)  circle (1pt) node[left] {$B$};   
\filldraw (5.25,3)  circle (1pt) node[above] {$B'$}; 
\filldraw (4.3875,1.0875)  circle (1pt) node[right] {$C$}; 
\filldraw (5.6625,2.3625) circle (1pt) node[right] {$C'$};    
\filldraw (3.75,1.0875) circle (1pt) node[left] {$D$};    
\filldraw (5.6625,3) circle (1pt) node[above] {$\ \  D'$};    
\filldraw (4.1,0.8)  circle (1pt) node[right] {$E$}; 
\filldraw (6,2.68) circle (1pt) node[right] {$E'$};    
\filldraw (3.75,0.8) circle (1pt) node[left] {$F$};    
\filldraw (6,3) circle (1pt) node[above] {$\ \ \ \ F'$};    
\draw (3.75,0.45) -- (6.3,3); 
\draw[dotted] (3.75,0.45) -- (3.75,1.5) ; 
\draw[dotted] (5.25,3) -- (6.3,3) ; 
\draw (3.75,1.5) -- (5.25,3); 
\draw (4.1,0.8) -- (3.75,0.8); 
\draw (6,2.68) -- (6,3); 
\draw[dotted] (3.75,0.45) -- (3.75,0);
\draw (3.75,0) node[below] {$\frac{d+1}{2d}$};
\draw[dotted] (5.025,1.725) -- (5.025,0) node[below] {$\frac{d+2}{2d}$};
\draw[dotted] (5.6625,2.3625) -- (5.6625,3);
\draw(6.3,0) node[below]  {$\frac{d+3}{2d}$};
\draw (0,3) node[left] {$\frac{d-1}{2d}$};
\draw[dotted] (5.025,1.725) -- (0,1.725) node[left] {$\frac{d-2}{2d}$};
\draw (0,0.45) node[left] {$\frac{d-3}{2d}$};
\draw[dotted] (0,3) -- (3.75,3.0);
\end{tikzpicture}

\end{center}
\caption[An illustration of condition \eqref{p_q} of Corollary \ref{thm:unisobine}. The two trapeziums $CDFE$ and $C'D'F'E'$ are new when $-\tfrac{(d-2)^2}4<a<-\tfrac{(d-2)^2}4+\tfrac14$.]{ An illustration of condition \eqref{p_q} of Corollary \ref{thm:unisobine}. $A',B',C',D',E'$ are dual points of $A,B,C,D,E$, respectively, and
        \begin{center}$
    \begin{array}{ll}
        A=(\frac{d+1}{2d},\frac{d-3}{2d}),&
        B=(\frac{d+1}{2d},\frac{d^2-3d+1}{2d(d+1)}),
        \\
        C=(\frac{d+2(1-\tilde \nu_0)}{2d},\frac{d-2(1+\tilde \nu_0)}{2d}), &
        D=(\frac{d+1}{2d},\frac{d-2(1+\tilde \nu_0)}{2d}),
        \\
        E=(\frac{d+2(1-\nu_0)}{2d},\frac{d-2(1+\nu_0)}{2d}),&
        F=(\frac{d+1}{2d},\frac{d-2(1+\nu_0)}{2d}).
    \end{array}
    $\end{center}
    The two trapeziums $CDFE$ and $C'D'F'E'$ are new when $-\frac{(d-2)^2}4 \le a < -\frac{(d-2)^2}4 + \frac14$.
 }
\label{figure_1}
\end{figure}

\subsection*{Application 5: Multiplier theorem.}
%

It follows from  Sogge \cite[Theorem 0.2.6]{Sogge} that $m\big(\sqrt{-\Delta}\big)$ is bounded in  $L^p(\mathbb R^d)$ provided
\begin{align}\label{equ:mulassum}
    \sum_{0\leq |\alpha|<n} \sup_{\lambda>0} \lambda^{-d}\big\||\lambda|^\alpha D^\alpha (\chi(\cdot/\lambda)m(\cdot))\big\|_{L^2(\mathbb R^d)}^2 <\infty,
\end{align}
for some integer $n>\frac{d}{2}$, whenever $\chi\in C_0^\infty(\mathbb R^d \setminus \{0\})$.
 This, together with $L^p$-boundedness of wave operators yields the following multiplier theorem for $\la$.

\begin{corollary}\label{cor:multhm} Let $m\in L^\infty(\mathbb R^d)$  satisfy the condition \eqref{equ:mulassum}. Then $m\big(\sqrt{\la}\big)$ is bounded in  $L^p(\mathbb R^d) $ for $1<p<\infty$ when $a>0$, and  for $p_0'<p<p_0=\tfrac{d}{\sigma}$ when $-\frac{(d-2)^2}{2}\leq a <0$.

\end{corollary}

\begin{remark}
Noting that the condition
$$\big|\partial_\lambda^jm(\lambda)\big|\lesssim \lambda^{-j},\quad \forall\;j\geq0$$
implies that $m$ satisfies the assumption \eqref{equ:mulassum}, Corollary \ref{cor:multhm} also implies the  Mikhlin-type multiplier theorem of Killip--Miao--Visan--Zhang--Zheng
 \cite{KMVZZ-Sobolev}.

\end{remark}


The rest of the paper is organized as follows: In Section \ref{s:prelim}, we give some preliminaries which will be used in the following sections. In Section \ref{s:const-of-wave-op}, we construct two intertwining operators which turns out to be the stationary wave operators. The proof of the $L^p$-boundedness is given in Section \ref{s:proof-of-lp-bddness}. In Section \ref{s:proof-of-wsp-bddness}, we give the $L^p$-boundedness of two Riesz-type operators, which implies the $W^{s,p}$-boundedness of wave operators. In Appendices \ref{s:proof-of-lem-deriv-est-for-remainder-term} and \ref{s:proof-of-lem-deriv-est-for-RT}, we give the proof for two technical lemmas.

We conclude this introduction by defining some notation which
will be used throughout this paper. If $X, Y$ are nonnegative quantities, we use $X\lesssim Y $ or
$X=O(Y)$ to denote the estimate $X\leq C Y$ for some constant $C>0$.
We use $\lfloor s\rfloor$ to denote the floor of $s$, i.e. the largest integer that does not exceed $s$.
$\langle f(r),g(r) \rangle $ denotes the inner product between $f,g
\in L^2(\mathbb R_{+}; r^{d-1} \dd r)$, i.e.
\begin{align*}
   \langle f(r) , g(r)\rangle =\int_0^\infty f(r) g(r) r^{d-1} \dd r.
\end{align*}
In the course of proving our results, we will need to further introduce more notation, which we  summarize in Table \ref{table-notation} at the end of this paper.

\section{Preliminaries}
\label{s:prelim}
In this section, we first recall some basic properties of gamma, Bessel and hypergeometric functions and Fox $H$-functions. Then, we prove the $L^p$-boundedness of a singular integral operator, and recall a  discrete multiplier theorem for spherical harmonics which is crucial in proving our main result.
\subsection{Gamma function and reciprocal gamma function.} For a complex number $z$ with $\Re z>0$, the gamma function $\Gamma(z)$ is defined by
\begin{equation*}
  \Gamma(z)=\int_0^\infty t^{z-1}e^{-t}\;dt,
\end{equation*}
and analytically continued to $\mathbb C \setminus \{0, -1, -2, \dots \}$.
Let $b, c$ be two positive constants. Then the following asymptotic expression is a generalization of the Stirling formula (see \cite{Tricomi}),
\begin{align}\label{asymptotic expansion for the quotient of gamma function}
 \frac{\Gamma(z+b)}{\Gamma(z+c)} \simeq z^{b-c}\left[1+C_1\frac{1}{z}+ C_2\frac{1}{z^2}+\ldots\right]
\end{align}
as $z \rightarrow \infty$ along any curve joining $z=0$ and $z=\infty$ provided $z \neq-b,-b-1, -b-2, \ldots,$ and $z \neq-c,-c-1, -c-2, \ldots,$ where
\begin{align*}
   C_1&=\frac{(b-c)(b+c-1)}{2 },&
    C_2&=\frac{1}{12}\binom{b-c}{2}\big[3(b+c-1)^2-(b-c+1)\big].
\end{align*}

 The reciprocal Gamma function $\frac{1}{\Gamma(z)}$  is an analytic function near $z= 0$  with the Taylor expansion
\begin{align*}
    \frac{1}{\Gamma (z)}= z +\gamma z^2+ \sum_{j=2}^\infty C_j z^j,
\end{align*}
 where $\gamma$ is Euler's constant, see  \cite[p. 256]{AS}.
\subsection{Polygamma function} As in \cite[p. 258, 260]{AS}, polygamma  functions $\psi^{(m)}(x)$ of order $m$ are defined as the $(m+1)$-th derivative of the logarithm of the gamma function:
\begin{align*}
    \psi^{(0)}(x)&\acoloneq  \psi(x)=\frac{\Gamma'(x)}{\Gamma(x)},\\
    \psi^{(m)}(x)&\acoloneq \frac{d^m \psi(x)}{dx^m},\quad m\in\mathbb{N}.
\end{align*}
\begin{lemma} The polygamma function has the following asymptotic expansion  as $|z|\to \infty$:
\begin{align*}
    \psi^{(m)}(z)& \sim (-1)^{m+1} \sum_{k=0}^\infty \frac{(k+m-1)! B_k}{k! z^{k+m}}, \quad m\geq 1,\\
    \psi^{(0)}& \sim \ln(z)-\sum_{k=1}^\infty \frac{B_k}{k z^k},
\end{align*}
 where $B_k$ are the Bernoulli numbers with $B_1=\tfrac12$.
\end{lemma}

\subsection{The Bessel functions}\label{subsec:Bes}  The Bessel function of the first kind of order $\nu$  \\($ \Re \nu >-\frac{1}{2}$) is defined by (see \cite{Wat})
\begin{align}\label{df first bf}
    J_{\nu}(z)=\big(\frac{1}{2}z\big)^{\nu} \sum_{k=0}^\infty (-1)^k\frac{\big(\frac{1}{4}z^2\big)^k }{k! \Gamma(\nu+k+1)}.
\end{align}
The Bessel function of the second kind (a.k.a. Weber's function) of order $\nu$  is
\begin{align*}
    Y_{\nu}(z)= \frac{J_{\nu}(z) \cos (\nu \pi)-J_{-\nu}(z)}{\sin (\nu \pi)}.
\end{align*}
	When $\nu $ is an integer, the right hand side is given by the limiting value
	\begin{align*}
	    Y_{n}(z)= \frac{1}{\pi} \frac{\partial  J_{\nu}(z)}{\partial \nu} \big|_{\nu =n}+\frac{(-1)^n}{\pi} \frac{\partial  J_{\nu}(z)}{\partial \nu} \big|_{\nu =-n}.
	\end{align*}
The Bessel function of the third kind (a.k.a. Hankel function) of order $\nu$ is
\begin{align*}
    H^{(1)}_{\nu}(z)=J_{\nu}(z)+i Y_{\nu}(z), \quad H^{(2)}_{\nu}=J_{\nu}(z)-i Y_{\nu}(z).
\end{align*}
\begin{lemma}[{Watson \cite[p. 429]{Wat}}]\label{Bessel integral-1}
    The following equality is valid
\begin{align*}
\int_0^\infty \frac{t}{t^2-r^2}J_{\nu} (at) J_{\nu} (bt) \dd t=\left\{\begin{array}{l}
\frac{1}{2} \pi \mathrm{i} J_{\nu}(b r) H_{\nu}^{(1)}(a r), b<a; \\
\frac{1}{2} \pi \mathrm{i} J_{\nu}(a r) H_{\nu}^{(1)}(b r), a<b.
\end{array}\right.
\end{align*}
\end{lemma}
Another formula we need is the indefinite integral for products of two Bessel functions (see \cite[p. 241]{Olver}).
\begin{lemma}\label{Bessel integral-II} Let $\mathscr C_\nu (z)$ and $\mathscr D_\nu (z)$ be any nontrivial linear combination of $J_\nu(z), Y_{\nu}(z), $ and $H_{\nu}^{(1)}(z)$ (not necessarily different). Then
\begin{align}\label{infinite integral}
     \int \mathscr C_\mu (z) \mathscr D_\nu (z) \frac{\dd z}{z}=-z\frac{\mathscr C_{\mu+1}(z)\mathscr D_{\nu}(z)-\mathscr C_{\mu}(z)\mathscr D_{\nu+1}(z)}{\mu^2-\nu^2}+\frac{\mathscr C_{\mu}(z)\mathscr D_{\nu}(z)}{\mu+\nu}.
 \end{align}
 \end{lemma}

From the proof of Proposition 2 in \cite{Dun} , we can derive
 \begin{align*}
 \int_{z}^{\infty} \frac{J_{\mu}(t) Y_{\nu}(t)}{t} \dd t
&=\frac{z Y_{\nu}(z)}{\mu+\nu}\Big(\frac{J_{\mu+1}(z)
    -J_{\nu+1}(z)}{\mu-\nu}\Big)-\frac{z Y_{\nu+1}(z)}{\mu+\nu}\left(\frac{J_{\mu}(z)-J_{\nu}(z)}{\mu-\nu}\right) \notag
    \\&\quad -\frac{J_{\mu}(z) Y_{\nu}(z)}{\mu+\nu},
\end{align*}
and
 \begin{align*}
\int_{0}^{z} \frac{J_{\mu}(t) J_{\nu}(t)}{t} \dd t=-z\frac{J_{\mu+1}(z)J_{\nu}(z)-J_{\mu}(z)J_{\nu+1}(z)}{\mu^2-\nu^2}+\frac{J_{\mu}(z)J_{\nu}(z)}{\mu+\nu}.
\end{align*}
Hence, combining the Wronskians for Bessel equations (see \cite[p. 222]{Olver}),
\begin{align*}
    J_{\nu+1}(z)Y_{\nu}(z)-J_{\nu}(z) Y_{\nu+1}(z)=\frac{2}{\pi z},
\end{align*}
we have
\begin{align}\label{resolvent function}
  &\quad \int_0^s Y_{\mu} (s\lambda) J_{\mu} (\rho \lambda)J_{\nu }(\rho \lambda) \frac{\dd \rho}{\rho}+\int_{s}^\infty J_{\mu} (s\lambda) Y_{\mu} (\rho \lambda)J_{\nu }(\rho \lambda) \frac{\dd \rho}{\rho}\nonumber\\
  &=\frac{2 }{(\nu^2 -\mu^2) \pi}\left[J_{\nu}(s\lambda)-J_{\mu}(s\lambda)\right].
\end{align}
%
%
%

 \subsection{The hypergeometric function ${}_2F_1$} See \cite[p. 556]{AS}. Recall that  Pochhammer's symbol $(z)_{n}$ is
$$
(z)_{n}\coloneq  z(z+1)(z+2) \ldots(z+n-1)=\frac{\Gamma(z+n)}{\Gamma(z)}, n \in \mathbb{N},
$$
with $(z)_{0}=1 .$
The hypergeometric function $_{2} F_{1}(a ; b ; c ; z)$ is
defined by the Gauss series
\[
\begin{aligned}
_{2}F_{1}(a ; b ; c ; z) &=\sum_{n=0}^{\infty} \frac{(a)_{n}(b)_{n}}{(c)_{n} n !} z^{n}=1+\frac{a b}{c} z+\frac{a(a+1) b(b+1)}{c(c+1) 2 !} z^{2}+\ldots \\
&=\frac{\Gamma(c)}{\Gamma(a) \Gamma(b)} \sum_{n=0}^{\infty} \frac{\Gamma(a+n) \Gamma(b+n)}{\Gamma(c+n) n !} z^{n}
\end{aligned}
\]
on the disk $|z|<1,$ and by analytic continuation elsewhere.
%
Near the branch point $z=1$, when $\Re (c-a-b)<0$, we have
\begin{align}\label{Asymptotic near one for hypergeometric function}
 \lim _{z \rightarrow 1^{-}}(1-z)^{a+b-c} {_{2}F_{1}}(a ; b ; c ; z)
=\frac{\Gamma(c) \Gamma(a+b-c)}{\Gamma(a) \Gamma(b)}.
\end{align}

\begin{lemma}[{\cite[p. 74]{Korenev}}]\label{Integral of Bessel function}
The following Weber--Sonin--Schafheitlin integral can be written as
\[
\begin{aligned}
\int_{0}^{\infty} t{J_{\mu}(a t) J_{\nu}(b t)} \dd t=&  \frac{2 b^{\nu} \Gamma\left(\frac{1}{2} \mu+\frac{1}{2} \nu+1\right)}{ a^{\nu+2} \Gamma(\nu+1) \Gamma\left(\frac{1}{2} \mu-\frac{1}{2} \nu\right)} \\
& \times {_{2}F_{1}}\left(\frac{\mu+\nu}{2}+1, \frac{\nu-\mu}{2}+1 ; \nu+1 ; \frac{b^{2}}{a^{2}}\right)
\end{aligned}
\]
for $0<b<a,$ and
\[
\begin{aligned}
\int_{0}^{\infty} t {J_{\mu}(a t) J_{\nu}(b t)} \dd  t=& \frac{2 a^{\mu} \Gamma\left(\frac{\mu+\nu}{2}+1\right)}{ b^{\mu+2} \Gamma\left(\frac{\nu-\mu}{2}\right) \Gamma(\mu+1)} \\
& \times {_{2}F_{1}}\left(\frac{\mu+\nu}{2}+1, \frac{\mu-\nu}{2}+1 ; \mu+1 ; \frac{a^{2}}{b^{2}}\right)
\end{aligned}
\]
for $0<a<b,$ as long as $\mu, \nu$ are such that the integral is convergent.
\end{lemma}

\subsection{Mellin transform}
The Mellin transform of a function is defined by
\[\mathcal M (f(s))(z)=\int_0^{\infty} f(s)s^{z-1}ds,\]
as long as the above integral is valid.

The inverse Mellin transform is defined by

\begin{align*}
    \mathcal M^{-1}(g(z))(s)=\frac{1}{2\pi i}\int_{\mathcal C} s^{-z}g(z) \dd z,
\end{align*}
where $\mathcal C$ is a curve in $\mathbb C$ which are specially chosen.

\subsection{Fox  $H$-functions}

  \begin{definition}[\cite{KilbasSaigo}] For integers $m, n, p, q$ such that $0\leq m\leq q, 0\leq n \leq p$, for $a_i, b_j\in \mathbb C$ and for $\alpha_i, \beta_j\in \mathbb {R}_{+}=(0, \infty )(i=1,2,\dots,p; j=1,2,\dots, q )$, the  Fox $H$-function $H_{p,q}^{m,n}(z)$ is defined via a Mellin-Barnes type integral in the form
  \[H_{p,q}^{m,n}(z)\coloneq H_{p,q}^{m,n}\left[z \middle |\begin{smallmatrix}(a_1,\alpha_1 ), \dots, (a_p, \alpha_p)\\
  (b_1,\beta_1),\dots, (b_q, \beta_q)	
  \end{smallmatrix}\right ]=\frac{1}{2\pi i}\int_{\mathcal C}\mathcal{H}_{p,q}^{m,n}(s)z^{-s}ds,
 \]
 with
\begin{align}
     \mathcal{H}_{p,q}^{m,n}(s)
 &\coloneq
  \mathcal{H}_{p,q}^{m,n}
  \!\!\left[
    \begin{smallmatrix}(a_1,\alpha_1 ), \dots, (a_p, \alpha_p)\notag \\
    (b_1,\beta_1),\dots, (b_q, \beta_q)	
    \end{smallmatrix}  \middle | s
  \right]
  \\&\coloneq
  \frac{\prod_{j=1}^m\Gamma(b_j+\beta_js)\prod_{i=1}^{n}\Gamma(1-a_i-\alpha_i s )}{\prod_{j=n+1}^p\Gamma(a_i+\alpha_is)\prod_{j=m+1}^{q}\Gamma(1-b_j-\beta_j s )}, \label{def: Hmnpq(s)}
\end{align}
  	and $\mathcal C$ is the infinite contour in the complex plane which separate the poles
  	\[b_{jl}=\frac{-b_j-l}{\beta_j} \quad (j=1,\dots,m; \ l=0,1,2,\dots ), \]
  	of the Gamma functions $\Gamma(b_j+\beta_j s)$ to the left of $\mathcal C$ and the poles
  	\[a_{ik}=\frac{1-a_i+k}{\alpha_i}\quad (i=1,\dots, n; \ k=0, 1,2,\dots), \]
  	of the Gamma functions $\Gamma(1-a_i-\alpha_i s)$ to the right of $\mathcal C$. 	
  \end{definition}

  The properties of the  Fox $H$-function $H_{p,q}^{m,n}(z)$ depend on the parameters of $a^{*}$,  $\Lambda$, $\varrho$ and $\delta$ which are given by
  \begin{align}
  	a^{*}&=\sum_{i=1}^{n}\alpha_i-\sum_{i=n+1}^{p}\alpha_i +\sum_{j=1}^{m}\beta_j-\sum_{j=m+1}^{q}\beta_j, \label{parameters a}\\
  	\Lambda &=\sum_{j=1}^{q}\beta_j-\sum_{i=1}^{p}\alpha_ i, \label{parameters tau}\\
  	\varrho &= \sum_{j=1}^q b_j-\sum_{i=1}^{p} a_i+ \frac{p-q}{2}, \label{parameters varrpho} \\
  	\delta &=\prod_{i=1}^p \alpha_ i^{-\alpha_ i}\prod_{j=1}^q \beta_{j}^{\beta_j} \label{parameters delta}.
  \end{align}
From the definition, it is immediate that
\begin{align*}
    \mathcal M(\mathcal{H}_{p,q}^{m,n}(s))(z)={H}_{p,q}^{m,n}(z),  \quad
     \mathcal M^{-1}({H}_{p,q}^{m,n}(z))(s)=\mathcal{H}_{p,q}^{m,n}(s).
\end{align*}

In this paper, we only consider the Fox $H$-functions with fixed indices ${H}_{4,4}^{2,2}(z)$ with $\delta=1$. For
  ${H}_{4,4}^{2,2}(z)$ with $\delta=1$, we have the following power series expansions
\begin{lemma}[Theorem 1.3, \cite{KilbasSaigo}]\label{series expansion 1} Let $\Lambda =0$ and $0<|z|<\delta=1$. Then the Fox $H$-function $H_{4,4}^{2,2}(z)$ has the power series expansion
\begin{align}
    H_{4,4}^{2,2}(z)=\sum_{n=0}^\infty h_{1n}^* z^{\frac{b_1+n}{\beta_1}}+\sum_{n=0}^\infty h_{2n}^* z^{\frac{b_2+n}{\beta_2}},
\end{align}
with the coefficients
    \begin{align*}
        h_{1n}^*&=\lim_{s\to b_{1n} } \left[(s-b_{1n}) \mathcal{H}_{4,4}^{2,2} \right]\\
  &= \frac{(-1)^n}{ n! \beta_1}\frac{\Gamma\big(b_2-[b_1+n]\frac{\beta_2}{\beta_1}\big)\Pi_{j=1}^2\Gamma\big(1-a_j+[b_1+n]
  \frac{\alpha_j}{\beta_1}\big)}{\Pi_{j=3}^4\Gamma\big(a_j-[b_1+n]\frac{\alpha_j}{\beta_1}\big) \Pi_{i=3}^4\Gamma\big(1-b_i-[b_1+n]\frac{\beta_i}{\beta_1}\big)},
    \end{align*}
    and
    \begin{align*}
        h_{2n}^*&=\lim_{s\to b_{2n} } \left[(s-b_{2n}) \mathcal{H}_{4,4}^{2,2} \right]\\
          &= \frac{(-1)^n}{ n! \beta_2}\frac{\Gamma\big(b_1-[b_2+n]\frac{\beta_1}{\beta_2}\big)\Pi_{j=1}^2\Gamma\big(1-a_j
          +[b_2+n]\frac{\alpha_j}{\beta_2}\big)}{\Pi_{j=3}^4\Gamma\big(a_j-[b_2+n]\frac{\alpha_j}{\beta_2}\big)
          \Pi_{i=3}^4\Gamma\big(1-b_i-[b_2+n]\frac{\beta_i}{\beta_2}\big)}.
    \end{align*}
\end{lemma}

\begin{lemma}[Theorem 1.4, \cite{KilbasSaigo}]\label{series expansion 2} Let $\Lambda=0$ and $|z|>\delta=1$. Then the Fox $H$-function $H_{4,4}^{2,2}(z)$ has the powerseries expansion
\begin{align}
    H_{4,4}^{2,2}(z)=\sum_{n=0}^\infty h_{1n} z^{\frac{a_1-1-n}{\alpha_1}}+\sum_{n=0}^\infty h_{2n} z^{\frac{a_2-1-n}{\alpha_2}},
\end{align}
  with the coefficients
    \begin{align*}
        h_{1n}&=\lim_{s\to a_{1n} } \left[-(s-a_{1n}) \mathcal{H}_{4,4}^{2,2} \right]\\
        &=\frac{(-1)^n}{ n! \alpha_1}\frac{\Pi_{i=1}^2\Gamma\big(b_i+[1-a_1+n]\frac{\beta_i}{\alpha_1}\big)\Gamma\big(1-a_2-[1-a_1+n]
        \frac{\alpha_2}{\alpha_1}\big)}{\Pi_{j=3}^4\Gamma\big(a_j-[1-a_1+n]\frac{\alpha_j}{\alpha_1}\big)\Pi_{i=3}^4
        \Gamma\big(1-b_i-[1-a_1+n]\frac{\beta_i}{\alpha_1}\big)},
    \end{align*}
    and
    \begin{align*}
       h_{2n}&=\lim_{s\to a_{2n} } \left[-(s-a_{2n}) \mathcal{H}_{4,4}^{2,2} \right]\\
       &= \frac{(-1)^n}{ n! \alpha_2}\frac{\Pi_{i=1}^2\Gamma\big(b_i+[1-a_2+n]\frac{\beta_1}{\alpha_2}\big)\Gamma\big(1-a_1-[1-a_2+n]\frac{\alpha_1}
       {\alpha_2}\big)}{\Pi_{j=3}^4\Gamma\big(a_j-[1-a_2+n]\frac{\alpha_j}{\alpha_2}\big)\Pi_{i=3}^4
       \Gamma\big(1-b_i-[1-a_2+n]\frac{\beta_i}{\alpha_2}\big)}.
    \end{align*}
\end{lemma}

%
	\subsection{Singular integral operators on the half-line}
We first easily get Young's inequality for convolutions on the half-line $(0,
\infty)$ with Haar measure $\dd r /r $ as follows:
\begin{lemma}[Young's inequality]  \label{Log-type Young's inequality}
For $1\leq p\leq \infty$, we have
\begin{align*}
	\left\|\int_{0}^\infty  f(\rho) g(r/\rho) \frac{\dd{\rho}}{\rho}\right\| _{L^p(\frac{\dd r}{r})}
	&\leq \|f(r)\| _{L^p(\frac{\dd r}{r})}\|g(r)\| _{L^1(\frac{\dd r}{r})}.
\end{align*}		
	\end{lemma}
	As a corollary of Lemma \ref{Log-type Young's inequality}, we have the following result.
\begin{lemma}\label{boundedness for CZ operator 2} The operator
\begin{align*}
    Tf(r)=  \int_{\frac{r}{2}}^{r}\left|\ln\!\Big({1-\Big(\frac{s}{r}\Big)^2}\Big)\right |f(s)\frac{\dd s}{s}+\int_{r}^{2r}\left|\ln\!\Big({1-\Big(\frac{r}{s}\Big)^2}\Big)\right|f(s)\frac{\dd s}{s}
\end{align*}
 is bounded in  $L^p\big(\mathbb R_{+}; \frac{\dd r}{r}\big) $ for $1\leq p\leq \infty$.	
\end{lemma}
\begin{proof} It is easy to notice that
\begin{align*}
	\int_0^\infty \left|\ln(1-s^2)\right|\chi_{\{1/2<s<1\}} \frac{\dd s}{s} <\infty.
\end{align*}	
Thus, by Lemma \ref{Log-type Young's inequality}, we deduce that $T$ is $L^p\big(\mathbb R_{+}; \frac{\dd r}{r}\big)$ bounded.
\end{proof}

The following lemma is a variant of Schur's test,  which will be used in the proof of Lemma \ref{boundedness for CZ operator 1} below.

\begin{lemma}[Schur's test with weights, \cite{KVZ12}]\label{lem:schur} Suppose $(X,d\mu)$ and $(Y,d\nu)$ are measure spaces, and $w(x,y)$ is a positive
measurable function defined on $X\times Y$. Let $K(x,y):~X\times Y\to\C$ satisfy
\begin{align*}
\sup_{x\in X}\int_Y
w(x,y)^\frac1p\big|K(x,y)\big|d\nu(y)=C_0<\infty,\\\label{est:scha2}
\sup_{y\in Y}\int_X
w(x,y)^{-\frac1{p'}}\big|K(x,y)\big|d\mu(x)=C_1<\infty,
\end{align*}
for some $1<p<\infty$. Then the operator $T$ defined by
$$Tf(x)=\int_Y K(x,y)f(y)d\nu(y)$$
is a bounded operator from $L^p(Y,d\nu)$ to  $L^p(X,d\mu)$. In particular,
\begin{equation*}
\big\|Tf\big\|_{ L^p(X,d\mu)}\leq
C_0^\frac1{p'}C_1^\frac1p\|f\|_{L^p(Y,d\nu)}.
\end{equation*}
\end{lemma}

Using the $L^p$-boundedness of the Hilbert transform and Schur's test in Lemma \ref{lem:schur}, we can prove the following lemma.

\begin{lemma}\label{boundedness for CZ operator 1} Let $\mu_0=\frac{d}{2}-1$ and $\nu_0=\sqrt{\mu_0^2+a}$. For $r\in \mathbb R_{+}$, define
\begin{align*}
	Tf(r)&=\int_\frac{r}{2}^r  \left(\frac{s}{r}\right)^{\mu_0+\frac{d}{2}-\frac{d}{p}+1}\frac{ f(s)}{1-(\frac{s}{r})^2} \frac{\dd s}{s}
	&-\int_r^{2r}  \left(\frac{r}{s}\right)^{\frac{d}{p}-\frac{d}{2}+\nu_0+1}\frac{f(s)}{1-(\frac{r}{s})^2}  \frac{\dd s}{s},
\end{align*}
 then $T$ is bounded in  $L^p(\mathbb R_{+}, \frac{\dd r}{r})$ for $1<p<\infty$.
\end{lemma}

\begin{proof} Note that $T$ is bounded in  $L^p(\mathbb R_{+}, \frac{\dd r}{r})$ if and only if $$\|r^{-\frac{1}{p}}Tf(r)\|_{L^p(\mathbb R_{+}; \dd r) }\lesssim \|r^{-\frac{1}{p}}f\|_{L^p(\mathbb R_{+}; \dd r)}.$$
Hence we define the following modified operator $\widetilde T$ with
\begin{align*}
	\widetilde Tf(r)&=\int_\frac{r}{2}^r  \left(\frac{s}{r}\right)^{\mu_0+\frac{d}{2}-\frac{d-1}{p}+1}\frac{ f(s)}{1-(\frac{s}{r})^2} \frac{\dd s}{s}
	&-\int_r^{2r}  \left(\frac{r}{s}\right)^{\frac{d-1}{p}-\frac{d}{2}+\nu_0+1}\frac{f(s)}{1-(\frac{r}{s})^2}  \frac{\dd s}{s},
\end{align*}
and it suffices to prove that $\widetilde T$ is bounded in $L^p(\mathbb R_{+}; \dd r)$.

 Let $f\in L^p(\mathbb R_{+}, \dd r)$, which we identify with its extension by zero to a function on $\mathbb R$.  For $r>0$, write $\widetilde Tf(r)=\int_{\mathbb R} \widetilde T(r, s) f(s) \dd s$ with
\begin{align*}
   \widetilde T(r,s)=\left(\frac{s}{r}\right)^{\mu_0+\frac{d}{2}-\frac{d-1}{p}}\frac{ r}{r^2-s^2} \chi_{\{1/2<s/r<1\}} +\left(\frac{r}{s}\right)^{\frac{d-1}{p}-\frac{d}{2}+\nu_0+1}\frac{ s}{r^2-s^2}\chi_{\{1<s/r<2\}}.
\end{align*}
We also identify $\widetilde Tf$ with its extension by zero.
It is easy to check that $\widetilde T(r,s)$ satisfies
\begin{align*}
	|\widetilde T(r, s)|&< C \frac{1}{|r-s|},\\
	|\partial_r \widetilde T(r, s)+\partial_s  \widetilde T(r, s)| &\leq  C \frac{1}{|r-s|^2}.
	\end{align*}
Combining this with \CZ{}  singular integral theory (See \cite[Chapter 5]{Duo}), in order to prove $T$ is bounded in  $L^p(\mathbb R; \dd r )$ for $1<p<\infty$, it suffices to prove that $T$ is bounded in  $L^2(\mathbb R; \dd r)$.
We write     $Tf(r)$ as follows,
\begin{align}\label{equ:tfrdef}
   \widetilde Tf(r)=\int_{-\infty}^\infty (F(r,s)-1)f(s)\frac{ \dd s}{r-s}+\int_{-\infty}^\infty f(s)\frac{\dd s}{r- s},
\end{align}
with \begin{align*}
   F(r,s)&=\left(\frac{s}{r}\right)^{d-\frac{d-1}{p}-1}\frac{r}{r+s} \chi_{\{r/2<s<r\}}-\frac{1}{2}
   +\left(\frac{r}{s}\right)^{\frac{d-1}{p}-\frac{d}{2}+\nu_0+1}\frac{s}{r+s} \chi_{\{r<s<2r\}}-\frac{1}{2} +1\\
   &\eqqcolon \big[F^{+}(r,s)-\frac{1}{2}\big]+\big[F^{-}(r,s)-\frac{1}{2}\big]+1.
\end{align*}
We denote
\begin{align*}
   \widetilde T^{+}f(r)&=\int_{-\infty}^\infty \big[F^{+}(r,s)-\frac{1}{2} \big]f(s)\frac{ \dd s}{r-s},\\
   \widetilde T^{-}f(r)&=\int_{-\infty}^\infty \big[F^{-}(r,s)-\frac{1}{2} \big]f(s)\frac{ \dd s}{r-s}.
\end{align*}
Using the boundedness of the Hilbert transform, we know that the convolution operator
\begin{align*}
   \int_{-\infty}^\infty f(s)\frac{\dd s}{r- s}
\end{align*}
is bounded in  $L^2(\mathbb R ; \dd r)$. Hence, it remains to show that the operators $\widetilde T^\pm$ are $L^2$-bounded.  To do this, we rewrite the kernels as
\begin{align*}
 F^{+}(r,s)-\frac{1}{2}=&\left(\frac{s}{r}\right)^{d-\frac{d-1}{p}-1}\left(\frac{r}{r+s}-\frac{1}{2}\right)\chi_{\{r/2<s<r\}}+\frac{1}{2}
 \left[\left(\frac{s}{r}\right)^{d-\frac{d-1}{p}-1}-1\right]\chi_{\{r/2<s<r\}}\\
    =&\frac{1}{2}\left(\frac{s}{r}\right)^{d-\frac{d-1}{p}-1}\frac{r-s}{2(r+s)}\chi_{\{r/2<s<r\}}+\frac{1}{2}
    \left[\left(\frac{s}{r}\right)^{d-\frac{d-1}{p}-1}-1\right]\chi_{\{r<s<2r\}},
    \end{align*}
    and
    \begin{align*}
  &F^{-}(r,s)-\frac{1}{2}\\
  = &\left(\frac{r}{s}\right)^{\frac{d-1}{p}-\frac{d}{2}+\nu_0+1}\left(\frac{s}{r+s}-\frac{1}{2}\right)\chi_{\{r<s<2r\}}
  +\frac{1}{2}\left[\left(\frac{r}{s}\right)^{\frac{d-1}{p}-\frac{d}{2}+\nu_0+1}-1\right]\chi_{\{r<s<2r\}}\\
    =&\frac{1}{2}\left(\frac{r}{s}\right)^{\frac{d-1}{p}-\frac{d}{2}+\nu_0+1}\frac{s-r}{2(r+s)}\chi_{\{r<s<2r\}
    }+\frac{1}{2}\left[\left(\frac{r}{s}\right)^{\frac{d-1s}{p}-\frac{d}{2}+\nu_0+1}-1\right]\chi_{\{r<s<2r\}}.
    \end{align*}
Let us now define
\begin{align*}\nonumber
\widetilde T^{+}f(r)&=\int_{-\infty}^\infty \big[F^{+}(r,s)-\frac{1}{2} \big]f(s)\frac{ \dd s}{r-s}\\\nonumber
&=\frac{1}{2} \int_{r/2}^ r\left(\frac{s}{r}\right)^{d-\frac{d-1}{p}-1}\frac{f(s)}{2(r+s)} \dd s +\frac{1}{2} \int_{r/2}^r \left[\left(\frac{s}{r}\right)^{d-\frac{d-1}{p}-1}-1\right]f(s)\frac{ \dd s}{r-s}\\
&\triangleq \widetilde T_1^{+} f(r)+ \widetilde T_2^{+}f(r).
  \end{align*}
Taking $w(r,s)=(r/s)^{1/2}$ in Lemma \ref{lem:schur}, we can check that
\begin{align*}
   \sup_{r>0} \int_0^\infty \frac{1}{r+s}\frac{r^{1/2}}{s^{1/2}}ds \lesssim 1,\quad\text{and}\quad
    \sup_{s>0}\int_0^\infty \frac{1}{r+s}\frac{s^{1/2}}{r^{1/2}}dr \lesssim 1.
\end{align*}
Hence, $\widetilde T_1^{+}$ is $L^2(\mathbb R ; \dd r)$ bounded.

For $\widetilde T_2^{+}$, a direct computation gives
\begin{align*}
    &\sup_{r}\int_{-\infty}^\infty \frac{1-\left(\frac{s}{r}\right)^{d-\frac{d-1}{p}-1}}{1-\frac{s}{r}}\chi_{\{r/2<s<r\}}\frac{\dd s}{s} \lesssim 1,\\
      &\sup_{s}\int_{-\infty}^\infty \frac{1-\left(\frac{s}{r}\right)^{d-\frac{d-1}{p}-1}}{1-\frac{s}{r}}\chi_{\{r/2<s<r\}}\frac{\dd r}{r} \lesssim 1.
     \end{align*}
This yields the $L^2(\mathbb R; \dd r)$ boundedness of $\widetilde T^+_{2}$. By the same argument, we get  the $L^2$-boundedness of $\widetilde T^{-}$.

In conclusion, we derive that $T$ is bounded in  $L^p(\mathbb R^+; \frac{\dd r}{r})$ for $1<p<\infty$.
\end{proof}

\subsection{A discrete multiplier theorem for spherical harmonics }
We denote $Y_{k,l}(\omega)$ to be the $l$-th spherical harmonic of degree $k$. These are homogeneous harmonic polynomials of degree $k$ restricted to the surface of unit sphere $\mathbb S^{d-1}$. It is known that  $Y_{k,l}(\omega)$ are  the eigenfunctions of the angular part of $\Delta$ with eigenvalues $-k(k+d-2)$.
$\{Y_{k,l}\}_{k,l}$ forms a complete orthonormal system in $L^2(\mathbb S^{d-1})$ in the sense that for any given function $F(\omega) \in L^2(\mathbb S^{d-1})$, we expand $F(\omega)$ in spherical harmonics
\[ F(\omega)= \sum_{k=0}^\infty\sum_{l=1}^{d_k} F_{k,l} Y_{k,l}(\omega), \]
with coefficients $F_{k,l}=\int_{\mathbb S^{d-1}}F(\omega)Y_{k,l}(\omega) \dd \sigma({\omega})$.
We write $\proj_k F$ for the projection of $F$ to the spherical harmonics of degree $k$,
\begin{align}\label{equ:projkfdef}
\proj_k F(\omega)&\coloneq  \sum_{l=1}^{d_k} F_{k,l} Y_{k,l}(\omega).
\end{align}
By making   use  of zonal harmonics $Z^{k}$, $F(\omega)$ can also  be written as
\begin{align}\label{zonal decomposition}
    F(\omega)=\sum_{k=0}^\infty \proj_k F(\omega)=\sum_{k=0}^\infty \int_{\mathbb S^{d-1}} Z^{k}(\theta \cdot \omega)  F(\theta) \dd \sigma(\theta) .
\end{align}
See  \cite[Chapter IV]{Stein Weiss} for more information about spherical harmonics and zonal harmonics.

Define the forward finite difference operator acting on sequences $(\myC_k)_{k\ge 0}$ by
\[ \fd \myC_k \coloneq  \myC_{k+1} - \myC_k\]
and inductively define the $N$th order forward finite difference operator $$\fd^N \myC_k \coloneq  \fd(\fd^{N-1} \myC_k),\quad N\in\mathbb{N}.$$ For the  decomposition \eqref{zonal decomposition}, we have the following multiplier theorem on $\mathbb S^{d-1}$, which is crucial  in proving our main result.

\begin{theorem}[Bonami and Clerc, \cite{BC}]
\label{multiplier theorem for spherical harmonic decomposition}
    Let $(\myC_k)_{k\ge 0}$ be a sequence of complex numbers satisfying the conditions
\begin{enumerate}[%
label=\textup{(BC\roman*)}, ref=BC\roman*
]%
    \item \label{uniform bound}
$	\sup_k \left|\myC_{k}\right| \lesssim 1$,
    \item \label{uniform difference bound}
	$\sup_{j} 2^{j(N-1)} \sum\limits_{k=2^ j}^{2 ^{j+1}}\left| \fd ^{N} \myC_{k}\right| \lesssim 1 $ for $N= \lfloor \frac{d-1}{2}\rfloor+1 $.
\end{enumerate}
Then for $1<p<\infty$, there holds
\begin{align}
	\left\|\sum_{k=0}^\infty \myC_{k} \proj_{k} F\right\|_{L^{p}(\mathbb S^{d-1})} \lesssim \left\|\sum_{k=0}^\infty  \proj_{k} F\right\|_{L^{p}(\mathbb S^{d-1})}.
\end{align}

\end{theorem}


\begin{remark}\label{finite differnce in between} Notice that if \eqref{uniform bound} and \eqref{uniform difference bound} are valid, by interpolation, we have for all $1\leq m \leq N$,
\begin{align*}
 \sup_{j} 2^{j(m-1)} \sum\limits_{k=2^ j}^{2 ^{j+1}}\left| \fd ^{m} \myC_{k}\right| \lesssim 1.
\end{align*}

\end{remark}
The following lemma gives a sufficient condition for \eqref{uniform difference bound} if the sequence can be viewed as a  smooth function of $k$.

\begin{lemma} Suppose that a bounded sequence $\myC_k=\myC(k)$ for some smooth function $\myC:[1,\infty)\to\mathbb R$. If one has
\begin{align*}
    \left| \frac{\dd^N \myC}{\dd k^N}\right| \lesssim_N \frac{1}{k^N},\quad N= \lfloor \tfrac{d-1}{2}\rfloor +1,
\end{align*}
    then, there holds
    \begin{align*}
    \sup_{k}  k^N | \fd{}^N \myC_k| \lesssim_N 1.
    \end{align*}
This estimate implies  \eqref{uniform difference bound} for the sequence $\{\myC_k\}_{k\ge0}$.\end{lemma}
\begin{proof} Using the Fundamental Theorem of Calculus and Remark \ref{finite differnce in between}, we have
\begin{align*}
   | \fd{} \myC(k)|&=\left|\int_{k}^{k+1} \myC'(y)dy \right|\lesssim \frac{1}{k}.
\end{align*}
 By induction, we have for $N= \lfloor \frac{d-1}{2}\rfloor+1$,
 \begin{align*}
    | \fd{}^N \myC(k)| &=\left|\int_{k}^{k+1}\int_{y_1}^{y_1+1} \dots \int_{y_{N-1}}^{y_{N-1}+1} \myC^{(N)}(z) dz dy_{N-1} \dots dy_1\right| \lesssim k^{-N}.
 \end{align*}
Hence,
\begin{align*}
     \sup_{j}  2^{j(N-1)}\sum_{k=2^j}^{2^{j+1}} | \fd{}^N \myC_k|& <      \sup_{j}  2^{j(N-1)} 2^j \sup_{2^j\leq k\leq 2^{j+1}}  | \fd{}^N \myC_k|\\
    & \lesssim \sup_{k} k^N | \fd{}^N \myC_k|<\infty.\qedhere
\end{align*}\end{proof}

Later, we need to determine whether the product of sequences satisfies the conditions \eqref{uniform bound} and \eqref{uniform difference bound} or not. This is achieved by the following lemma.\begin{lemma}\label{Difference estimates for product}
 Suppose that $(\myG_k)_{k\geq 0}$ and $(\myF_k)_{k\geq 0}$ both satisfy  \eqref{uniform bound}, $(\myG_k)_{k\geq 0}$ satisfies \eqref{uniform difference bound}, and $(\myF_k)_{k\geq 0}$ satisfies the stronger condition that
  \begin{align} \label{uniform difference bound strong}
 	\sup_{k} k^m \left|\fd ^{m} \myF_{k}\right| \lesssim 1,
 \end{align}
   for $1\leq m \leq  N=\lfloor \frac{d-1}{2}\rfloor +1$,  then
 $\{\myC_k\}_{k\geq 0}\coloneq \{\myG_k\myF_k\}_{k\geq 0}$ also satisfies the assumptions \eqref{uniform bound} and \eqref{uniform difference bound} in Theorem \ref{multiplier theorem for spherical harmonic decomposition}.	
\end{lemma}
\begin{proof}

For  $ N=\lfloor \frac{d-1}{2}\rfloor +1$, using the Leibniz rule for finite difference, we have
\begin{align*}
	\fd ^N( \myF_k\myG_k)=\sum_{m=0}^N \binom{N}{m} \fd ^{N-m} \myF_k\fd^{m}\myG_{k+N-m}.
	\end{align*}
Hence, by Remark \ref{finite differnce in between}, we get
\begin{align*}
	&\sup_j 2^{j(N-1)}\sum_{k=2^j}^{2^{j+1}}|\fd ^N( \myF_k \myG_k)|\\
	=&\sup_j 2^{j(N-1)}\sum_{k=2^j}^{2^{j+1}}\left|\sum_{m=0}^N \binom{N}{m} \fd ^{N-m} \myF_k\fd^{m}\myG_{k+N-m}\right|\\
	=&\sup_j \sum_{k=2^j}^{2^{j+1}}\left|\sum_{m=0}^N \binom{N}{m} 2^{j(N-m)}\fd ^{N-m} \myF_k \ 2^{j(m-1)}\fd^{m}\myG_{k+N-m}\right|\\
	\leq & C_N \max_{1\leq m \leq N}\sup_{j} 2^{jm}\left|\fd ^{m} \myF_{2^j}\right| \sup_j \sum_{k=2^j+N-m}^{2^{j+1}+N-m}\Big|\sum_{m=0}^N  \ 2^{j(m-1)}\fd^{m}\myG_{k}\Big| \lesssim_N 1.\qedhere
\end{align*}		
	
\end{proof}
\section{Construction of stationary wave operators}
\label{s:const-of-wave-op}

In this section, we  construct two operators $W$ and $W^*$ which ``intertwine'' $-\Delta$ and $\mathcal L_a$ in $L^2(\mathbb R^d)$. Later, we prove that these two intertwining operators are the wave operators $W_{-}$ and $W_{-}^*$ as in \eqref{stationary wave operators} and \eqref{stationary adjoint wave operators} defined by the stationary method which is introduced by Kato and  Kuroda \cite{KatoKuroda} for $d\geq 2$.
\subsection{Spherical harmonics decomposition of $L^2(\mathbb R^d)$.}
Using the spherical harmonics expansion,  for any function $f(x) \in L^2(\mathbb R^d)$, we can decompose it as
\begin{align} \label{decomposition}
    f(r, \omega)= \sum_{k=0}^\infty \sum_{l=1}^{d_k} f_{k,l}(r ) Y_{k,l}(\omega)\eqqcolon  \sum_{k=0}^\infty \proj_k f(r, \omega)
\end{align}
where $r=|x|, \omega=\frac{x}{|x|}$, and each $f_{k,l}(r) \in L^2(\mathbb R_{+}; r^{d-1} \dd r)$.
For this decomposition, we have the following identity
$$
\int_{\mathbb  R^{d}}|f(x)|^{2} d x=\sum_{k=0}^\infty \sum_{l=1}^{d_k}\int_{0}^{\infty}\left|f_{k,l}(r)\right|^{2} r^{d-1}\dd{r}.
$$
This gives an isometric isomorphism between $L^2(\mathbb R^d)$ and the countable direct sum of  $\sh_{k,l}=\{f(r)Y_{k,l}(\omega)|\; f \in L^{2}(\mathbb  R_{+}; r^{ d-1}\dd{r}) \}\cong L^{2}(\mathbb  R_{+}; r^{ d-1}\dd{r})$:
\begin{align*}
    \mathscr I: L^2(\mathbb R^d)  &\to    \bigoplus_{k,l}\sh_{k,l}\\
              f &\mapsto  \{f_{k,l}\}_{\substack{0\le k<\infty
\\ 1\le l \le d_k }}.
    \end{align*}
%

%
\subsection{Bessel transform and Hankel transform}
\label{ss:bessel-hankel}
Each $\sh_{k,l}$ reduces $-\Delta$ to
\begin{align*}
	\myH_{0;k,l}\ &\acoloneq  -\Delta|_{\sh_{k,l}}=-\frac{\dd^2}{\dd r^2}-\frac{d-1}{r}\frac{\dd}{\dd r}+\frac{1}{r^2}\left(k(k+d-2)\right)\nonumber \\
	&= -\frac{\dd^2}{\dd r^2}-\frac{d-1}{r}\frac{\dd}{\dd r}+\frac{1}{r^2}\left[\left(\frac{d-2}{2}+k\right)^2-\left(\frac{d-2}{2}\right)^2\right].
\end{align*} If we denote $\mu_k=\frac{d-2}{2}+k$, then $\myH_{0;k,l}$ has the generalized eigenfunction,
\begin{align*}	
e_{k}(r, \lambda)=(\lambda r)^{-\frac{d-2}{2}} J_{\mu_k}\left(\lambda r\right) \quad \text{with eigenvalue } \lambda^2>0
\end{align*}
in the sense
\begin{align}\label{equ:lapeigen}
	\myH_{0;k,l} e_{k}(r, \lambda)=\lambda^2 e_{k}(r, \lambda),
\end{align}
where the $J_{\mu_k}$ are Bessel functions of the first kind  of order $\mu_k$ as in Subsection \ref{subsec:Bes}.
Now we define the \emph{Bessel transform} in $L^2(\mathbb R_+; r^{d-1} \dd{r})$ by
\begin{align}\label{Bessel transform}
	\bessel_{\mu_k} f(\lambda)\coloneq \langle f(r), e_k(r,\lambda) \rangle=
	\int_0^\infty f(r) (\lambda r)^{-\frac{d-2}{2}}J_{\mu_k}(r\lambda) r^{d-1}\dd {r},
\end{align}
here $\langle \cdot ,\cdot  \rangle $ denotes the inner product in $L^2(\mathbb R_+; r^{d-1} \dd{r})$.
\begin{lemma}[Properties of the Bessel Transform]\label{prp bt} There holds
\begin{align}
\int_0^\infty |\bessel_{\mu_k}f(\lambda) |^2\lambda^{d-1}\dd{\lambda}&=\int_0^\infty |f(r) |^2 r^{d-1}\dd{r};\label{plancherel for B}\\
	\bessel_{\mu_k}\bessel_{\mu_k}&=I\label{Identity for B};\\
	\bessel_{\mu_k} (\myH_{0;k,l} f)(\lambda)&=\lambda ^2 \bessel_{\mu_k} f(\lambda); \label{diagonalize B}
\end{align}
\end{lemma}
\begin{proof} Using the integral formula
\begin{align*}
    \int_0^\infty \lambda J_{\nu}(s\lambda)J_{\nu}(r\lambda) \dd \lambda=\frac{\delta(r-s)}{s},
\end{align*}
we can easily obtain   \eqref{plancherel for B} and \eqref{Identity for B} by the definition of Bessel transform.

Since
\begin{align*}
    \bessel_{\mu_k} (\myH_{0;k,l} f)(\lambda)&=\langle \myH_{0;k,l} f(r), e_k(r,\lambda) \rangle=\langle f(r), \myH_{0;k,l} e_k(r,\lambda) \rangle\\
    &=\langle f(r), \lambda^2 e_k(r,\lambda) \rangle=\lambda ^2 \bessel_{\mu_k} f(\lambda),
\end{align*}
we get \eqref{diagonalize B}.
\end{proof}

 Restricting  the operator  $\la$ to $\sh_{k,l}$, we have
 \[{\myH_{k,l}}\coloneq  \la |_{\sh_{k,l}}=-\frac{\dd^2}{\dd r^2}-\frac{1}{r}\frac{\dd}{\dd r}+\frac{1}{r^2}\left[\left(\frac{d-2}{2}+k\right)^2-\left(\frac{d-2}{2}\right)^2+a\right], \] with the generalized eigenfunction
\begin{align*}
	\widetilde e_{k}(r, \lambda)= (\lambda r)^{-\frac{d-2}{2}}J_{ \nu_k}\left(\lambda r\right)  \text {, with eigenvalue } \lambda^2>0,
\end{align*}
such that
\begin{align}\label{equ:laeigen}
   \myH_{ k,l}\widetilde e_{k}(r, \lambda)=\lambda^2  \widetilde e_{k}(r, \lambda).\end{align}
Here $\nu_k=\sqrt{(\frac{d-2}{2}+k)^2+a}$.
We define the Hankel transform on $L^2(\mathbb  R_{+}; r^{d-1} \dd{r})$ by
\begin{align}\label{Hankel transform}
	\hankel_{\nu_k}f (\lambda): =\langle   f(r), \widetilde e_k(r, \lambda) \rangle =\int_0^\infty f(\lambda)(r\lambda)^{-\frac{d-2}{2}} J_{\nu_k} (r\lambda) r^{d-1}\dd {r}.
\end{align}
By the same argument as in the proof of Lemma \ref{prp bt}, we have the following properties of the Hankel transform.
\begin{lemma}[Properties of the  Hankel transform]\label{prp ht}One has
 \begin{align}
\int_0^\infty |\hankel_{\nu_k}f(\lambda) |^2\lambda^{d-1}\dd{\lambda}&=\int_0^\infty |f(r) |^2 r^{d-1}\dd{r};\label{plancherel for H}\\
	\hankel_{\nu_k}\hankel_{\nu_k}&=I\label{Identity for H};\\
\hankel_{\nu_k} ( \myH_{k,l} f)(\lambda)&=\lambda ^2 \hankel_{\nu_k} f(\lambda).
\end{align}
\end{lemma}
\subsection{Construction of intertwining operators }	
Now we define the operators $W_{k}=\hankel_{\nu_k}\bessel_{\mu_k}$ and $W^*_{k}=\bessel_{\mu_k}\hankel_{\nu_k}$ on each $\sh_{k,l}$. Both operators are unitary on $\sh_{k,l}$ by \eqref{plancherel for B} and \eqref{plancherel for H}, in the sense that
\begin{align}\label{unitary of W k}
	\int_0^\infty |W_{k}f(r) |^2r^{d-1}\dd{r}&=\int_0^\infty |f(r) |^2 r^{d-1}\dd{r},
\end{align}
and
\begin{align}\label{unitary of W k *}
	\int_0^\infty |W_{k}^*f(r) |^2r^{d-1}\dd{r}&=\int_0^\infty |f(r) |^2 r^{d-1}\dd{r}.
\end{align}
Next, we  utilize the spherical harmonics decomposition to define two operators $W,W^*:L^2(\mathbb R^d)\to L^2(\mathbb R^d)$ by
\begin{align}
	Wf(r,\omega )&\coloneq \sum_{k=0}^{\infty}\sum_{l=1}^{d_k} W_k f_{k,l}(r ) Y_{k,l}(\omega)=\sum_{k=0}^{\infty}W_k \proj_k f(r,\omega),\label{definition of wave operators}\\
W^*f(r,\omega )&\coloneq \sum_{k=0}^{\infty}\sum_{l=1}^{d_k} W^*_k f_{k,l}(r ) Y_{k,l}(\omega)=\sum_{k=0}^{\infty} W_k^* \proj_k f(r,\omega)\label{adjoint wave operators},
\end{align}
where $\proj_k $ is as in \eqref{decomposition}.

\begin{lemma}[$L^2$ boundedness of wave operator and intertwining property]\label{lem:l2bound}
The operators $W$ and $W^*$ given by \eqref{definition of wave operators} and \eqref{adjoint wave operators} are unitary  on $L^2(\mathbb R^d)$ and have the following intertwining property
   \begin{align}
\mathcal L_a W=W(-\Delta), \quad W^*\mathcal L_a =(-\Delta) W^*.
\end{align}
\end{lemma}
\begin{proof} By \eqref{unitary of W k}, we have
\begin{align*}
	\|Wf(r,\omega)\|^2_{L^2(\mathbb  R^d)} &=\sum_{k=0 }^\infty \sum_{l=1} ^{d_k}\int_{0}^\infty |W_k f_{k,l}(r)|^2 r^{d-1}\dd{r}\\
	&=\sum_{k=0 }^\infty \sum_{l=1} ^{d_k}\int_{0}^\infty  | f_{k,l}(r)|^2 r^{d-1}\dd{r}\\
	&=\|f(x)\|^2_{L^{2}(\mathbb  R^d)}.
\end{align*}
This implies that $W$ is a unitary operator on $L^2(\mathbb R^d)$.

To prove  $\mathcal L_a W=W(-\Delta),$ note that for any $f(x)\in L^2(\mathbb  R^d)$, we have
 \begin{align*}
 	W(-\Delta)f(r,\omega)&=W(-\Delta) (\sum_{k=0}^\infty \sum_{l=1}^{d_k} f_{k,l}(r ) Y_{k,l}(\omega))\\
 	&=W (\sum_{k=0}^\infty\sum_{l=1}^{d_k} \myH_{0;k,l} f_{k,l}(r ) Y_{k,l}(\omega))\\
 	&=\sum_{k=0}^\infty\sum_{l=1}^{d_k}   \hankel_{\nu_k}\left[ \bessel_{\mu_k}\myH_{0;k,l} f_{k,l}(\lambda )\right](r ) Y_{k,l}(\omega)\\
 	&=\sum_{k=0}^\infty\sum_{l=1}^{d_k}   \hankel_{\nu_k}\left[ \lambda^2 \bessel_{\mu_k}  f_{k,l}(\lambda)\right](r ) Y_{k,l}(\omega).
 \end{align*}
 On the other hand,
 \begin{align*}
 	\la W f(r,\omega)&=\la W (\sum_{k=0}^\infty\sum_{l=1}^{d_k}  f_{k,l}(r ) Y_{k,l}(\omega))\\
 	&=\la (\sum_{k=0}^\infty \sum_{l=1}^{d_k} \hankel_{\nu_k}\bessel_{\mu_k}f_{k,l}(r ) Y_{k,l}(\omega))\\
 	&=\sum_{k=0}^\infty\sum_{l=1}^{d_k} \myH_{k,l} \hankel_{\nu_k} \bessel_{\mu_k}  f_{k,l}(r ) Y_{k,l}(\omega)\\
 	&=\sum_{k=0}^\infty \sum_{l=1}^{d_k}\hankel_{\nu_k}\hankel_{\nu_k} \myH_{k,l} \hankel_{\nu_k} \bessel_{\mu_k} f_{k,l}(r ) Y_{k,l}(\omega)\\
 	&=\sum_{k=0}^\infty\sum_{l=1}^{d_k}   \hankel_{\nu_k}\left[ \lambda^2 \bessel_{\mu_k}  f_{k,l}(\lambda)\right](r ) Y_{k,l}(\omega).
 \end{align*}
Hence, $W(-\Delta)=\la W$ on $L^2(\mathbb  R^d)$. By the same argument, we obtain the desired result for the operator $W^\ast$.

 Therefore, we complete the proof of Lemma \ref{lem:l2bound}.
\end{proof}

 \subsection{The equivalence to the wave operators} \label{sect: equivalence}
 In this subsection, we prove that the intertwining operators $W$ and $W^*$ defined by \eqref{definition of wave operators} and \eqref{adjoint wave operators} are in fact the stationary  wave operators $W_{-}$ and $W_{-}^*$ given by \eqref{stationary wave operators} and \eqref{stationary adjoint wave operators}, i.e.
 \begin{equation}\label{equ:wequw-w}
   W=W_{-},\quad W^*=W_{-}^*.
 \end{equation}
%

 We use $\dd E_{\sqrt{-\Delta}}(\lambda)$ and $\dd E_{\sqrt{\la}} (\lambda)$ to denote the spectral measure for $\sqrt{-\Delta}$ and $\sqrt{\la}$. By previous eigenfunction expansions (see also \cite{KatoKuroda}), we have
 \begin{align*}
 	\frac{\dd E_{\sqrt{-\Delta}}(\lambda)}{\dd \lambda} \bigg |_{\sh_{k,l}}f(r) =\langle e_k(r, \lambda ),  f(r)\rangle e_k(r, \lambda ) \lambda^{d-1},
 \end{align*}
 and
 \begin{align*}
 	\frac{\dd E_{\sqrt{\la}}(\lambda)}{\dd \lambda} \bigg |_{\sh_{k,l}}f(r) =\langle \widetilde e_k(r, \lambda ),  f(r)\rangle \widetilde e_k(r, \lambda ) \lambda^{d-1}.
 \end{align*}
The spectral measures for $-\Delta$ and $\la$ are given by
 \begin{align*}
{\dd E_{{-\Delta}}(\lambda^2)}={\dd E_{\sqrt{-\Delta}}(\lambda)}, \text{ and} \quad
\dd E_{\la}(\lambda^2) ={\dd E_{\sqrt{\la}}(\lambda)}.
 \end{align*}
Hence,
\begin{align*}
    \frac{\dd E_{-\Delta}}{\dd \lambda}(\lambda^2)=\frac{1}{2\lambda}   \frac{\dd E_{\sqrt{-\Delta}}(\lambda)}{\dd \lambda},\ \text{and}\quad  \frac{\dd E_{\la}}{\dd \lambda}(\lambda^2)=\frac{1}{2\lambda}   \frac{\dd E_{\sqrt{\la}}(\lambda)}{\dd \lambda}.
\end{align*}
Thus, the wave operator $W_-$ is also given by
 \begin{align}
 	W_{-}
 	&=\int_{0}^\infty \frac{ \dd E_{\la} }{\dd \lambda}(\lambda^2)(\la-\lambda^2) R_0(\lambda^2-i0) \dd (\lambda^2)\nonumber\\
 	&=\int_{0}^\infty \frac{ \dd E_{\sqrt\la}(\lambda) }{\dd \lambda}(\la-\lambda^2) R_0(\lambda^2-i0) \dd \lambda. \label{stationary W minus}
 \end{align}
By the functional calculus, on each $\sh_{k,l}$, we have
\begin{align}\label{resolvent representation}
    R_0(\lambda^2-i\varepsilon ) |_{\sh_{k,l}} f(r)=\int_0^\infty \frac{\langle e_k(r, \mu),  f(r)\rangle e_k(r, \mu) \mu^{d-1}}{\mu^2-(\lambda^2-i\varepsilon )} \dd \mu.
\end{align}
 To prove $W=W_{-}$ and $W^*=W_{-}^*$, it suffices to prove that on each spherical harmonic subspace $\sh_{k,l}$,
\begin{align}
	W_{-}|_{\sh_{k,l}}= W_{k},\   \text{and},\  W_{-}^*|_{\sh_{k,l}}= W_{k}^*.
\end{align}
 Using the following limit
\begin{align}
    \lim_{\varepsilon \to 0}\int_0^\infty \frac{\varepsilon f(\mu)}{(\mu^2-\lambda^2)^2+\varepsilon^2} \dd \mu= \frac{\pi}{2} \frac{f(\lambda)}{\lambda},
\end{align}
and by \eqref{stationary W minus},  \eqref{resolvent representation}, we have
\begin{align*}
	&\quad W_{-}|_{\sh_{k,l}} f(r)\\
	&=\int_0^\infty \widetilde e_k(r, \lambda) \langle \widetilde e_k(r, \lambda), (\myH_{k,l}-\lambda^2)  R_0(\lambda^2-i0) f(r)	\rangle \lambda^{d-1}\dd \lambda\\
   &=\lim_{\varepsilon \to 0}\int_0^\infty \widetilde e_k(r, \lambda) \left\langle \widetilde e_k(r, \lambda),  (\myH_{k,l}-\lambda^2)\int_0^\infty \tfrac{\langle e_k(r, \mu),  f(r)\rangle e_k(r, \mu) \mu^{d-1}}{\mu^2-(\lambda^2-i\varepsilon )} \dd \mu \right	\rangle\lambda^{d-1} \dd \lambda\\
    &=\lim_{\varepsilon \to 0}\int_0^\infty \widetilde e_k(r, \lambda) \left\langle \widetilde  e_k(r, \lambda), \int_0^\infty \tfrac{\langle e_k(r, \mu),  f(r)\rangle (\mu^2-\lambda^2+V(r))e_k(r, \mu) \mu^{d-1}}{\mu^2-(\lambda^2-i\varepsilon )} \dd \mu \right	\rangle\lambda^{d-1} \dd \lambda\\
    &=\int_0^\infty \widetilde e_k(r, \lambda) \left\langle \widetilde  e_k(r, \lambda), \int_0^\infty \tfrac{\langle e_k(r, \mu),  f(r)\rangle (\mu^2-\lambda^2+V(r))e_k(r, \mu) \mu^{d-1}}{\mu^2-\lambda^2} \dd \mu \right	\rangle\lambda^{d-1} \dd \lambda\\
    &\quad -\frac{i \pi }{2}\int_0^\infty \widetilde e_k(r, \lambda) \left\langle \widetilde  e_k(r, \lambda),  {\langle e(r, \lambda ),  f(r)\rangle V(r)e_k(r, \lambda) \lambda^{d-2}} \right	\rangle\lambda^{d-1} \dd \lambda\\
    &\eqqcolon \myone +\mytwo -\mythree,
    \end{align*}
where
 \begin{align*}
    \myone &=\int_0^\infty \widetilde e_k(r, \lambda) \left\langle \widetilde  e_k(r, \lambda), \int_0^\infty {\langle e_k(r, \mu),  f(r)\rangle e_k(r, \mu) \mu^{d-1}} \dd \mu \right	\rangle\lambda^{d-1} \dd \lambda,\\
   \mytwo &={\int_0^\infty \widetilde e_k(r, \lambda) \left\langle \widetilde  e_k(r, \lambda), \int_0^\infty \frac{\langle e_k(r, \mu),  f(r)\rangle V(r)e_k(r, \mu) \mu^{d-1}}{\mu^2-\lambda^2} \dd \mu \right	\rangle\lambda^{d-1} \dd \lambda},\\
    \mythree & =\frac{i \pi }{2} \int_0^\infty \widetilde e_k(r, \lambda) \left\langle \widetilde  e_k(r, \lambda),  {\langle e(r, \lambda ),  f(r)\rangle V(r)e_k(r, \lambda) \lambda^{d-2}} \right	\rangle\lambda^{d-1} \dd \lambda.
      \end{align*}
    Using the properties \eqref{Identity for B} and \eqref{Identity for H}, we have  \begin{align}
       \myone = f(r).
   \end{align}
To estimate $\mytwo $, by Lemma \ref{Bessel integral-1}, we obtain
 \begin{align*}
    \mytwo &=r^{-\frac{d-2}{2}}\int_0^\infty \int_0^\infty \int_0^\infty \lambda \rho J_{\nu_k} (r\lambda)J_{\nu_k}(\rho\lambda)
    \\&\quad\times{\int_0^\infty \frac{\mu}{\mu^2-\lambda^2} J_{\mu_k }(s\mu)J_{\mu_k }(\rho \mu)\dd \mu}V(\rho){\dd \rho} \dd \lambda f(s)\lambda s^{d/2} \dd s\\
    &=r^{-\frac{d-2}{2}} \frac{i\pi }{2}\int_0^\infty \int_0^\infty \lambda \rho J_{\nu_k} (r\lambda)\int_0^\infty  \rho J_{\nu_k}(\rho\lambda)\Big
    \{J_{\mu_k}(\rho  \lambda) H_{\mu_k}^{(1)}(s\lambda){{\mathbf{1}}_{\{\rho<s\}}}\\
    &\quad +J_{\mu_k}(s  \lambda) H_{\mu_k}^{(1)}(\rho\lambda){{\mathbf{1}}_{\{\rho>s\}}}\Big \}V(\rho){\dd \rho} \dd \lambda f(s) s^{d/2} \dd s\\
    &=r^{-\frac{d-2}{2}}\int_0^\infty \int_0^\infty \lambda J_{\nu_k} (r\lambda)\mytwo ^{\text{in}}(s, \lambda)
    f(s) s^{d/2} \dd \lambda \dd s,
\end{align*}
where
\begin{align*}
    \mytwo ^{\text{in}}(s,\lambda)=\frac{i \pi }{2}  \int_0^\infty \!\!\rho V(\rho) J_{\nu_k}(\rho\lambda)\big \{J_{\mu_k}(\rho  \lambda) H_{\mu_k}^{(1)}(s\lambda){\mathbf{1}}_{\{\rho<s\}}+J_{\mu_k}(s  \lambda) H_{\mu_k}^{(1)}(\rho\lambda){\mathbf{1}}_{\{\rho>s\}}\dd \rho\}.
\end{align*}
Using the equality \eqref{resolvent function}, and the relationship  $H_{\mu}^{(1)}(z)=J_{\mu}(z)+i Y_{\mu}(z)$, we derive
\begin{align*}
   \mytwo ^{\text{in}}(s,\lambda) =&\frac{i \pi  a}{2} J_{\mu_k}(s\lambda) \int_0^\infty J_{\nu_k}(\rho \lambda)J_{\mu_k}(\rho \lambda)\frac{\dd \rho }{\rho}\\
   &-\frac{a \pi}{2} \left[\int_0^s Y_{\mu_k} (s\lambda) J_{\mu_k} (\rho \lambda)J_{\nu_k }(\rho \lambda) \frac{\dd \rho}{\rho}+\int_{s}^\infty J_{\mu_k} (s\lambda) Y_{\mu_k} (\rho \lambda)J_{\nu_k }(\rho \lambda) \frac{\dd \rho}{\rho}\right]\\
  =& \frac{i \pi  a}{2} J_{\mu_k}(s\lambda) \int_0^\infty J_{\nu_k}(\rho \lambda)J_{\mu_k}(\rho \lambda)\frac{\dd \rho }{\rho}+J_{\mu_k}(s\lambda)-J_{\nu_k}(s\lambda).
 \end{align*}
 Putting $\mytwo ^{\text{in}}(s,\lambda)$ into $\mytwo $ gives
 \begin{align*}
     \mytwo &=\frac{i \pi  a}{2} r^{-\frac{d-2}{2}} \int_0^\infty \int_0^\infty \lambda J_{\nu_k} (r\lambda) J_{\mu_k}(s\lambda) \int_0^\infty J_{\nu_k}(\rho \lambda)J_{\mu_k}(\rho \lambda)\frac{\dd \rho }{\rho}
    f(s) s^{d/2} \dd \lambda \dd s\\
    &\quad+r^{-\frac{d-2}{2}} \int_0^\infty \int_0^\infty \lambda J_{\nu_k} (r\lambda) J_{\mu_k}(s\lambda) f(s)s^{d/2}\lambda \dd \lambda \dd s\\
    &\quad-r^{-\frac{d-2}{2}}  \int_0^\infty \int_0^\infty \lambda J_{\nu_k} (r\lambda) J_{\nu_k}(s\lambda) f(s)s^{d/2}\lambda \dd \lambda \dd s\\
   & =\mythree +W_k f(r)-f(r).
 \end{align*}
Summing $\myone, \mytwo $ and $\mythree $, we have
\begin{align*}
    W_{-}|_{\mathscr H_{k,l}}= W_k.
\end{align*}
Hence the operator $W$ defined by \eqref{definition of wave operators} equals to $W_{-}$. A similar calculation shows that $W^*=W_{-}^*$.

\section{Proof of the $L^p$-boundedness}\label{s:proof-of-lp-bddness}
In this section, we prove Theorem \ref{thm:main}. Since
\begin{align*}
    \|W_{+}f\|_{L^p(\mathbb R^d)}=\|\overline{W_{+}f}\|_{L^p(\mathbb R^d)}=\|{W_{-}\overline{f}}\|_{L^p(\mathbb R^d)},
\end{align*}
and
\begin{align*}
    \|W_{+}^*f\|_{L^p(\mathbb R^d)}=\|\overline{W_{+}^*f}\|_{L^p(\mathbb R^d)}=\|{W_{-}^*\overline{f}}\|_{L^p(\mathbb R^d)},
\end{align*}
it suffices to show the $L^p$-boundedness  for $W_{-}$ and $W_{-}^*$.
By \eqref{equ:wequw-w}, it is equivalent to prove the $L^p$-boundedness for $W$ and $W^*$. We only give the proof for $W$, since the proof for $W^*$ is exactly the same.

By definition \eqref{definition of wave operators}, we have
\begin{align*}
	W f(r, \omega )&=\sum_{k=0}^\infty \int_0^\infty \int_0^\infty (\lambda s)^{-\frac{d-2}{2}} J_{\mu_k}(s\lambda) s^{d-1}\proj_k f(s, \omega) \dd {s}(\lambda r)^{-\frac{d-2}{2}}  J_{\nu_k}(\lambda r) \lambda^{d-1}\dd {\lambda}\\
	&=\sum_{k=0}^\infty\int_0^\infty  r^{1-\frac{d}{2}}{\int_0^\infty s^{\frac{d}{2}} \lambda J_{\mu_k}(s\lambda) J_{\nu_k}(r \lambda) \dd {\lambda} }\;\proj_k f(s, \omega)\dd {s}\\
	&=\sum_{k=0}^\infty\int_0^\infty K_{k}(r,  s)\proj_k f(s, \omega) \frac{\dd s}{s},
  \end{align*}
with
\begin{align*}
	K_{k}(r,  s)= \frac{s^{\frac{d}{2}+1}}{r^{\frac{d}{2}-1}}\int_0^\infty  \lambda J_{\mu_k}(s\lambda) J_{\nu_k}(r \lambda) \dd {\lambda}.
\end{align*}
To prove
$$ 	\|Wf\|_{L^p(\mathbb  R^d)} \lesssim  \|f\|_{L^p(\mathbb  R^d)},$$
it is equivalent to show
\begin{align}\label{equ:mainred2}
	\big\|r^{\frac{d}{p}}Wf(r,\omega)\big\|_{L^p( \frac{\dd r}{r}\cdot\dd \omega)}\leq 	\big\|r^{\frac{d}{p}}f(r,\omega)\big\|_{L^p(\frac{\dd r}{r} \cdot \dd \omega)}.
\end{align}
We define the modified kernels
\[\widetilde K_{k}(r,s)= \left(\frac{s}{r}\right)^{-\frac{d}{p}} K_{k}(r,s),\]
and the modified wave operator
\begin{align}\label{equ:wmoddef}
    \widetilde W f(r, \omega)=\sum_{k=0}^\infty\int_0^\infty \widetilde K_{k}(r,s)\proj_k f(s, \omega) \frac{\dd s}{s}.
\end{align}
Then, it suffices to prove that
   \begin{equation}\label{equ:redtildw}
     \big\|\widetilde Wf(r, \omega)\big\|_{L^p(\frac{\dd r}{r} \cdot \dd \omega)}\lesssim \|f(r, \omega)\|_{L^p(\frac{\dd r}{r} \cdot \dd \omega)}.
   \end{equation}

In order to give the explicit formula for the kernel  functions $\widetilde K_{k}(r,s)$, we introduce $a_k$ and $b_k$, and give some  estimates for of $a_k$ and $b_k$ and its derivatives. Recall $\mu_k=\frac{d-2}{2}+k$ and $\nu_k=\sqrt{(\frac{d-2}{2}+k)^2+a}$.  We denote
\begin{align}\label{def of a_k and b_k}
    a_k=\frac{\mu_k+\nu_k}{2},\quad b_k=\frac{\mu_k-\nu_k}{2},
    \end{align}and  then we have the following estimates.
\begin{lemma}\label{estimates for $a_k$ and $b_k$}
For $2 \leq N \leq \big  \lfloor \frac{d-1}{2} \big  \rfloor +1$,
\begin{align}
 |a_k| \lesssim  k, \quad \left|\frac{\dd a_k}{\dd k}\right| \lesssim  1, \quad \left|\frac{\dd^N a_k}{\dd k^N}\right|  \lesssim  \frac{1}{k^{N}}, \label{estimates for $a_k$}
\end{align}and for $0 \leq N \leq \big  \lfloor \frac{d-1}{2} \big  \rfloor +1$,\begin{align} \left|\frac{\dd^N b_k}{\dd k^N}\right|  \lesssim  \frac{1}{k^{N+1}}.\label{estimates for $b_k$}
\end{align}

\end{lemma}
\begin{proof}

It is easy to see that $\mu_k$ and $\nu_k$ satisfy $\mu_k, \nu_k \lesssim k$. Also, $|\mu_k-\nu_k |= |\frac{-a}{\mu_k+\nu_k}|\lesssim \frac{1}{k}$, and  $\frac{d\nu_k}{\dd k}=\frac{\mu_k}{\nu_k}$, thus \eqref{estimates for $a_k$} and \eqref{estimates for $b_k$} are valid. \end{proof}

\subsection*{Explicit formula for the kernel functions}  By Lemma \ref{Integral of Bessel function},
 $\widetilde K_{k}(r,  s)$ can be written in terms of  ${}_2F_1$ as follows:
 \begin{align}
	 & \widetilde K_{k}(r,  s)= \frac{s^{\frac{d}{2}+1-\frac{d}{p}}}{r^{\frac{d}{2}-1-\frac{d}{p}}}\int_0^\infty  \lambda J_{\mu_k}(s\lambda) J_{\nu_k}(r \lambda) \dd {\lambda} \nonumber\\\nonumber
&=\begin{cases}
	 2\left(\frac{s}{r}\right)^{\frac{d}{2}-\frac{d}{p}+1+\mu_k} \frac{ \Gamma(a_{k}+1)}{ \Gamma(a_{k}+b_k+1)\Gamma(-b_{k})}\  {{}_{2}F_{1}}\left(a_{k}+1, b_{k}+1; a_k+b_k+1; \left(\frac{s}{r}\right)^2\right),
	 \\
	 \hfill 0<s<r;\\
	2{\left(\frac{r}{s}\right)^{\frac{d}{p}-\frac{d}{2}+1+\nu_k}}\frac{ \Gamma(a_{k}+1)}{ \Gamma(a_k-b_k+1)\Gamma(b_k)} \ {_{2}F_{1}}\left(a_{k}+1, 1-b_k; a_k-b_k+1; \left(\frac{r}{s}\right)^2\right),
    \\
     \hfill 0<r<s;
\end{cases}\\
&=\begin{cases} \label{power series representation}
	2\left(\frac{s}{r}\right)^{\frac{d}{2}-\frac{d}{p}+1+\mu_k} \frac{\sin(-\pi b_{k})}{\pi}  \sum \limits_{n=0}^\infty \frac{\Gamma(a_{k}+n+1)\Gamma(b_{k}+n+1)}{\Gamma(a_{k}+b_{k}+n+1) \Gamma(n+1)}\left(\frac{s}{r}\right)^{2n},& 0<s<r;\\
	2\left(\frac{r}{s}\right)^{\frac{d}{p}-\frac{d}{2}+1+\nu_k} \frac{\sin(\pi b_{k})}{\pi}   \sum \limits_{n=0}^\infty \frac{\Gamma(a_{k}+n+1)\Gamma(1-b_k+n)}{\Gamma(a_k-b_k+n+1) \Gamma(n+1)}\left(\frac{r}{s}\right)^{2n},& 0<r<s.
\end{cases}
\end{align}
For ease of notation, we define for $n\geq 0$,
 \begin{align*}
 A_{k,n}^{+}&=\frac{\Gamma(a_{k}+n+1)\Gamma(b_{k}+n+1)}{\Gamma(a_{k}+b_{k}+n+1) \Gamma(n+1)},\\
 A_{k,n}^{-}&= \frac{\Gamma(a_{k}+n+1)\Gamma(1-b_k+n)}{\Gamma(a_k-b_k+n+1) \Gamma(n+1)}.
 \end{align*}	
 Then the kernels $\widetilde K_{k}(r,  s)$  can  be  rewritten as
\begin{align}\label{power series formula}
\widetilde K_{k}(r,  s)=\begin{cases}
	2\left(\frac{s}{r}\right)^{\frac{d}{2}-\frac{d}{p}+1+\mu_k} \frac{\sin(-\pi b_{k})}{\pi}  \sum\limits_{n=0}^\infty A_{k,n}^{+}\left(\frac{s}{r}\right)^{2n},& 0<s<r;\\
	2\left(\frac{r}{s}\right)^{\frac{d}{p}-\frac{d}{2}+1+\nu_k} \frac{\sin(\pi b_{k})}{\pi}   \sum\limits_{n=0}^\infty A_{k,n}^{-}\left(\frac{r}{s}\right)^{2n},& 0<r<s.
\end{cases}
\end{align}

Using the definition of ${}_2F_1$ and the asymptotic behavior \eqref{Asymptotic near one for hypergeometric function}, we have the following lemma.
\begin{lemma}\label{asymptotic behavior of kernels}
The kernel functions $\widetilde K_{k}(r,s)$ have the following asymptotic behavior:
\begin{align*}
		\widetilde K_{k}(r,  s)\sim
		\begin{cases}
			 \frac{2 \Gamma(\frac{\mu_k+\nu_k}{2}+1)}{ \Gamma(\mu_k+1)\Gamma(\frac{\nu_k-\mu_k}{2})}\left(\frac{s}{r}\right)^{\frac{d}{2}-\frac{d}{p}+1 +\mu_k}, &  \frac{s}{r} \to 0;\\
			\frac{2}{\pi} \sin \pi (\frac{\nu_k-\mu_k}{2}) \left(\frac{s}{r}\right)^{\frac{d}{2}-\frac{d}{p}+1 +\mu_k} (1-\frac{s^2}{r^2})^{-1},   & \frac{s}{r} \to {1^{-}};\\
				\frac{2}{\pi} \sin \pi (\frac{\mu_k-\nu_k}{2})\left(\frac{r}{s}\right)^{\frac{d}{p}-\frac{d}{2}+1 +\nu_k} (1-\frac{r^2}{s^2})^{-1}, &  \frac{s}{r} \to 1^{+};\\
			 \frac{2 \Gamma(\frac{\mu_k+\nu_k}{2}+1)}{ \Gamma(\nu_k+1)\Gamma(\frac{\mu_k-\nu_k}{2})}\left(\frac{r}{s}\right)^{\frac{d}{p}-\frac{d}{2}+1 +\nu_k}, &  \frac{s}{r} \to \infty.
		\end{cases}
\end{align*}
\end{lemma}

The asymptotic behavior of $\widetilde K_{k}(r,s)$ suggests that we should split $\widetilde W$ into the three parts as follows:
\begin{align*}
	 \widetilde W f(r, \omega)&=\int_{0}^{\infty} \sum_{k=0}^\infty r^{\frac{d}{p}} K_k(r,s) \proj_{k}f(s, \omega) s^{-\frac{d}{p}}\frac{\dd s}{s}\\
	 &=\int_{0}^{\frac{r}{2}} \sum_{k=0}^\infty  \widetilde K_k(r,s) \proj_{k}f(s, \omega) \frac{\dd s}{s}+\int_{\frac{r}{2}}^{2r} \sum_{k=0}^\infty  \widetilde K_k(r,s) \proj_{k}f(s, \omega) \frac{\dd s}{s}\\
	 &\quad +\int_{2r}^{\infty} \sum_{k=0}^\infty  \widetilde K_k(r,s) \proj_{k}f(s, \omega) \frac{\dd s}{s}\\
	 &\acoloneq \widetilde W_1 f(r, \omega)+\widetilde W_2 f(r, \omega)+\widetilde W_3 f(r, \omega).
\end{align*}
Then, showing \eqref{equ:redtildw} is reduced to proving that
\begin{align}
	\| \widetilde W_i f(r, \omega)\|_{L^p(\frac{\dd r}{r} \cdot \dd \omega)} \lesssim \| f(r, \omega)\|_{L^p(\frac{\dd r}{r} \cdot \dd \omega)}, \ \text{for}\  i= 1, 2, 3.
\end{align}
We denote $\widetilde K_{k,1}(r,s)$, $ \widetilde K_{k,2}(r,s)$ and $\widetilde K_{k,3}(r,s)$ by
\begin{align*}
    \widetilde K_{k,1}(r,s)&= \widetilde K_{k}(r,s) \chi_{\{0<s<r/2\}},\\
      \widetilde K_{k,2}(r,s)&= \widetilde K_{k}(r,s) \chi_{\{r/2<s<2r\}},\\
         \widetilde K_{k,3}(r,s)&= \widetilde K_{k}(r,s) \chi_{\{2r<s\}}.
\end{align*}

\subsection*{{Case 1: $0<\frac{s}{r}<2$}}
By \eqref{asymptotic expansion for the quotient of gamma function} it is easy to verify that $A_{k,n}^{+}$ are uniformly bounded in $n$ and $k$.
 Thus, by the formula \eqref{power series formula}, for any fixed $r, s$, \begin{align}\label{Case-1 bd}
	|\widetilde K_{k,1}(r,s)|\lesssim  \left|\frac{\sin\pi b_k}{\pi}
	\right| \left(\frac{s}{r}\right)^{\frac{d}{2}+\mu_k+1 -\frac{d}{p}}
	\lesssim  \left(\frac{s}{r}\right)^{\frac{d}{2}+\mu_0+1 -\frac{d}{p}}.
\end{align}
For the finite difference of order $N=\lfloor \frac{d-1}{2}\rfloor +1$, we have
\begin{align}\label{Case-1 fd}
	\sup _{j} 2^{j(N-1)} \sum_{k=2^ j}^{2 ^{j+1}}\left|\fd ^{N}\widetilde K_{k,1}(r,s)\right|
	 &\lesssim 2^N \sum_{k=0}^\infty k^{N-1} | \widetilde K_{k,1}(r,s)| \nonumber\\
	 &\lesssim 2^N\sum_{k=0}^\infty k^{N-1}\left(\frac{s}{r}\right)^{\frac{d}{2}+\mu_k+1 -\frac{d}{p}}\nonumber\\
	 &\lesssim 2^N\left(\frac{s}{r}\right)^{d-\frac{d}{p}}.
\end{align}
This inequality together with
\eqref{Case-1 bd} shows  that $\widetilde K_{k,1}(r,s)$ satisfies \eqref{uniform bound} and \eqref{uniform difference bound} in  Theorem \ref{multiplier theorem for spherical harmonic decomposition}. Hence, for $1<p<\infty$, we have
\begin{align*}
	\left \|\sum_{k=0}^\infty  \widetilde K_{k,1}(r,s) \proj_{k }f(s, \omega) \right \|_{L^{p}(\dd \omega)}  \lesssim   \left(\frac{s}{r}^{}\right)^{d-\frac{d}{p}}\left \|\sum_{k=0}^\infty   \proj_{k}f(s, \omega) \right \|_{L^{p}(\dd \omega)}.
\end{align*}
Furthermore, notice $r ^{d-\frac{d}{p}}\chi_{\{0<r<1/2\}} \in L^1(\frac{\dd r}{r})$ for all $1<p<\infty$, hence
\begin{align}\label{W 1}
	\|\widetilde W_1 f(r, \omega )\|_{ L^p (\frac{\dd r}{r} \cdot \dd \omega)}
	&\lesssim  \left\|\int_0^{r/2}  \left(\frac{s}{r}^{}\right)^{d-\frac{d}{p}} \left \|\sum_{k=0}^\infty   \proj_{k }f(s, \omega) \right \|_{L^{p}(\dd \omega)} \frac{\dd s}{s} \right \|_{L^p (\frac{\dd r}{r})}\nonumber\\
	&\lesssim \left \|\sum_{k=0}^\infty   \proj_{k  }f(r, \omega) \right \|_{L^p(\frac{\dd r}{r} \cdot \dd \omega)} =\| f(r,\omega)\|_{L^{p}(\frac{\dd r}{r} \cdot \dd \omega)}.
	\end{align}
Thus, $\widetilde {W}_1$ is bounded in  $ L^p (\frac{\dd r}{r} \cdot \dd \omega)$.
\subsection*{Case 2: $\frac{s}{r}>2$.} Using \eqref{power series representation}, and the fact that $A_{k,n}^{-}$ are uniformly bounded in $k, n$, we have
for fixed $r,s $,
\begin{align}
		|\widetilde K_{k,3}( r, s)|\lesssim \left(\frac{r}{s}\right)^{\frac{d}{p}+1+\nu_{k}-\frac{d}{2}} \lesssim \left(\frac{r}{s}\right)^{\frac{d}{p}+1+\nu_{0}-\frac{d}{2}}.\label{case-3-bd}
\end{align}
For the finite difference of order  $N= \lfloor \frac{d-1}{2}\rfloor+1 $, we have the following bound:
\begin{align}\sup _{j} 2^{j(N-1)} \sum_{k=2^ j}^{2 ^{j+1}}\left|\fd^{N} \widetilde K_{k,3}(r,s)\right|
	 &\lesssim  2^N \sum_{k=0}^\infty k^{N-1} | \widetilde K_{k,3}(r,s)| \nonumber\\
	 &\lesssim  2^N\sum_{k=0}^\infty k^{N-1}\left(\frac{r}{s}\right)^{\frac{d}{p}+1+\nu_{k}-\frac{d}{2}} \nonumber\\
	 &\lesssim  2^N\left(\frac{r}{s}\right)^{\frac{d}{p}+1+\nu_{0}-\frac{d}{2}}.  	\label{case-3-fd}
\end{align}
From \eqref{case-3-bd} and \eqref{case-3-fd}, we know that
$\widetilde K_{k,3}(r,s)$ satisfy the condition \eqref{uniform bound} and \eqref{uniform difference bound} in  Theorem \ref{multiplier theorem for spherical harmonic decomposition}, which implies the following bound is valid:
\begin{align*}
	\left \|\sum_{k=0}^\infty  \widetilde K_{k,3}(r,s) \proj_{k }f(s, \omega) \right \|_{L^{p}(\dd \omega)}  \lesssim   \left(\frac{r}{s}\right)^{\frac{d}{p}+1+\nu_{0}-\frac{d}{2}}\left \|\sum_{k=0}^\infty   \proj_{k }f(s, \omega) \right \|_{L^{p}(\dd \omega)}.
\end{align*}
 To get the boundedness for the radial part, we divide the estimates into two cases, that is, $a> 0$, and $-(\frac{d-2}{2} )^2 \leq a<0$. If $a > 0$, then $\nu_0\geq \frac{d}{2}-1$, thus $\frac{d}{p}+1+\nu_{0}-\frac{d}{2}>0$ for $1<p<\infty$.
If $-(\frac{d-2}{2})^2 \leq a<0$, then $\frac{d}{p}+1+\nu_{0}-\frac{d}{2}>0$ as long as $p<\frac{d}{\sigma}=\frac{d}{\frac{d-2}{2}-\sqrt{(\frac{d-2}{2})^2+a}}$.

Therefore,  if $\max \big\{0, \frac{\sigma}{d}\big\}<\frac{1}{p}<\min\big\{1, \frac{d-\sigma}{d}\big\}$, it follows that $r^{\frac{d}{p}+1+\nu_{0}-\frac{d}{2}} \chi_{\{0<r<1/2\}}\in L^1(\frac{\dd r}{r})$.
 In both cases, we have
\begin{align}\label{W 3}
	\|\widetilde W_3 f(r, \omega )\|_{ L^p (\frac{\dd r}{r}\cdot \dd \omega)}
	&\lesssim  \bigg\|\int_{2r}^{\infty}  \left(\frac{r}{s}\right)^{\frac{d}{p}+1+\nu_{0}-\frac{d}{2}} \bigg \|\sum_{k=0}^\infty   \proj_{k}f(s, \omega) \bigg \|_{L^{p}(\dd \omega)} \frac{\dd s}{s} \bigg \|_{L^p (\frac{\dd r}{r})}\nonumber\\
	&\lesssim \bigg \|\sum_{k=0}^\infty   \proj_{k }f(r, \omega) \bigg \|_{L^p (\frac{\dd r}{r}\cdot \dd \omega)}=\| f(r,\omega)\|_{L^p (\frac{\dd r}{r}\cdot \dd \omega)}.
\end{align}
Hence, ${\widetilde W}_{3}$ is bounded in  $ L^p (\frac{\dd r}{r} \cdot \dd \omega)$.

For the rest part of this section, we remain to show that ${\widetilde W}_{2}$ is bounded in  $ L^p (\frac{\dd r}{r} \cdot \dd \omega)$.

\subsection*{Case 3: $\frac{1}{2}<\frac{s}{r}<2$}
Now we deal with the term $\widetilde W_2 f(r, \omega )$. From Lemma  \ref{asymptotic behavior of kernels}, we know that  the kernels have a singularity of type $(s-r)^{-1}$ at $s=r$, and we cannot get a uniform bound in $r, s$  for $\widetilde K_k(r,s)$. Instead, we can use the theory of \CZ{} operators to prove the  $L^p$-boundedness of the radial part. To achieve this goal, for each kernel $\widetilde K_k(r,s)$, we define an approximate kernel $\widetilde K_k^{\text{ap}}(r,s)$ with the property that the index $k$ and variables $r, s $ are separated, and that they  also have the same singularity as $\widetilde K_k(r,s)$. For the approximating operator, we first estimate the radial part.  After that, we use the multiplier theorem for spherical harmonics to get the boundedness for the angular part. For the error term, we will estimate the angular part first, then the radial part using  Young's inequality.

     In the remainder of this paper, we use the shorthand notation $\chi_{+}=\chi_{\{r/2<s<r\}}$ and $\chi_{-}=\chi_{\{r<s<2r\}}$. Sometimes we omit $\chi_{+}$ and $\chi_{-}$ when we have already restricted the range of $r,s$ appropriately.

Before we give the details for the proof, we introduce a few lemmas to estimate the finite differences.
\begin{lemma} \label{difference estimates for $s<r$}
The following sequences
 \begin{align*}
    \myT_k^{+}(r,s)&=\left(\frac{s}{r}\right)^{\mu_k}\chi_+,\\
    {\widetilde{ \myT_k}}^{+}(r,s)& =\frac{\sin(-\pi b_{k})}{\pi}\frac{1}{1-\left(\frac{s}{r}\right)^2}\left[\left(\frac{s}{r}\right)^{{\mu_k-\mu_0}}-1\right] \chi_+,
             \end{align*}
             satisfy \eqref{uniform bound} and \eqref{uniform difference bound} with implicit constants that do not depend on $r,s$.
   \end{lemma}

\begin{proof}\noindent\emph{ Boundedness}: Notice that $\left(\frac{s}{r}\right)^{\mu_k}= (1-(1-\frac{s}{r}))^{\mu_k}\lesssim \min \{1, \mu_k (1-\frac{s}{r})\}$ as $k$ goes to $\infty$, this
gives the uniform bound for $\myT_k^{+}(r,s)$ and $   \widetilde{ \myT}_k^{+}(r,s)$. Therefore, we only need to estimate their finite difference.

\noindent\emph{ Finite difference estimate for $\myT_k^+(r,s)$.} Notice that $\mu_{k+1}=\mu_k+1$, direct computation  gives
\begin{align*}
		|\fd^{N} \myT_k^+|= \left(\frac{s}{r}\right)^{\mu_k}\left(1-\frac{s}{r}\right)^{N}.
\end{align*}
Thus, taking  $N= \lfloor \frac{d-1}{2}\rfloor+1$,
\begin{align*}
	\sup_j 2^{j(N-1)}\sum_{k=2^j}^{2^{j+1}} |\fd ^{N} \myT_k^+|\lesssim   \sup_j 2^{j(N-1)} \left(\frac{s}{r}\right)^{2^{j}}\left(1-\frac{s}{r}\right)^N\frac{1-\left(\frac{s}{r}\right)^{2^{j}}}{\left[1-\left(\frac{s}{r}\right)\right]}  \lesssim 1.
\end{align*}
The above bound is valid because $x^N k^{N}(1-x)^k$ is uniformly bounded for $x\in(0, \frac{1}{2})$ and $k\in  \mathbb N$, by
\begin{align*}
	x^N k^{N}(1-x)^k\le (xk)^N  e^{-kx}\le \sup_{t\ge0} t^N e^{-t}\lesssim_N 1.
\end{align*}

\noindent\emph{ Finite difference estimate for $\widetilde \myT_k^{+}(r,s) $.} To estimate this term, let us denote
  $\myF_k=\frac{\sin \pi b_k}{\pi}$ and $\myG_k=\frac{1}{1-\left(\frac{s}{r}\right)^2}\left[\left(\frac{s}{r}\right)^{{\mu_k-\mu_0}}-1\right] \chi_+$. Using \eqref{estimates for $b_k$} and the \FDB{},  it is easy to prove that for  $N= \lfloor \frac{d-1}{2}\rfloor+1 $,
\begin{align}
    |\fd^N \myF_k |\leq \left|\frac{\dd^N \myF_k  }{\dd k^N} \right| \lesssim \frac{1}{k^{N+1}}.
\end{align}
 Hence, by the Leibniz rule for  the finite difference, we have
\begin{align*}
    |\fd^{N}\widetilde \myT_k^+|&=\left|\sum_{m=0}^N \binom{N}{m} \fd ^{N-m}\left( \frac{\sin \pi b_k}{\pi}\right)\fd^{m}\myG_{k+N-m} \right|\\
    &\lesssim  \frac{1}{1-(\frac{s}{r})^2}\sum_{m=1}^N \left(\frac{1}{k}\right)^{N-m+1}\left(\frac{s}{r}\right)^{\mu_{k+N-m}-\mu_0}\left(1-\frac{s}{r}\right)^{m}\\
    &\quad + \frac{1}{1-(\frac{s}{r})^2} \left(\frac{1}{k}\right)^{N+1}\left[\left(\frac{s}{r}\right)^{\mu_{k+N}-\mu_0}-1\right]\\
     &\lesssim  \frac{1}{1-(\frac{s}{r})^2}\sum_{m=1}^N \left(\frac{1}{k}\right)^{N-m+1}\left(\frac{s}{r}\right)^{\mu_{k+N-m}-\mu_0}\left(1-\frac{s}{r}\right)^{m}+ k^{-N}.
\end{align*}
Therefore,
\begin{align*}
    &\quad \sup_j 2^{j(N-1)}\sum_{k=2^j}^{2^{j+1}} |\fd ^{N}\widetilde \myT_k^+|\\
    &\lesssim \sup_j 2^{j(N-1)}\sum_{k=2^j}^{2^{j+1}}   \sum_{m=1}^N \frac{1}{1-(\frac{s}{r})^2}  \left(\frac{1}{k}\right)^{N-m+1}\left(\frac{s}{r}\right)^{\mu_{k+N-m}-\mu_0}\left(1-\frac{s}{r}\right)^{m}+1 \\
    &\lesssim     \frac{1}{1-(\frac{s}{r})^2} \sum_{m=1}^N \sup_j 2^{j(N-1)}\sum_{k=2^j}^{2^{j+1}}\left(\frac{1}{k}\right)^{N-m+1}\left(\frac{s}{r}\right)^{\mu_{k+N-m}-\mu_0}\left(1-\frac{s}{r}\right)^{m}+1\\
       &\lesssim \frac{1}{1-(\frac{s}{r})^2} \sum_{m=1}^N \sup_j 2^{j(N-1)}(2^{-j})^{N-m+1}\left(1-\frac{s}{r}\right)^{m}\left(\frac{s}{r}\right)^{2^{j}}
       \frac{1-\left(\frac{s}{r}\right)^{2^{j}}}{1-\left(\frac{s}{r}\right)}+1
        \\
       &\lesssim  \sum_{m=1}^N  \sup_k k^{(m-1)} \left(1-\frac{s}{r}\right)^{m-1}\left(\frac{s}{r}\right)^{k} +1<\infty.
\end{align*}%
This finishes the proof.
%
\end{proof}
Later, we need to estimate the finite differences for the sequence $\left(\frac{r}{s}\right)^{\nu_k} \chi_{-}$. In the proof of Lemma \ref{difference estimates for $s<r$}, we used the fact that $\mu_k$ is an arithmetic sequence, which allowed us to compute the sum explicitly.
 However, $ \nu_k$ is not generally in arithmetic progression. Instead, notice that $\nu_k-\mu_k$ goes to $0$ as $k  \to \infty$. Let us  define $\gamma_k= \nu_k-\mu_k$. By Lemma \ref{estimates for $a_k$ and $b_k$}, it is easy to deduce that
 \begin{align}\label{finite difference for gamma  k}
   |\fd^N \gamma_k|\lesssim \frac{C}{k^{N+1}}.
\end{align}
In order to estimate the finite difference for some sequence which involves $\left(\frac{r}{s}\right)^{\nu_k}\chi_{-}$, we first introduce the following lemma which can be proved by induction.

\begin{lemma}\label{lemma:long lemma}Let $\theta \in(1/2,1)$ and define $F(t)\coloneq 1-\theta^t$ for $t\in \mathbb R$. For every $N\ge 1$, there exist $L\in\mathbb N$ and finite sequences $a^N_\ell, b^N_{\ell,j},c^N_\ell,d^N_{\ell,i},e^N_{\ell,i}$ ($\ell=1,2,\dots,L$)  such that (dropping the $N$ index for legibility)
\begin{align}
d_{\ell,i} \ge 1, &\quad \sum_{i=0}^{c_\ell} d_{\ell,i} = N,\\
\fd^N\theta^{\gamma_k} &=\sum_{\ell=0}^N a_{\ell}
	\bigg(
	\prod_{j=0}^{N-1} \theta^{b_{\ell,j}\gamma_{k+j} }\bigg)
	\bigg(\prod_{i=0}^{c_\ell}F(\fd^{d_{\ell,i}} \gamma_{k+e_{\ell,i}})\bigg).
\label{badlemma-form}\end{align}
\end{lemma}
\begin{proof}
	This follows from repeated application of the following basic calculations: if $s_k$ is any sequence, then
	\begin{align}
		\fd (\theta^{s_k}) &= \theta^{s_k} F(\fd s_k),\label{badlemma-basecase}\\
		\fd F(s_k) &=- \theta^{s_k} F(\fd s_k);
	\end{align}
and the following Leibniz rule for the finite difference of product of $n$ terms:
$$	\fd (a^1_k\dots a^n_k) = \sum_{r=1}^n a^1_k\dots a^{r-1}_k (\fd a^r_k)a^{r+1}_{k+1} \dots a^n_{k+1}.$$
The base case $N=1$ is \eqref{badlemma-basecase} which is of the correct form \eqref{badlemma-form}. For $N=2$,
\begin{align}\nonumber \fd ^2(\theta^{\gamma_k}) =\fd (\theta^{\gamma_k}F(\fd \gamma_k))
&= (\fd \theta^{\gamma_k}) F(\fd \gamma_{k+1}) +\theta^{\gamma_k} \fd{}F(\fd{} \gamma_k)
\\
&=\theta^{\gamma_k} F(\fd \gamma_k) F(\fd \gamma_{k+1})-\theta^{\gamma_k}\theta^{\fd \gamma_k}F(\fd^2\gamma_k)
\end{align}	which is again of the form \eqref{badlemma-form}. Inductively, if a derivative falls on a term $\theta^{\gamma_{k+j}}$, it creates a term $F(\fd \gamma_{k+j})$, or else it falls on a term $F(\fd^{d_{\ell,i}} \gamma_{k+e_{\ell,i}})$; in either case we obtain the new coefficients $d^{N+1}_{\ell,i}\ge 1$ sum to $N$, as claimed.
\end{proof}
By Lemma \ref{lemma:long lemma},  \eqref{finite difference for gamma  k} and the fact that for $\theta\in (1/2,1), t>0$,
\begin{align*}
    |F(t)|\lesssim  t(1-\theta),
    \end{align*}
we have
\begin{align}\label{finite difference for s r gamma k}
    \fd ^N \left(\frac{r}{s}\right)^{\gamma_k} \leq \frac{1}{k^{N+1}}.
\end{align}

\begin{lemma} \label{difference estimates for $r<s<2r$}
The following sequences
 \begin{align*}
    \myT_k^{-}(r,s)&=\left(\frac{r}{s}\right)^{\nu_k}\chi_-,\\
    \widetilde{ \myT}_k^{-}(r,s)& =\frac{\sin(\pi b_{k})}{\pi}\frac{1}{1-(\frac{r}{s})^2}\left[\left(\frac{r}{s}\right)^{{\nu_k-\nu_0}}-1\right]\chi_-,
             \end{align*}
               satisfy \eqref{uniform bound} and \eqref{uniform difference bound} with constants that do not depend on $r,s$.
   \end{lemma}
   \begin{proof}\noindent\emph{Boundedness.} It is easy to verify that both $\myT_k^{-}(r,s)$
   and $\widetilde{ \myT}_k^{-}(r,s)$ are uniformly bounded in $k, r$ and $s$.

 \noindent\emph{Finite difference for $\myT_k^{-}(r,s)$.}  We write
\begin{align*}
\left(\frac{r}{s}\right)^{\nu_k}=\left(\frac{r}{s}\right)^{\mu_k}\left(\frac{r}{s}\right)^{\gamma_k}.
\end{align*}
From Lemma \ref{difference estimates for $s<r$}, we know that
\begin{align}\label{fd for r s mu k}
  \sup_{j} 2^{j(N-1)}\sum_{k=2^j}^{2^{j+1}}| \fd^N{\left(\frac{r}{s}\right)^{\mu_k}}|<\infty.
\end{align}
Applying Lemma \ref{Difference estimates for product} with $\myG_k=\left(\frac{r}{s}\right)^{\mu_k} \chi_{-}$ and $\myF_k=\left(\frac{s}{r}\right)^{\gamma_k}\chi_{-}$, and using the estimate \eqref{finite difference for s r gamma k} and \eqref{fd for r s mu k}, we deduce that $\left(\frac{r}{s}\right)^{\nu_k}$ satisfy \eqref{uniform difference bound}.

 \noindent\emph{Finite difference for $\widetilde{ \myT}_k^{-}(r,s)$.} Notice that
 $$\fd ^m \left[\frac{1}{1-(\frac{r}{s})^2} \left\{\left(\frac{r}{s}\right)^{{\nu_k-\nu_0}}-1\right\}\right]= \frac{\left(\frac{r}{s}\right)^{-\nu_0}}{1-(\frac{r}{s})^2}\fd ^m\left[ \left(\frac{r}{s}\right)^{{\mu_k}} \left(\frac{r}{s}\right)^{{\gamma_k}}\right]$$
 for $1\leq m \leq N= \lfloor \frac{d-1}{2}\rfloor+1 $.
Then,
\begin{align*}
    &\quad \fd^N  \widetilde{ \myT}_k^{-}(r,s)\\
    &= \fd ^N \left[\frac{\sin\pi b_{k}}{\pi}\right] \frac{1}{1-(\frac{r}{s})^2} \left[\left(\frac{r}{s}\right)^{{\nu_k-\nu_0}}-1\right]\\
&\quad   + \sum_{m=1}^N \binom{N}{m} \fd^{N-m}\left[\frac{\sin\pi b_{k}}{\pi}\right] \fd^{m}  \left[\frac{\left(\frac{r}{s}\right)^{-\nu_0}}{1-(\frac{r}{s})^2} \left(\frac{r}{s}\right)^{{\mu_{k+N-m}}} \left(\frac{r}{s}\right)^{{\gamma_{k+N-m}}}\right]\\
  &= \fd ^N \left[\frac{\sin\pi b_{k}}{\pi}\right] \frac{1}{1-(\frac{r}{s})^2} \left[\left(\frac{r}{s}\right)^{{\nu_k-\nu_0}}-1\right]\\
  &\quad+ \sum_{m=1}^N\!
  \binom{N}{m}
  \fd^{N-m}\!
  \Big[
    \frac{\sin\pi b_{k}}{\pi}
  \Big]
  \sum_{l=0}^m
  \binom{m}{l}
  \fd^{l}\!\!
  \left[
    \frac{(\frac{r}{s})^{-\nu_0}(\frac{r}{s})^{{\mu_{k+N-m}}}}
        {1-(\frac{r}{s})^2}
    \right]
    \fd^{m-l}\!
    \Big[
        \Big(
            \frac{r}{s}
        \Big)^{{\gamma_{k+N-l}}}
    \Big]\\
&  \lesssim  \frac{1}{k^{N}}+\sum_{m=1}^N\binom{N}{m} \frac{1}{k^{N-m+1}}  \sum_{l=0}^m \binom{m}{l} \frac{\left(\frac{r}{s}\right)^{-\nu_0}}{1-(\frac{r}{s})^2} \left(\frac{r}{s}\right)^{{\mu_{k+N-m}}} \left(1-\frac{r}{s}\right)^l \frac{1}{k^{m-l}}.
\end{align*}
Therefore
\begin{align*}
   &\quad \sup_{j} 2^{j(N-1)} \sum_{k=2^j}^{2^{j+1}}|\fd^N  \widetilde{ \myT}_k^{-}(r,s)|\\
 &\lesssim
 \sup_{j}  2^{j(N-1)} \sum_{k=2^j}^{2^{j+1}}   \frac{1}{k^{N}}
 \\
 &\quad +\sup_{j} \sum_{m=1}^N\binom{N}{m} \sum_{l=0}^m \binom{m}{l} \frac{\left(\frac{r}{s}\right)^{-\nu_0}}{1-(\frac{r}{s})^2}\sum_{k=2^j} ^{\mathclap{2^{j+1}}} \frac{k^{N-1}}{k^{N-l+1}}\left(\frac{r}{s}\right)^{{\mu_{k+N-m}}} \! \left(1-\frac{r}{s}\right)^l \\
  &\lesssim  1+\sup_{k} \sum_{m=1}^N\binom{N}{m} \frac{1}{k}  \frac{1}{1-(\frac{r}{s})^2}\left(\frac{r}{s}\right)^{\mu_k} \\
   &\quad +\sup_{k} \sum_{m=1}^N\binom{N}{m}  \sum_{l=1}^m \binom{m}{l} \frac{1}{1-(\frac{r}{s})^2} \frac{1}{k^{2-l}} \left(\frac{r}{s}\right)^{{\mu_k}}\left(1-\frac{r}{s}\right)^l \frac{1-\left(\frac{r}{s}\right)^{\mu_k}}{1-\left(\frac{r}{s}\right)}\\
  &\lesssim  1+1+\sup_{k}\sum_{m=1}^N\binom{N}{m}  \sum_{l=1}^m \binom{m}{l}k^{l-1} \left(1-\frac{r}{s}\right)^{l-1} \left(\frac{r}{s}\right)^{{\mu_k}}\\
  &\lesssim 1.\qedhere
  \end{align*}
%
\end{proof}

Now we are in position to prove the following inequality
\begin{align}
	\| \widetilde W_2 f(r, \omega)\|_{L^p(\frac{\dd r}{r} \cdot \dd \omega)} \lesssim   \| f(r, \omega)\|_{L^p(\frac{\dd r}{r} \cdot \dd \omega)}.
\end{align}

 First, using the asymptotic behavior in \eqref{asymptotic expansion for the quotient of gamma function} for gamma functions, we have for any fixed $k$,  as $n$ goes to $\infty$,
\begin{align}
    A_{k,n}^{+}&=1+\frac{a}{4(n+1)}+O_k\big(\frac{1}{(n+1)^2}\big),\label{expansion for $A_{k,n}^{+}$}\\
    A_{k,n}^{-}&=1-\frac{a}{4(n+1)}+O_k\big(\frac{1}{(n+1)^2}\big).\label{expansion for $A_{k,n}^{-}$}
\end{align}
Denote
\begin{align}\label{definition of E}
    E_{k,n}^{+}&\acoloneq  A_{k,n}^{+}-1-\frac{a}{4(n+1)},\quad E_{k,n}^{-} \acoloneq   A_{k,n}^{-}-1+\frac{a}{4(n+1)}, \\
   E_{k}^+(r,s)&\acoloneq \sum_{n=0}^{\infty}E_{k,n}^{+} \left(\frac{s}{r}\right)^{2n}\chi_+,
 \quad E_{k}^-(r,s)\acoloneq \sum_{n=0}^{\infty}E_{k,n}^{-}\left(\frac{r}{s}\right)^{2n}\chi_-.
 \end{align}
Using \eqref{expansion for $A_{k,n}^{+}$} and \eqref{expansion for $A_{k,n}^{-}$}, and the following Taylor series,
\begin{align}\label{infinite sum}
    \sum_{n=0}^\infty \frac{x^n}{n+1} =\frac{\ln(1-x)}{x},\quad x\in (0,1),
\end{align}
we can write
\begin{align}\nonumber
    \widetilde K_{k,2} &= 2\left(\frac{s}{r}\right)^{\frac{d}{2}-\frac{d}{p}+1+\mu_k} \frac{\sin(-\pi b_{k})}{\pi} \left[\frac{1}{1-(\frac{s}{r})^2}+\frac{a}{4}\left(\frac{s}{r}\right)^{-2} \ln\!\Big({1-\Big(\frac{s}{r}\Big)^2}\Big) +E_{k}^{+}\right]\chi_+
   \nonumber \\\nonumber
    & \quad +2\left(\frac{r}{s}\right)^{\frac{d}{p}-\frac{d}{2}+1+\nu_k} \frac{\sin(\pi b_{k})}{\pi} \left[\frac{1}{1-(\frac{r}{s})^2} -\frac{a}{4}\left(\frac{r}{s}\right)^{-2} \ln\!\Big({1-\Big(\frac{r}{s}\Big)^2}\Big) +E_{k}^{-} \right]\chi_-
    \\\label{singularity of kernels}
    &=\vcentcolon \widetilde K_{k,2}^1 (r,s)+\widetilde K_{k,2}^2 (r,s)+\widetilde K_{k,2}^3 (r,s)
\end{align}
with
\begin{align*}
       \widetilde K_{k,2}^1(r,s)&=2\left(\frac{s}{r}\right)^{\frac{d}{2}-\frac{d}{p}+1+\mu_k} \frac{\sin(-\pi b_{k})}{\pi}\frac{1}{1-\left(\frac{s}{r}\right)^2}\chi_+\\
      &\quad +2\left(\frac{r}{s}\right)^{\frac{d}{p}-\frac{d}{2}+1+\nu_k} \frac{\sin(\pi b_{k})}{\pi}\frac{1}{1-(\frac{r}{s})^2}\chi_-;\\
        \widetilde K_{k,2}^2 (r,s)&=\frac{1}{2}a \left(\frac{s}{r}\right)^{\frac{d}{2}-\frac{d}{p}-1+\mu_k} \frac{\sin(- \pi b_k)}{\pi}\ln\!\Big({1-\Big(\frac{s}{r}\Big)^2}\Big) \chi_+\\
    &\quad -\frac{1}{2}a \left(\frac{r}{s}\right)^{\frac{d}{p}-\frac{d}{2}-1+\nu_k} \frac{\sin (\pi b_k)}{\pi}\ln\!\Big({1-\Big(\frac{r}{s}\Big)^2}\Big) \chi_-;\\
    \widetilde K_{k, 2}^3 (r,s)&=2 \left(\frac{s}{r}\right)^{\!\frac{d}{2}-\frac{d}{p}+1+\mu_k} \frac{\sin(- \pi b_k)}{\pi} E_{k}^+(r,s)\\
    &\quad +2 \left(\frac{r}{s}\right)^{\!\frac{d}{p}-\frac{d}{2}+1+\nu_k} \frac{\sin (\pi b_k)}{\pi}E_{k}^-(r,s).
\end{align*}
We denote the corresponding operators $\widetilde W_{2}^1, \widetilde W_{2}^2$ and $\widetilde W_{2}^3$ by
\begin{align}
   \widetilde W_{2}^{i} f(r, \omega)=\int_0^\infty \sum _{k=0}^\infty \widetilde K_{k,2}^i(r,s) \proj_k f(s, \omega)\frac{\dd s}{s} \quad i=1, 2, 3.
\end{align}
\vspace{0.2in}
\noindent\emph{\bf Estimate for the first term $\widetilde W_{2}^{1}$.}

 Step 1: \emph{Estimate for the approximate kernels: }Let us define the approximate kernels by
\begin{align*}
	\widetilde K_{k,2}^{1, \text{ap}}(r,s)&\coloneq  2 \left(\frac{s}{r}\right)^{\frac{d}{2}-\frac{d}{p}+1+\mu_0} \frac{\sin (-\pi b_k)}{\pi}\frac{1}{1-\left(\frac{s}{r}\right)^2}\chi_+
	\\ &\quad + 2 \left(\frac{r}{s}\right)^{\frac{d}{p}-\frac{d}{2}+1+\nu_0} \frac{\sin \pi b_k}{\pi}\frac{1}{1-(\frac{r}{s})^2}\chi_-,
\end{align*}
and  the corresponding operator $\widetilde W_{2}^{1, \text{ap}}f (r, \omega)$   by
\begin{align*}
	\widetilde W_{2}^{1, \text{ap}}f(r,\omega)\ &\acoloneq \sum_{k=0}^\infty \int_{0}^\infty \widetilde K_{k,2}^{1, \text{ap}}(r,s)\proj_k f(s, \omega)\frac{\dd s}{s}\\
	&=2\int_\frac{r}{2}^r  \left(\frac{s}{r}\right)^{\frac{d}{2}-\frac{d}{p}+1+\mu_0}\frac{1}{1-(\frac{s}{r})^2}  \sum_{k=0}^\infty \frac{\sin (-\pi b_k)}{\pi}  \proj_{k }f(s, \omega)\frac{\dd s}{s}\\
	&\quad +2\int_{r}^{2r}  \left(\frac{r}{s}\right)^{\frac{d}{p}-\frac{d}{2}+1+\nu_0}\frac{1}{1-(\frac{r}{s})^2}  \sum_{k=0}^\infty\frac{\sin \pi b_k}{\pi}  \proj_{k }f(s, \omega)\frac{\dd s}{s}.
	\end{align*}
Using  Lemma \ref{boundedness for CZ operator 1}, we have
\begin{align*}
	\|\widetilde W_{2}^{1, \text{ap}}f(r,\omega)\|_{L^p(\frac{dr}{r})}\lesssim  \bigg\|\sum_{k=0}^\infty \sin \pi b_k \proj_{k }f(r, \omega)\bigg\|_{L^p(\frac{dr}{r})}.
\end{align*}
It is easy to check that the sequence $\{\sin \pi b_k\}$ satisfy the condition \eqref{uniform bound} and \eqref{uniform difference bound} in Theorem \ref{multiplier theorem for spherical harmonic decomposition}, thus
\begin{align}\label{ W 2-approximation operator }
	\|\widetilde W_{2}^{1, \text{ap}}f(r,\omega)\|_{L^p(\frac{dr}{r}\cdot d\omega)}\lesssim  \bigg\|\sum_{k=0}^\infty  \proj_{k }f(r, \omega)\bigg\|_{L^p(\frac{dr}{r}\cdot d\omega)}.
\end{align}

Step 2:\noindent\emph{ Estimate for the error term}. Define the kernels of the error terms
\begin{align*}
\widetilde K_{k,2}^{1, \text{err}}(r,s)&=\widetilde K_{k,2}^{1}(r,s)-\widetilde K_{k,2}^{1, \text{ap}}(r,s)\\
    &=2\left(\frac{s}{r}\right)^{\frac{d}{2}-\frac{d}{p}+1+\mu_0} \frac{\sin(-\pi b_{k})}{\pi}\frac{1}{1-\left(\frac{s}{r}\right)^2}\left[\left(\frac{s}{r}\right)^{{\mu_k-\mu_0}}-1\right] \chi_+\\
    &\quad +2\left(\frac{r}{s}\right)^{\frac{d}{p}-\frac{d}{2}+1+\nu_0} \frac{\sin(\pi b_{k})}{\pi}\frac{1}{1-(\frac{r}{s})^2}\left[\left(\frac{r}{s}\right)^{{\nu_k-\nu_0}}-1\right] \chi_-,
\end{align*}
and the associated operator
\begin{align*}
   \widetilde W_{2}^{1, \text{err}} f(r, \omega)=\int_0^\infty \sum _{k=0}^\infty \widetilde K_{k,2}^{1, \text{err}}(r,s) \proj_k f(s, \omega)\frac{\dd s}{s}.
\end{align*}
By Lemma \ref{difference estimates for $s<r$} and Lemma \ref{difference estimates for $r<s<2r$},  we have
\begin{align*}
   | \widetilde K_{k,2}^{1, \text{err}}(r,s)|\lesssim 1,
\end{align*}
and for $N=\lfloor \frac{d-1}{2} \rfloor+1$,
\[  \sup_{j}2^{j(N-1)} \sum_{k=2^j}^{2^{j+1}}|\fd^N \widetilde K_{k,2}^{1, \text{err}}(r,s)|\lesssim 1.
\]
That is, $\widetilde K_{k,2}^{1, \text{err}}(r,s)$ satisfies \eqref{uniform bound} and \eqref{uniform difference bound}, hence
\begin{align*}
  \| \widetilde W_{2}^{1, \text{err}} f(r, \omega) \|_{L^p(\dd \omega)} \leq \int_{r/2}^r  \|\sum _{k=0}^\infty \proj_k f(s, \omega)\|_{L^p(\dd \omega)}\frac{\dd s}{s}.
\end{align*}
Furthermore, we have
\begin{align}\label{lp bd for error term}
   \| \widetilde W_{2}^{1, \text{err}} f(r, \omega)\|_{L^p(\frac{dr}{r}\cdot \dd\omega)}  \leq \|f(r, \omega)\|_{L^p(\frac{dr}{r}\cdot \dd\omega)}.
\end{align}
Combing \eqref{ W 2-approximation operator } and \eqref{lp bd for error term}, we get the $L^p$-boundedness of $\widetilde W_{2}^1$.

\vspace{0.2in}
{{\noindent\emph{\bf Estimate for the second term $\widetilde W_{2}^{2} $.}}}
Taking $\myF_k= \frac{\sin (-\pi b_k)}{\pi}$ and $$\myG_k=\left(\frac{s}{r}\right)^{\frac{d}{2}-\frac{d}{p}-1+\mu_k}\chi_++\left(\frac{r}{s}\right)^{\frac{d}{p}-\frac{d}{2}-1+\nu_k}\chi_-$$
in Lemma \ref{Difference estimates for product}, and by the estimates in  Lemma \ref{difference estimates for $s<r$} and Lemma \ref{difference estimates for $r<s<2r$}, we have
\begin{align*}
    | \widetilde K_{k,2}^2 (r,s)|&\lesssim \left|\ln\!\Big({1-\Big(\frac{s}{r}\Big)^2}\Big)\right|\chi_++\left|\ln\!\Big({1-\Big(\frac{r}{s}\Big)^2}\Big)\right|\chi_-,\\
   \sup_j   2^{j(N-1)} \sum_{k=2^j}^{2^{j+1}}|\fd^N \widetilde K_{k,2}^2(r,s)|&\lesssim
      \left|\ln\!\Big({1-\Big(\frac{s}{r}\Big)^2}\Big)\right|\chi_++\left|\ln\!\Big({1-\Big(\frac{r}{s}\Big)^2}\Big)\right|\chi_-.
\end{align*}
By Theorem \ref{multiplier theorem for spherical harmonic decomposition}, we have
\begin{align*}
  \|\widetilde W_{2}^{2}  f(r, \omega) \|_{L^p(\dd \omega)}& \leq \int_{r/2}^r \left|\ln\!\Big({1-\Big(\frac{s}{r}\Big)^2}\Big)\right|\left \|\sum _{k=0}^\infty  \proj_k f(s, \omega)\right\|_{L^p(\dd \omega)}\frac{\dd s}{s}\\
  & \quad +\int_r^{2r}\left|\ln\!\Big({1-\Big(\frac{r}{s}\Big)^2}\Big)\right|\left\|\sum _{k=0}^\infty  \proj_k f(s, \omega)\right\|_{L^p(\dd \omega)}\frac{\dd s}{s}.
\end{align*}
Using Lemma \ref{boundedness for CZ operator 2}, we deduce
\begin{align}\label{W 2 second term}
      \|\widetilde W_{2}^{2}  f(r, \omega) \|_{L^p(\frac{\dd r}{r} \cdot \dd{\omega})} \lesssim  \|f(r, \omega)\|_{L^p(\frac{\dd r}{r} \cdot \dd{\omega})}.
\end{align}

\vspace{0.2in}

\noindent\emph{\bf Estimate for the third term $\widetilde W_{2}^{3}$.} Recall that $\widetilde K_{k, 2}^3$ is given by
\begin{align*}
  \widetilde K_{k, 2}^3 (r,s)
  & = 2 \left(\frac{s}{r}\right)^{\frac{d}{2}+1+\mu_k-\frac{d}{p}}\frac{\sin (-\pi b_k)}{\pi}\sum_{n=0}^{\infty} E_{k,n}^{+}\left(\frac{s}{r}\right)^{2n}\chi_+\\
   &\quad +2 \left(\frac{r}{s}\right)^{\frac{d}{p}+1-\nu_k+\frac{d}{2}} \frac{\sin \pi b_k}{\pi}\sum_{n=0}^{\infty}E_{k,n}^{-}\left(\frac{r}{s}\right)^{2n} \chi_-. \end{align*}
The following identities are valid:
\begin{align} \label{A equal to E}
    A_{k,n}^{+}=1+\frac{a}{4(n+1)}+E_{k,n}^{+},\
     A_{k,n}^{-}=1-\frac{a}{4(n+1)}+E_{k,n}^{-}.
\end{align}
When we consider $A_{k,n}^{+}$, $A_{k,n}^{-}$, $E_{k,n}^{+}$and  $E_{k,n}^{-}$ as smooth functions of $k$, then by \eqref{A equal to E} for $N\geq 1$, we have
\begin{align*}
      \frac{\dd^N A_{k,n}^{+}}{\dd k^N}=  \frac{\dd^N E_{k,n}^{+}}{\dd k^N},\quad \text{and}
      \quad
      \frac{\dd^N A_{k,n}^{-}}{\dd k^N}=  \frac{\dd^N E_{k,n}^{-}}{\dd k^N}.
\end{align*}
For their derivatives, we have the following estimates.
   \begin{lemma}\label{derivative estimate for the reminder term}
For $A_{k,n}^{+}$, $A_{k,n}^{-}$, $E_{k,n}^{+}$and  $E_{k,n}^{-}$, and $N\geq 1$,  we have
\begin{align}
     \bigg|\frac{\dd^N A_{k,n}^{+}}{\dd k^N}\bigg|= \bigg| \frac{\dd^N E_{k,n}^{+}}{\dd k^N}\bigg| \lesssim \frac{1}{k^N} \frac{1}{n+1}, \label{derivative for A plus}\\
          \bigg|\frac{\dd^N A_{k,n}^{-}}{\dd k^N}\bigg|=\bigg|\frac{\dd^N E_{k,n}^{-}}{\dd k^N}\bigg| \lesssim \frac{1}{k^N} \frac{1}{n+1}  \label{derivative for A minus}.
\end{align}

\end{lemma}
\begin{proof} See Appendix \ref{appendix:proof-of-lemma-derivative estimate for the reminder term}.
\end{proof}
Using Lemma \ref{derivative estimate for the reminder term},  \eqref{infinite sum} and \eqref{definition of E}, we have
\begin{align}
   \left | \frac{\dd^N E_{k}^+(r,s)}{\dd k^N} \right|&\lesssim \frac{1}{k^N} \Big|\ln\left(1-\big(\frac{s}{r}\big)^2\right)\Big|\chi_+, \\
  \left  |\frac{\dd^N E_{k}^-(r,s)}{\dd k^N} \right| &\lesssim \frac{1}{k^N} \Big|\ln\left(1-\big(\frac{r}{s}\big)^2 \right)\Big|\chi_-.
\end{align}

Let
    \begin{align*}
        { \myE_k^{+}}(r,s)&= \frac{2\sin(-\pi b_{k})}{\pi}\left(\frac{s}{r}\right)^{\frac{d}{2}-\frac{p}{d}+1+\mu_k} E_k^+(r,s )\chi_+,\\
            \myE_k^{-}(r,s)&=  \frac{2\sin(\pi b_{k})}{\pi}\left(\frac{r}{s}\right)^{\frac{d}{p}-\frac{p}{2}+1+\nu_k} E_k^-(r,s )\chi_-.
   \end{align*}
          Then, using Lemma \ref{Difference estimates for product} and Lemma \ref{difference estimates for $s<r$} again yields
          \begin{align*}
                 |\myE_k^{+}(r,s) |\lesssim \Big|\ln\!\Big({1-\Big(\frac{s}{r}\Big)^2}\Big)\Big|\chi_+, \\
               \sup_{j }2^{j(N-1)} \sum_{k=2^j}^{2^{j+1}} |\fd ^N   \myE_k^{+}(r,s) |\lesssim      \Big|\ln\!\Big({1-\Big(\frac{s}{r}\Big)^2}\Big)\Big|\chi_+,\\
                |\myE_k^{-}(r,s) |\lesssim \Big|\ln\!\Big({1-\Big(\frac{r}{s}\Big)^2}\Big)\Big|\chi_-, \\
               \sup_{j }2^{j(N-1)} \sum_{k=2^j}^{2^{j+1}} |\fd ^N   \myE_k^{-}(r,s) |\lesssim     \Big| \ln\!\Big({1-\Big(\frac{r}{s}\Big)^2}\Big)\Big|\chi_-.
          \end{align*}
Hence, by Theorem \ref{multiplier theorem for spherical harmonic decomposition}, we have
\begin{align*}
  \|\widetilde W_{2}^{3} f(r, \omega) \|_{L^p(\dd \omega)} \leq& \int_{r/2}^r \Big|\ln\!\Big({1-\Big(\frac{s}{r}\Big)^2}\Big)\Big| \Big\|\sum _{k=0}^\infty  \proj_k f(s, \omega)\Big\|_{L^p(\dd \omega)}\frac{\dd s}{s}\\
  &+\int_r^{2r}\Big|\ln\!\Big({1-\Big(\frac{r}{s}\Big)^2}\Big)\Big|\Big\|\sum _{k=0}^\infty  \proj_k f(s, \omega)\Big\|_{L^p(\dd \omega)}\frac{\dd s}{s}.
\end{align*}
Furthermore, Lemma \ref{boundedness for CZ operator 2} gives
\begin{align}\label{W 2 third term}
      \| \widetilde W_{2}^{3} f(r, \omega) \|_{L^p(\frac{\dd r}{r} \cdot \dd{\omega})} \lesssim  \|f(r, \omega)\|_{L^p(\frac{\dd r}{r} \cdot \dd{\omega})}.
\end{align}
Combining \eqref{ W 2-approximation operator }, \eqref{lp bd for error term}, \eqref{W 2 second term} and \eqref{W 2 third term}, we have proved that
\begin{align}\label{W 2}
   \| \widetilde W_2 f(r)\|_{L^p(\frac{\dd r}{r} \cdot \dd{\omega})} \lesssim  \|f(r, \omega)\|_{L^p(\frac{\dd r}{r} \cdot \dd{\omega})}.
\end{align}
This inequality together with
\eqref{W 1} and \eqref{W 3} implies that $\tilde W$ is bounded in $L^p(\frac{\dd r}{r} \cdot \dd{\omega})$. Therefore,  we conclude the proof of Theorem \ref{thm:main}.\hfill\qedsymbol

\section{Proof of the $W^{s,p}$-boundedness}
\label{s:proof-of-wsp-bddness}%
In this section, we will prove the $W^{s,p}$-%
boundedness of the stationary wave operators. Recall the operators
$\mathscr B_{\mu_k}$ and $\mathscr H_{\nu_k}$ introduced in \eqref{Bessel transform} and \eqref{Hankel transform} of Subsection \ref{ss:bessel-hankel}. On each subspace of spherical harmonics, we  define the fractional powers $(\la)^{s/2}$ and $(-\Delta)^{s/2}$ with $s\in\R$ by
\begin{align*}
   (\la)^{s/2}|_{\mathscr H_{k,l}}\coloneq  \mathscr H_{\nu_k}\lambda^{s}\mathscr H_{\nu_k}, \quad
   (-\Delta)^{s/2}|_{\mathscr H_{k,l}}\coloneq  \mathscr B_{\mu_k}\lambda^{s}\mathscr B_{\mu_k},
\end{align*}
and introduce the following Riesz-type operators for $
\alpha,\beta$ satisfying $-d<\alpha<2+2\nu_0, \alpha \neq 0$ and $-2\nu_0-2<\beta<d, \beta \neq 0$:
\begin{align}
  \mathcal R^{\alpha}  f(r, \omega)&\coloneq  (-\Delta)^{\alpha/2}(\la)^{-\alpha/2} f=\sum_{k=0}^\infty\mathscr B_{\mu_k}\lambda^{\alpha}\mathscr B_{\mu_k} \mathscr H_{\nu_k}\lambda^{-\alpha}\mathscr H_{\nu_k}\proj_k f(r, \omega),\label{df of RT}\\
    \mathcal R^{-\beta}  f(r, \omega)&\coloneq  (\la)^{\beta/2}(-\Delta)^{-\beta/2}f=\sum_{k=0}^\infty\mathscr H_{\nu_k}\lambda^{\beta}\mathscr H_{\nu_k}\mathscr B_{\mu_k}\lambda^{-\beta}\mathscr B_{\mu_k} \proj_k f(r, \omega).\label{df of IRT}
\end{align}
Here we assume that $\alpha \neq 0$ and $\beta \neq 0$, since the operator for $\alpha=0$ or $\beta=0$ is identity operator and the $L^p$-boundedness is trivial.  
Due to the calculation \eqref{reduction}, to prove Theorem \ref{thm:main 2}, we only need to prove the $L^p$-boundedness of $\mathcal R^\alpha$ and $\mathcal R^{-\beta}$.

 First, we derive the kernels of $\mathcal R^{\alpha}$ and $\mathcal R^{-\beta}$ using the Mellin transform. These operators are Mellin multipliers on each subspace with the multipliers given by a ratio of products of Gamma functions. The details of this calculation are presented below.

Using the formula for the Mellin transform of Bessel functions, it is easy to verify that
\begin{align*}
   ( \mathcal M \mathscr H_{\nu_k} f)(z)&= 2^{z-{\lambda_0}-1}\frac{\Gamma(\frac{z-{\lambda_0} +\nu_k)}{2})}{\Gamma(1-\frac{z-{\lambda_0}-\nu_k}{2})}(\mathcal M f)(d-z),\\
    ( \mathcal M \mathscr B_{\mu_k} f)(z)&= 2^{z-{\lambda_0}-1}\frac{\Gamma(\frac{z-{\lambda_0} +\mu_k)}{2})}{\Gamma(1-\frac{z-{\lambda_0}-\mu_k}{2})}(\mathcal M f)(d-z),
\end{align*}
where
\begin{align} \label{lambda zero}
    \lambda_0=\tfrac{d-2}{2}.
\end{align}
Hence, on the $k$th spherical harmonic subspace,
\begin{align}\label{Mellin convolution}
   \mathcal M   \mathcal R^\alpha  f(z) = \mathcal H_{4,4}^{2,2} (k, \alpha ; z) \mathcal Mf(z),
\end{align}
with
\begin{align*}
  & \mathcal H_{4,4}^{2,2} (k, \alpha ; z)\\
  =& \mathcal H_{4, 4}^{2, 2}\left[
\begin{array}{lllll}
    (-\frac{\nu_k+{\lambda_0}}{2}, \frac{1}{2}),\hspace{-1em}
    &(-\frac{\mu_k+{\lambda_0}+\alpha}{2}, \frac{1}{2}), \hspace{-1em}
    &(\frac{\nu_k-{\lambda_0}}{2}, \frac{1}{2}),
    &(\frac{\mu_k-{\lambda_0}-\alpha}{2}, \frac{1}{2})\\
    (\frac{\mu_k-{\lambda_0}}{2}, \frac{1}{2}),
    &(\frac{\nu_k-{\lambda_0}-\alpha}{2}, \frac{1}{2}),
    &(-\frac{\mu_k+{\lambda_0}}{2}, \frac{1}{2}),\hspace{-1em}
    &(-\frac{\nu_k+{\lambda_0}+\alpha}{2}, \frac{1}{2})
\end{array}
\Big | k, \alpha ; z \right]\\
  =&\frac{\Gamma(\frac{z}{2}+\frac{\mu_k-{\lambda_0}}{2})\Gamma(\frac{z}{2}+\frac{\nu_k-{\lambda_0}-\alpha}{2})\Gamma(1-\frac{z}{2}
  +\frac{\nu_k+{\lambda_0}}{2})\Gamma(1-\frac{z}{2}+\frac{\mu_k+{\lambda_0}+\alpha}{2})}{\Gamma(\frac{z}{2}+\frac{\nu_k-{\lambda_0}}{2})
  \Gamma(\frac{z}{2}+\frac{\mu_k-{\lambda_0}-\alpha}{2})\Gamma(1-\frac{z}{2}+\frac{\mu_k+{\lambda_0}}{2})\Gamma(1-\frac{z}{2}
  +\frac{\nu_k+{\lambda_0}+\alpha}{2})} ,
\end{align*}
where $\mathcal H_{4, 4}^{2, 2}$ is as defined in \eqref{def: Hmnpq(s)}, with particular choices of the parameters $b_{k,j}$, $\beta_{k,j}$, $a_{k,i}$,and $\alpha_{k,i}$. We have written $k,\alpha;z$ instead of $z$ to emphasize the dependence on $k$ and $
\alpha$. To guarantee that the poles of $\Gamma (b_{k,j}+\beta_{k,j} z)$ and $\Gamma (1-a_{k,i}-\alpha_{k,i} z)$ do not coincide,  we need that
\begin{align*}
   \max\{ b_{k,10}, b_{k, 20}\} <\min \{a_{k, 10}, a_{k, 20}\},
\end{align*}
which is equivalent to
\begin{align}\label{condition to define H-function}
     \max\{ \lambda_0-\mu_k, \lambda_0+\alpha -\nu_k \} < \min \{2+\nu_k+{\lambda_0}, 2+\mu_k+{\lambda_0}+\alpha\}.
\end{align}
To make sure \eqref{condition to define H-function} are valid for all $k \geq 0$, we naturally assume that
\begin{align}\label{restriction on $a$}
   -d< \alpha<2+2\nu_0.
\end{align}
Therefore, taking the inverse Mellin transform on both sides of \eqref{Mellin convolution}, we obtain
\begin{align*}
    \mathcal R^\alpha f (r, \omega)&=\sum_{k=0}^\infty \frac{1}{2\pi i} \int_{C} r^{-z}  \mathcal H_{4,4}^{2,2} (k, \alpha ; z)\mathcal M {\proj_k f}(z, \omega)\dd z\\
    &=\sum_{k=0}^\infty \int_0^\infty \underbrace{ \frac{1}{2\pi i} \int_{C} r^{-z}  \mathcal H_{4,4}^{2,2} (k, \alpha ; z)s^{z} \dd z }_{\coloneq  K_k^{\alpha}(r, s)}\proj_k f \frac{\dd s}{s}\\
    &=\sum_{k=0}^\infty \int_0^\infty K_k^\alpha (r,s)\proj_k f(s, \omega)\frac{\dd s}{s},
\end{align*}
where
\begin{align}\label{Kernel function of alpha}
    K_k^\alpha (r,s)&= \frac{1}{2\pi i} \int_{C} \left(\frac{r}{s}\right)^{-z} \mathcal H_{4,4}^{2,2} (k, \alpha ; z)\dd z\\\nonumber
    &=H_{4, 4}^{2, 2}\left[k, \alpha;\,  \frac{r}{s}
\, \middle | \hspace{-1.1em}
\begin{array}{lllll}
    &(-\frac{\nu_k+{\lambda_0}}{2}, \frac{1}{2}),\hspace{-1em}
    &(-\frac{\mu_k+{\lambda_0}+\alpha}{2}, \frac{1}{2}), \hspace{-1em}
    &(\frac{\nu_k-{\lambda_0}}{2}, \frac{1}{2}),
    &(\frac{\mu_k-{\lambda_0}-\alpha}{2}, \frac{1}{2})\\
    &(\frac{\mu_k-{\lambda_0}}{2}, \frac{1}{2}),
    &(\frac{\nu_k-{\lambda_0}-\alpha}{2}, \frac{1}{2}),
    &(-\frac{\mu_k+{\lambda_0}}{2}, \frac{1}{2}),\hspace{-1em}
    &(-\frac{\nu_k+{\lambda_0}+\alpha}{2}, \frac{1}{2})
\end{array}
\right].\end{align}
Hence, we have the following lemma:
\begin{lemma} Let $-d<\alpha<2+2\nu_0$ and $\alpha \neq 0$. Then the operator $\mathcal R^{\alpha}$ defined by \eqref{df of RT} can be written in the following form:
\begin{align*}
     \mathcal R^\alpha f (r, \omega)=\sum_{k=0}^\infty \int_0^\infty K_k^\alpha (r,s)\proj_k f(s, \omega)\frac{\dd s}{s},
\end{align*}
with $K_k^\alpha (r,s)$  given by \eqref{Kernel function of alpha}.  In addition, the parameters defined in \eqref{parameters a},  \eqref{parameters tau}, \eqref{parameters varrpho},  \eqref{parameters delta} are
\begin{align}
    a_k^*=0, \  \Lambda_k=0, \ \delta_k=1, \ \varrho_k=0.
\end{align}
\end{lemma}

For the inverse operators $\mathcal R^{-\beta}$, we only need to exchange the positions of $\nu_k$ and $\mu_k$. We record this in the following lemma:

\begin{lemma} Let $-2\nu_0-2<\beta<d$ and $\beta \neq 0$.  Then the operator $\mathcal R^{-\beta}$ is given by
\begin{align*}
     \mathcal R^{-\beta} f (r, \omega)=\sum_{k=0}^\infty \int_0^\infty K_k^{-\beta} (r,s)\proj_k f(s, \omega)\frac{\dd s}{s},
\end{align*}
where
\begin{align}\label{Kernel function}
    &K_k^{-\beta} (r,s) \\\nonumber
    =&H_{4, 4}^{2, 2}\left[k, \beta;\,  \frac{r}{s}
\, \middle | \hspace{-1.1em}
\begin{array}{lllll}
    &(-\frac{\mu_k+{\lambda_0}}{2}, \frac{1}{2}), \hspace{-1em}
    &(-\frac{\nu_k+{\lambda_0}+\beta}{2}, \frac{1}{2}), \hspace{-1em}
    &(\frac{\mu_k-{\lambda_0}}{2}, \frac{1}{2}),
    &(\frac{\nu_k-{\lambda_0}-\beta}{2}, \frac{1}{2})\\
    &(\frac{\nu_k-{\lambda_0}}{2}, \frac{1}{2}),
    &(\frac{\mu_k-{\lambda_0}-\beta}{2}, \frac{1}{2}),
    &(-\frac{\nu_k+{\lambda_0}}{2}, \frac{1}{2}),\hspace{-1em}
    &(-\frac{\mu_k+{\lambda_0}+\beta}{2}, \frac{1}{2})
\end{array}
\right].\end{align}
In addition, then
the parameters defined in \eqref{parameters a},  \eqref{parameters tau}, \eqref{parameters varrpho},  \eqref{parameters delta} are
\begin{align}
    a_k^*=0, \  \Lambda_k=0, \ \delta_k=1, \ \varrho_k=0.
\end{align}

\end{lemma}

\subsection{The power series representation of the $K_k^\alpha(r,s)$ when $r\neq s$}
\subsection*{Case 1: When $\alpha$ is not an even integer.} Using  Lemma \ref{series expansion 1} and Lemma \ref{series expansion 2} and the formula
$$\Gamma(z) \Gamma(1-z)=\frac{\pi}{\sin(\pi z)},$$
 the kernel functions $K_k^\alpha(r,s)$ have the following power series expansions,
\begin{align} \label{power series for R alpha}
& H_{4, 4}^{2, 2}\Big(k, \alpha; \frac{r}{s}\Big)
 \\\nonumber =&
   \begin{cases}
    \sum \limits_{n=0}^\infty  h_{k, 1n}^\alpha \left(\frac{s}{r}\right)^{\nu_k+{\lambda_0}+2+2n}+\sum\limits_{n=0}^\infty  h_{k, 2 n} ^\alpha\left(\frac{s}{r}\right)^{\mu_k+{\lambda_0}+\alpha+2+2n}, & 0<\frac{s}{r}<1;\\
    \sum \limits_{n=0}^\infty  h_{k, 1n}^{\alpha*}\left(\frac{r}{s}\right)^{\mu_k-{\lambda_0}+2n}+\sum\limits_{n=0}^\infty h_{k,2 n}^{\alpha* }\left(\frac{r}{s}\right)^{\nu_k-{\lambda_0}-\alpha+2n}, & 0<\frac{r}{s}<1,
   \end{cases}
\end{align}
with the coefficients satisfying
\begin{align}
    h_{k, 1n}^\alpha =h_{k,2 n}^{\alpha* }, \quad h_{k, 2 n} ^\alpha= h_{k, 1n}^{\alpha*}.
\end{align}
We can compute the coefficients
\begin{align*}
    h_{k, 1n}^\alpha&=\lim_{z\to a_{k,1n}^{\alpha}} [-(z-a_{k,1n})\mathcal{H}_{4, 4}^{2, 2}(k, \alpha; z)]\\
   &=\frac{2 \sin( \pi \alpha/2) \sin(\pi (\nu_k-\mu_k)/2)}{\pi \sin(\pi (\nu_k-\mu_k-\alpha)/2)}\frac{\prod \limits_{j=1}^4 \Gamma(b_{k,i}+(1-a_{k,1}+n))}{\prod_{j=2}^4\Gamma(a_{k,j}+(1-a_{k,1}+n)) \Gamma(n+1)}\\
   &\coloneq C_k^\alpha A_{k,1n}^{\alpha},
\end{align*}
where
\begin{align}
   C_k^\alpha &= \frac{2 \sin( \pi \alpha/2) \sin( \pi (\mu_k-\nu_k)/2)}{\pi \sin( \pi (\mu_k-\nu_k+\alpha)/2)},\\
 A_{k,1n}^{\alpha}&=  \frac{\Gamma(\frac{\mu_k+\nu_k}{2}+n+1)\Gamma(\nu_k-\frac{\alpha}{2}+n+1)\Gamma(\frac{\nu_k-\mu_k}{2}+n+1)\Gamma(-\frac{\alpha}{2}+n+1)}
 {\Gamma(\frac{\nu_k-\mu_k-\alpha}{2}+n+1)\Gamma(\nu_k+n+1)\Gamma(\frac{\nu_k+\mu_k-\alpha}{2}+n+1)\Gamma(n+1)};
\end{align}
Similarly,
\begin{align*}
    h_{k,2 n}^{\alpha }&=\lim_{z\to a_{k, 2n}^\alpha} [-(z-a_{k, 2n}^\alpha)\mathcal{H}_{4, 4}^{2, 2}(k, \alpha; z)]\\
   &=\frac{-2 \sin( \pi \alpha/2) \sin( \pi (\mu_k-\nu_k)/2)}{\pi \sin( \pi (\mu_k-\nu_k+\alpha)/2)}\frac{\prod \limits_{j=1}^4 \Gamma(b_{k,i}+(1-a_{k,2}+n))}{\prod \limits_{j=1, 3,4 }\Gamma(a_{k,j}+(1-a_{k,2}+n) \Gamma(n+1)}\\
   &=-C_k^\alpha A_{k,2n}^\alpha,
\end{align*}
with
\begin{align}
 A_{k,2n}^\alpha=  \frac{\Gamma( \mu_k+\frac{\alpha}{2}+n+1)\Gamma(\frac{\nu_k+\mu_k}{2}+n+1)\Gamma(\frac{\alpha}{2}+n+1)\Gamma(\frac{\mu_k-\nu_k}{2}+n+1)}
 {\Gamma(\frac{\mu_k-\nu_k+\alpha}{2}+n+1)\Gamma(\frac{\nu_k+\mu_k+\alpha}{2}+n+1)\Gamma(\mu_k+n+1)\Gamma(n+1)}.
\end{align}
For the coefficients $ h_{k, 1n}^{\alpha*}, h_{k,2 n}^{\alpha*}$, a direct computation gives that
    \begin{align*}
    h_{k, 1n}^{\alpha*}&=\lim_{z\to b_{k,1n}^\alpha} [(z-b_{k, 1n}^\alpha)\mathcal{H}_{4, 4}^{2, 2}(k, \alpha; z)]\\
   &=\frac{-2 \sin( \pi \alpha/2) \sin(\pi (\nu_k-\mu_k)/2)}{\pi \sin(\pi (\nu_k-\mu_k-\alpha)/2)}\frac{\prod \limits_{j=1}^4 \Gamma(b_{k,1}+(1-a_{k,j}+n))}{\prod \limits_{j=2}^4\Gamma(1-b_{k,j}+b_1+n) \Gamma(n+1)}\\
   &=\tfrac{-2 \sin( \pi \alpha/2) \sin(\pi (\nu_k-\mu_k)/2)}{\pi \sin(\pi (\nu_k-\mu_k-\alpha)/2)}\tfrac{\Gamma(\frac{\nu_k+\mu_k}{2}+n+1)\Gamma( \mu_k+\frac{\alpha}{2}+n+1)\Gamma(\frac{\mu_k-\nu_k}{2}+n+1)\Gamma(\frac{\alpha}{2}+n+1)}{\Gamma(\frac{\mu_k-\nu_k+\alpha}{2}+n+1)
   \Gamma(\mu_k+n+1)\Gamma(\frac{\nu_k+\mu+\alpha}{2}+n+1)\Gamma(n+1)}\\
   &=-C_{k}^\alpha A_{k, 2n}^\alpha \\
   &=h_{k,2 n}^{\alpha };
\end{align*}
and
 \begin{align*}
    h_{k,2 n}^{\alpha*}&=\lim_{z\to b_{k, 2n}^\alpha} [(z-b_{k, 2n}^\alpha)\mathcal{H}_{4, 4}^{2, 2}(k, \alpha; z)]\\
   &=\frac{2 \sin( \pi \alpha/2) \sin( \pi (\mu_k-\nu_k)/2)}{\pi \sin( \pi (\mu_k-\nu_k+\alpha)/2)}\frac{\prod \limits_{j=1}^4 \Gamma(b_{k,2}+(1-a_{k,j}+n))}{\prod \limits_{j=1,3, 4}\Gamma(b_{k,2}+(1-b_{k,j}+n) \Gamma(n+1)}\\
   &=\tfrac{2 \sin( \pi \alpha/2) \sin( \pi (\mu_k-\nu_k)/2)}{\pi \sin( \pi (\mu_k-\nu_k+\alpha)/2)}\tfrac{\Gamma(\nu_k-\frac{\alpha}{2}+n+1)\Gamma(\frac{\mu_k+\nu_k}{2}+n+1)
   \Gamma(-\frac{\alpha}{2}+n+1)\Gamma(\frac{\nu_k-\mu_k}{2}+n+1)}{\Gamma(\frac{\nu_k-\mu_k-\alpha}{2}+n+1)
   \Gamma(\frac{\nu_k+\mu_k-\alpha}{2}+n+1)\Gamma(\nu_k+n+1)\Gamma(n+1)}\\
   &=C_k^\alpha A_{k,1n}^\alpha\\
   &=h_{k, 1n}^\alpha.
\end{align*}

\subsection*{Case 2: When $\alpha=2m$ is a positive even integer.} Using the fact that $\frac{1}{\Gamma(-n)}=0$ for nonnegative integer $n$, a direct computation gives for all $n \geq 0$,
\begin{align*}
  h_{k,2n}^\alpha &=  h_{k, 1n}^{\alpha*}\\
   & = \frac{2(-1)^n }{n!} \frac{\Gamma(\frac{\nu_k-\mu_k-\alpha}{2}-n)\Gamma(\frac{\nu_k+\mu_k}{2}+n+1)\Gamma(\mu_k+\frac{\alpha}{2}+n+1)}
   {\Gamma(\frac{\nu_k-\mu_k}{2}-n)\Gamma(-\frac{\alpha}{2}-n)\Gamma(\mu_k+n+1) \Gamma(\frac{\mu_k+\nu_k+\alpha}{2}+n+1)}\\
   &=0 ;
\end{align*}
and
\begin{align*}
  h_{k,1n}^\alpha & = h_{k,2 n}^{\alpha*} \\
   & = \frac{2(-1)^n }{n!} \frac{\Gamma(\frac{\mu_k-\nu_k+\alpha}{2}-n)\Gamma(\nu_k-\frac{\alpha}{2}+n+1)\Gamma(\frac{\nu_k+\mu_k}{2}+n+1)}
   {\Gamma(\frac{\mu_k-\nu_k}{2}-n)\Gamma(\frac{\alpha}{2}-n) \Gamma(\frac{\mu_k+\nu_k-\alpha}{2}+n+1)\Gamma(\nu_k+n+1)}.
\end{align*}
Hence, when $n\ge m$, $h^\alpha_{k,1n}=0$, and when $0\le n < m$, $h^\alpha_{k,1n}\neq 0$.
Thus, in this case, the kernel functions are given by
\begin{align}
   K_{k}^\alpha(r,s) = H_{4, 4}^{2, 2}(k, \alpha; r, s)=
   \begin{cases}
    \sum\limits_{n=0}^{m } h_{k, 1n}^\alpha \left(\frac{s}{r}\right)^{\nu_k+{\lambda_0}+2+2n}, \quad z=\frac{s}{r}<1;\\
    \sum\limits_{n=0}^{m} h_{k, 1n}^{\alpha}\left(\frac{r}{s}\right)^{\nu_k-{\lambda_0}-\alpha+2n}, \quad z=\frac{r}{s}<1.
   \end{cases}
\end{align}
In this case, the kernels have been reduced to a finite sum of polynomials of either $\frac{s}{r}\chi_{\{0<\frac{s}{r}<1\}}$ or $\frac{r}{s}\chi_{\{0<\frac{r}{s}<1\}}$.  In particular,
   $ K^{\alpha}_{k}(r,s) $ are uniformly bounded for all $r, s$.
 Furthermore, it defines an $L^p$-bounded operator on the radial part as long as
   \begin{align}
   \lambda_0-\nu_0+\alpha <\frac{d}{p}<\lambda_0 +\nu_0+2,
   \end{align}
which can also be written
\begin{align}
    \frac{\sigma+\alpha}{d}<\frac{1}{p}<\frac{d-\sigma}{d}.
\end{align}
 \subsection*{Case 3: When $\alpha= -2m$ is a negative even integer.}
In this case,
 \begin{align*}
  h_{k,2n}^\alpha & = h_{k,1 n}^{\alpha*} \\
   & = \frac{2(-1)^n }{n!} \frac{\Gamma(\frac{\mu_k-\nu_k+\alpha}{2}-n)\Gamma(\nu_k-\frac{\alpha}{2}+n+1)\Gamma(\frac{\nu_k+\mu_k}{2}+n+1)}
   {\Gamma(\frac{\mu_k-\nu_k}{2}-n)\Gamma(-\frac{\alpha}{2}-n) \Gamma(\frac{\mu_k+\nu_k-\alpha}{2}+n+1)\Gamma(\nu_k+n+1)} .
\end{align*}
Hence, when $n\ge m$, $h^\alpha_{k,2n}=0$, and when $0\le n < m$, $h^\alpha_{k,2n}\neq 0$. For $h_{k,1n}^\alpha$,
\begin{align*}
  h_{k,1n}^\alpha &=  h_{k, 2n}^{\alpha*}\\
   & = \frac{(-1)^n 2}{n!} \frac{\Gamma(\frac{\nu_k-\mu_k-\alpha}{2}-n)\Gamma(\frac{\nu_k+\mu_k}{2}+n+1)\Gamma(\mu_k+\frac{\alpha}{2}+n+1)}
   {\Gamma(\frac{\nu_k-\mu_k}{2}-n)\Gamma(\frac{\alpha}{2}-n)\Gamma(\mu_k+n+1) \Gamma(\frac{\mu_k+\nu_k+\alpha}{2}+n+1)}\\
   &=0.
\end{align*}

Thus, the kernel function in this case is given by
\begin{align}
   K_{k}^\alpha(r,s) = H_{4, 4}^{2, 2}(k, \alpha; r, s)=
   \begin{cases}
    \sum\limits_{n=0}^{m-1} h_{k, 2 n}^\alpha \left(\frac{s}{r}\right)^{\mu_k+{\lambda_0}+\alpha+2+2n}, \quad &z=\frac{s}{r}<1;\\
    \sum\limits_{n=0}^{m-1} h_{k, 2 n}^{\alpha}\left(\frac{r}{s}\right)^{\mu_k-{\lambda_0}+2n}, \quad &z=\frac{r}{s}<1.
   \end{cases}
\end{align}
Similarly to the previous case, the kernels have thus been reduced to a finite sum of polynomials of either $\frac{s}{r}\chi_{\{0<\frac{s}{r}<1\}}$ or $\frac{r}{s}\chi_{\{0<\frac{r}{s}<1\}}$.  This implies that
   $ K^{\alpha}_{k}(r,s) $ are uniformly bounded for all $r, s$.
   To guarantee the $L^p$-boundedness, we need the restriction
   \begin{align}
 \lambda_0-\mu_0<\frac{d}{p}< \mu_0+{\lambda_0}+\alpha+2,
   \end{align}
which is equivalent to
\begin{align}
    0<\frac{1}{p}<\frac{d+\alpha}{d}.
\end{align}

Using the asymptotic expansion for Gamma functions \eqref{asymptotic expansion for the quotient of gamma function}, we have the following lemma.

\begin{lemma} \label{ub for Riesz transform}The coefficients $h_{k,1n}^\alpha, h_{k,2n}^\alpha, C_k^\alpha, A_{k, 1n}^\alpha, A_{k, 2n}^\alpha$ are uniformly bounded in $k$ and $n$,
\begin{align}
    |h_{k,1n}^\alpha|,\ |h_{k,2n}^\alpha|,\ |C_k^\alpha|,\ |A_{k, 1n}^\alpha|,\ |A_{k, 2n}^\alpha| \lesssim 1.
\end{align}
For $C_k^\alpha$, we have
\begin{align}
    \frac{d^N C_k^\alpha}{dk^N} \lesssim_N   \frac{1}{k^{N+1}} \   \text{for all} \quad N\geq 0.
\end{align}
For each fixed $k$ and as $n\to\infty$, $A^\alpha_{k,1n}, A^\alpha_{k,2n}$ have the asymptotic behavior
 \begin{align*}
   A_{k, 1n}^\alpha &=1+O_{\alpha, k}\left(\frac{1}{(n+1)^2}\right),\\
 A_{k,2 n}^{\alpha } &=1+O_{\alpha, k}\left(\frac{1}{(n+1)^2}\right).
 \end{align*}
 By $O_{\alpha,k}(\cdot)$, we mean that the implicit $O(\cdot)$ constant depends on $\alpha$ and $k$.

\end{lemma}
\subsection{Singularity at $r=s$ of $\mathcal R^{\alpha}$.}
To find the singularity at $r=s$, we notice that the Mellin multipliers $\mathcal H_{4,4}^{2,2}(k, \alpha; z)$ are uniformly bounded in $k, n$, we conjecture that the singularity of  $H_{4,4}^{2,2}(k, \alpha;  r/s )$ is $\frac{1}{r-s}$.  In the following lemma, we give the proof.

\begin{lemma}\label{sin at $r=s$.} As $r\to s$, we have
\begin{align}
   K_{k}^\alpha(r,s) \sim
   \begin{cases}
     \frac{C_k^\alpha}{1-(\frac{s}{r})^2}\left\{\left( \frac{s}{r}\right)^{\nu_k+{\lambda_0}+2}-\left( \frac{s}{r}\right)^{\mu_k+{\lambda_0}+\alpha+2}\right\}, \ \text{as} \ \frac{s}{r} \to 1^{-};\\
      \frac{C_k^\alpha}{1-(\frac{r}{s})^2} \left\{\left( \frac{r}{s}\right)^{\mu_k-{\lambda_0}} -\left( \frac{r}{s}\right)^{\nu_k-{\lambda_0}-\alpha}\right\}, \  \text{as} \  \frac{r}{s} \to 1^{-}.
   \end{cases}
\end{align}
\end{lemma}

\begin{proof} For a Gauss series $ G(z)=\sum\limits_{n=0}^\infty a_n z^n$,
assume that the convergence radius is $1$, and $\lim\limits_{n\to \infty} a_n=a$, then by Abel's theorem, we have
\begin{align*}
    \lim_{z\to 1^{-}}(1-z) G(z)&= \lim_{z\to 1^{-}}\sum_{n=0}^\infty(a_{n+1}-a_n)z^{n+1}+a_0 \\
    &=\sum_{n=0}^\infty(a_{n+1}-a_n)+a_0=a.
\end{align*}
From this result and the fact that
\begin{align}
    \lim_{n\to \infty} A_{k,1n}^\alpha=1, \quad \lim_{n\to \infty} A_{k,2n}^\alpha=1,
\end{align}
the lemma follows.
\end{proof}
This lemma implies that  as $r$ close to $s$, $K^{\alpha}_{k}(r,s)$ has a singularity of type
\begin{align}
    H_{4, 4}^{2, 2}(k, \alpha;  r, s) \sim (r-s)^{-1} \quad ( r \to s).
\end{align}

\subsection{Proof of Proposition \ref{bdd for Riesz transform}} Here we only give the details for $\mathcal R^{\alpha}$, as the proof for $\mathcal R^{-\beta}$ is similar. We may also assume that $\alpha\neq 2m$ where $m$ is an integer; when $\alpha=2m$, the details are much simpler.
 To obtain the $L^p$-boundedness, we use the power series representation \eqref{power series for R alpha} and the discrete multiplier theorem (Theorem \ref{multiplier theorem for spherical harmonic decomposition}) for decompositions into spherical harmonics.
Similarly to the modified kernels in Section \ref{s:proof-of-lp-bddness}, we define
\begin{align*}
    \widetilde{K}^{\alpha}_k(r,s)\coloneq \left(\frac{s}{r}\right)^{-\frac{d}{p}}K^{\alpha}_k (r,s),
\end{align*}
and the modified operator
\begin{align*}
   \widetilde{\mathcal R}^{\alpha} f(r, \omega )&\coloneq \sum_{k=0}^\infty \int_0^\infty
   \widetilde{K}^{\alpha}_k(r,s) \mathbb {P}_k f(s, \omega) \frac{\dd s}{s}.
\end{align*}
Then $\mathcal R^\alpha $ is bounded in $L^p(\mathbb R^d)$ if and only if
$\widetilde{\mathcal R}^{\alpha}$ is bounded in $L^p(\frac{\dd r}{r} \cdot \dd \omega)$. Due to the singularity at $r=s$ of the kernel functions,  we divide the operator into three parts
\begin{align*}
  \widetilde{\mathcal R}^{\alpha} f(r, \omega ) &=\sum_{k=0}^\infty \int_0^{\frac{r}{2}}\widetilde{K}^{\alpha}_k(r,s) \mathbb {P}_k f(s, \omega) \frac{\dd s}{s}+\sum_{k=0}^\infty \int_{\frac{r}{2}}^{2r}\widetilde{K}^{\alpha}_k(r,s) \mathbb {P}_k f(s, \omega) \frac{\dd s}{s}\\
   &\quad +\sum_{k=0}^\infty \int_{2r}^{\infty}\widetilde{K}^{\alpha}_k(r,s) \mathbb {P}_k f(s, \omega) \frac{\dd s}{s}\\
   &=: \widetilde{\mathcal R}^{\alpha}_{1} f(r, \omega)+\widetilde{\mathcal R}^{\alpha}_{2} f(r, \omega)+\widetilde{\mathcal R}^{\alpha}_{3} f(r, \omega).
\end{align*}
We will prove that $ \widetilde{\mathcal R}^{\alpha}_{i} $ is bounded in $L^p(\frac{\dd r}{r} \cdot \dd \omega)$ for $i=1,2, 3$ and for suitable $p$.
\subsection*{Case 1: $0<\frac{s}{r}<\frac{1}{2}$. } By Lemma \ref{ub for Riesz transform}, $h_{k, 1n}^\alpha, h_{k,2 n}^{\alpha }$ are uniformly bounded in $k$ and $n$. It follows that
\begin{align*}
   | \widetilde K^\alpha_{k,1}(r,s)|&\lesssim \left(\frac{s}{r}\right)^{\nu_k+{\lambda_0} +2-\frac{d}{p}} + \left(\frac{s}{r}\right)^{\mu_k+\alpha+\lambda_0+2-\frac{d}{p}}\\
   &\lesssim  \left(\frac{s}{r}\right)^{\nu_0+\lambda_0+2-\frac{d}{p}} + \left(\frac{s}{r}\right)^{\mu_0+\alpha+\lambda_0+2-\frac{d}{p}}.
\end{align*}
For the finite difference of $\widetilde K^\alpha_{k,1}(r,s)$ with $1\leq N\leq \lfloor \frac{d-1}{2}\rfloor+1 $, we have
\begin{align*}
  &\quad\sup_{j} 2^{j(N-1)}\sum_{k=2^j}^{2^{j+1}} |\fd ^N \widetilde K^\alpha_{k,1}(r,s)|\\
  &\leq 2^N \sum_{k=0}^\infty k^{N-1}|\widetilde K^\alpha_{k,1}(r,s)| \\
  &\lesssim \ 2^N \sum_{k=0}^\infty k^{N-1}\left[\left(\frac{s}{r}\right)^{\nu_k+{\lambda_0} +2-\frac{d}{p}} + \left(\frac{s}{r}\right)^{\mu_k+\alpha+\lambda_0+2-\frac{d}{p}}\right]\\
  &\lesssim 2^N \left[\left(\frac{s}{r}\right)^{\nu_0+\lambda_0+2-\frac{d}{p}} + \left(\frac{s}{r}\right)^{\mu_0+\alpha+\lambda_0+2-\frac{d}{p}}\right].
\end{align*}
 For the angular part, using the multiplier theorem, Theorem \ref{multiplier theorem for spherical harmonic decomposition},
\begin{align*}
	&\quad \left \|\sum_{k=0}^\infty  \widetilde K_{k,1}^\alpha (r,s) \proj_{k }f(s, \omega) \right \|_{L^{p}(\dd \omega)}\\
	& \lesssim   \left[\left(\frac{s}{r}\right)^{\nu_0+\lambda_0+2-\frac{d}{p}} + \left(\frac{s}{r}\right)^{\mu_0+\alpha+\lambda_0+2-\frac{d}{p}}\right]\left \|\sum_{k=0}^\infty   \proj_{k}f(s, \omega) \right \|_{L^{p}(\dd \omega)}.
\end{align*}
To estimate the radial part,  if
 \begin{equation}
     \nu_0+\lambda_0+2-\frac{d}{p}>0  \ \text{and} \ \mu_0+\alpha+\lambda_0+2-\frac{d}{p}>0
 \end{equation}
 then $\left[r^{\nu_0+\lambda_0+2-\frac{d}{p}}+r^{\mu_0+\alpha+\lambda_0+2-\frac{d}{p}}\right]\chi_{\{0<r<1/2\}} \in L^{1}\left(\frac{\dd r}{r}\right)$.
 Thus,
  \begin{align*}\label{R 1}
	\|\widetilde{ \mathcal R}_{1}^{\alpha} f(r, \omega )\|_{ L^p (\frac{\dd r}{r} \cdot \dd \omega)}
	&\lesssim  \left\|\int_0^{r/2}  \left(\frac{s}{r}\right)^{\nu_0+\lambda_0+2-\frac{d}{p}} \left \|\sum_{k=0}^\infty   \proj_{k }f(s, \omega) \right \|_{L^{p}(\dd \omega)} \frac{\dd s}{s} \right \|_{L^p (\frac{\dd r}{r})}\nonumber\\
	&\quad + \left\|\int_0^{r/2}  \left(\frac{s}{r}\right)^{\mu_0+\alpha+\lambda_0+2-\frac{d}{p}} \left \|\sum_{k=0}^\infty   \proj_{k }f(s, \omega) \right \|_{L^{p}(\dd \omega)} \frac{\dd s}{s} \right \|_{L^p (\frac{\dd r}{r})}\nonumber\\
	&\lesssim \left \|\sum_{k=0}^\infty   \proj_{k  }f(r, \omega) \right \|_{L^p(\frac{\dd r}{r} \cdot \dd \omega)} =\| f(r,\omega)\|_{L^{p}(\frac{\dd r}{r} \cdot \dd \omega)}.
	\end{align*}
That is, if $\frac{1}{p}<\min\{ \frac{d+\alpha}{d}, \frac{d-\sigma}{d}\}$, then  $\widetilde {\mathcal R}_{1}^{\alpha} $ is bounded in  $ L^p (\frac{\dd r}{r} \cdot \dd \omega)$.

\subsection {Case 2:  $\frac{s}{r}>2$. }
Since $h_{k, 1n}^\alpha, h_{k,2 n}^{\alpha }$ are uniformly bounded in $k, n$, we have
\begin{align*}
   | \widetilde K^\alpha_{k,3}(r,s)|&\lesssim \left(\frac{r}{s}\right)^{\nu_k-{\lambda_0} -\alpha+\frac{d}{p}} + \left(\frac{r}{s}\right)^{\mu_k-{\lambda_0}+\frac{d}{p}}\\
   &\lesssim  \left(\frac{r}{s}\right)^{\nu_0-\lambda_0-\alpha+\frac{d}{p}} + \left(\frac{r}{s}\right)^{\mu_0-\lambda_0+\frac{d}{p}}.
\end{align*}
For the finite difference of $\widetilde K^\alpha_{k,3}(r,s)$, we have
\begin{align*}
  &\quad\sup_{j} 2^{j(N-1)}\sum_{k=2^j}^{2^{j+1}} |\fd ^N  \widetilde K^\alpha_{k,3}(r,s)|\\
  &\leq 2^N \sum_{k=0}^\infty k^{N-1}|\widetilde  K^\alpha_{k,3}(r,s)| \\
  &\lesssim \ 2^N \sum_{k=0}^\infty k^{N-1}\left[\left(\frac{r}{s}\right)^{\nu_k-{\lambda_0} -\alpha+\frac{d}{p}} + \left(\frac{r}{s}\right)^{\mu_k-{\lambda_0}+\frac{d}{p}}\right]\\
  &\lesssim 2^N \left[\left(\frac{r}{s}\right)^{\nu_0-\lambda_0-\alpha+\frac{d}{p}} + \left(\frac{r}{s}\right)^{\mu_0-\lambda_0+\frac{d}{p}}\right].
\end{align*}
By Theorem \ref{multiplier theorem for spherical harmonic decomposition}, we have
\begin{align*}
	&\quad \left \|\sum_{k=0}^\infty  \widetilde K_{k,3}^\alpha (r,s) \proj_{k }f(s, \omega) \right \|_{L^{p}(\dd \omega)}\\
	& \lesssim   \left[\left(\frac{r}{s}\right)^{\nu_0-\lambda_0-\alpha+\frac{d}{p}} + \left(\frac{r}{s}\right)^{\mu_0-\lambda_0+\frac{d}{p}}\right]\left \|\sum_{k=0}^\infty   \proj_{k}f(s, \omega) \right \|_{L^{p}(\dd \omega)}.
\end{align*}
 If
\begin{align}
    \nu_0-\lambda_0-\alpha+\frac{d}{p}>0, \ \text{and} \ \mu_0-\lambda_0+\frac{d}{p}>0
\end{align} then $r^{\nu_0-\lambda_0-\alpha+\frac{d}{p}} \chi_{\{0<r<\frac{1}{2}\}}+r^{\mu_0-\lambda_0+\frac{d}{p}}\chi_{\{0<r<\frac{1}{2}\}} \in L^{1}\left(\frac{\dd r}{r}\right)$.
 Thus,
\begin{align*}
	\|\widetilde {\mathcal R}_{1}^{\alpha} f(r, \omega )\|_{ L^p (\frac{\dd r}{r} \cdot \dd \omega)}
	&\lesssim  \left\|\int_{2r}^\infty  \left(\frac{r}{s}\right)^{\nu_0-\lambda_0-\alpha+\frac{d}{p}} \left \|\sum_{k=0}^\infty   \proj_{k }f(s, \omega) \right \|_{L^{p}(\dd \omega)} \frac{\dd s}{s} \right \|_{L^p (\frac{\dd r}{r})}\nonumber\\
	&\quad + \left\|\int_{2r}^\infty   \left(\frac{r}{s}\right)^{\mu_0-\lambda_0+2+ \frac{d}{p}} \left \|\sum_{k=0}^\infty   \proj_{k }f(s, \omega) \right \|_{L^{p}(\dd \omega)} \frac{\dd s}{s} \right \|_{L^p (\frac{\dd r}{r})}\nonumber\\
	&\lesssim \left \|\sum_{k=0}^\infty   \proj_{k  }f(r, \omega) \right \|_{L^p(\frac{\dd r}{r} \cdot \dd \omega)} =\| f(r,\omega)\|_{L^{p}(\frac{\dd r}{r} \cdot \dd \omega)}.
	\end{align*}
That is, if $\frac{1}{p}>\max\{0, \frac{\alpha+\sigma}{d}\}$, then  $\widetilde {\mathcal R}_{3}^{\alpha} $ is bounded in  $ L^p  (\frac{\dd r}{r} \cdot \dd \omega)$.

\subsection {Case 3: $\frac{1}{2}<\frac{s}{r}<2$.}
Let us define
\begin{align}
    E_{k, 1n}^\alpha\coloneq A_{k, 1n}^\alpha-1;\\
    E_{k, 2n}^\alpha\coloneq A_{k, 2n}^\alpha-1.
\end{align}
By definition, for $N\geq 1$, we have
\begin{align}
    \frac{d^NE_{k, 1n}^\alpha }{dk^N}= \frac{d^N A_{k, 1n}^\alpha }{dk^N},
    \qquad
         \frac{d^NE_{k, 2n}^\alpha }{dk^N}= \frac{d^N A_{k, 2n}^\alpha }{dk^N}.
\end{align}
 We rewrite the kernel functions as
\begin{align*}
    K_{k, 2}^\alpha (r, s)=& C_k^\alpha
    \sum \limits_{n=0}^\infty A_{k, 1n}^\alpha  \left(\frac{s}{r}\right)^{\nu_k+{\lambda_0}+2+2n}\chi_{+} -C_k^\alpha\sum\limits_{n=0}^\infty A_{k,2 n}^{\alpha } \left(\frac{s}{r}\right)^{\mu_k+{\lambda_0}+\alpha+2+2n}\chi_{+}\\
    & -C_k^\alpha \sum \limits_{n=0}^\infty A_{k,2 n}^{\alpha }  \left(\frac{r}{s}\right)^{\mu_k-{\lambda_0}+2n}\chi_{-}+ C_k^\alpha\sum\limits_{n=0}^\infty A_{k, 1n}^\alpha \left(\frac{r}{s}\right)^{\nu_k-{\lambda_0}-\alpha+2n}\chi_{-}\\
   =:&  K_{k, 2}^{\alpha, \text{main} } (r, s) +K_{k, 2}^{\alpha, \textup{rem}} (r, s).
    \end{align*}
    with
\begin{align*}
    K_{k, 2}^{\alpha, \text{main} } (r, s)&\coloneq C_k^\alpha\frac{\left\{\left(\frac{s}{r}\right)^{\nu_k+{\lambda_0}+2}- \left(\frac{s}{r}\right)^{\mu_k+{\lambda_0}+\alpha+2}\right\}}{1-\left(\frac{s}{r}\right)^2}\chi_{+}
    -C_k^\alpha\frac{\left\{\left(\frac{r}{s}\right)^{\mu_k-{\lambda_0}}- \left(\frac{r}{s}\right)^{\nu_k-{\lambda_0}-\alpha}\right\}} {1-\left(\frac{r}{s}\right)^2}\chi_{-};\\
 K_{k, 2}^{\alpha, \textup{rem}} (r, s)&\coloneq {C_k^\alpha\left(\frac{s}{r}\right)^{\nu_k+{\lambda_0}+2} \sum \limits_{n=0}^\infty E_{k,1n}^{\alpha}\left(\frac{s}{r}\right)^{2n}\chi_{+}-C_k^\alpha\left(\frac{s}{r}\right)^{\mu_k+{\lambda_0}+\alpha+2}\sum\limits_{n=0}^\infty  E_{k,2n}^{\alpha}\left(\frac{s}{r}\right)^{2n}}\chi_{+}\\
  &\quad-{C_k^\alpha\left(\frac{r}{s}\right)^{\mu_k-{\lambda_0}} \sum \limits_{n=0}^\infty E_{k,2n}^{\alpha}\left(\frac{r}{s}\right)^{2n}\chi_{-}+C_k^\alpha\left(\frac{r}{s}\right)^{\nu_k-{\lambda_0}-\alpha}\sum\limits_{n=0}^\infty  E_{k,1n}^{\alpha}\left(\frac{r}{s}\right)^{2n}}\chi_{-}.
\end{align*}
      We denote the modified kernels
    \begin{align*}
       \widetilde K_{k, 2}^{\alpha, \text{main} }(r,s)& \coloneq  \left(\frac{s}{r}\right)^{-\frac{d}{p}} K_{k, 2}^{\alpha, \text{main} }(r,s), \\
         \widetilde K_{k, 2}^{\alpha, \textup{rem} }(r,s)&\coloneq  \left(\frac{s}{r}\right)^{-\frac{d}{p}} K_{k, 2}^{\alpha, \textup{rem} }(r,s).
    \end{align*}
    and  the associated operators are
\begin{align*}
      \widetilde{\mathcal R}_2^{\alpha, \text{main}} f(r, \omega )&\coloneq \sum_{k=0}^\infty \int_0^\infty
   \widetilde{K}^{\alpha, \text{main}}_k(r,s) \mathbb {P}_k f(s, \omega) \frac{\dd s}{s}; \\
   \widetilde{\mathcal R}_2^{\alpha, \textup{rem}} f(r, \omega )&\coloneq \sum_{k=0}^\infty \int_0^\infty
   \widetilde{K}^{\alpha, \textup{rem}}_k(r,s) \mathbb {P}_k f(s, \omega) \frac{\dd s}{s}.
\end{align*}

\subsubsection{The estimate of the term $\widetilde{\mathcal R}_2^{\alpha, \text{main}}$}
Let us define the approximated kernels
\begin{align*}
  \widetilde{K}^{\alpha, \text{main, ap}}_{k,2}(r,s)&= C_k^\alpha\frac{\left\{\left(\frac{s}{r}\right)^{\nu_0+\lambda_0+2-\frac{d}{p}}-\left(\frac{s}{r}\right)^{\mu_0+\lambda_0+\alpha+2-\frac{d}{p}}\right\}}{1-\left(\frac{s}{r}\right)^2}\chi_{+}\\
  &\quad -C_k^\alpha
  \frac{\left\{\left(\frac{r}{s}\right)^{\mu_0-\lambda+\frac{d}{p}}- \left(\frac{r}{s}\right)^{\nu_0-\lambda-\alpha+\frac{d}{p}}\right\}} {1-\left(\frac{r}{s}\right)^2}\chi_{-},
\end{align*}

and
\begin{align*}
      &\widetilde K_{\textup{err},k}^{\alpha, \text{main}}(r,s)=  \widetilde{K}^{\alpha, \text{main}}_{k,2}(r,s)-  \widetilde{K}^{\alpha, \text{main, ap}}_{k,2}(r,s)\\
      &=\left[C_k^\alpha\left(\frac{s}{r}\right)^{\lambda_0+2-\frac{d}{p}}\left\{\left(\frac{s}{r}\right)^{\nu_k}-
      \left(\frac{s}{r}\right)^{\nu_0}\right\}-C_k^\alpha\left(\frac{s}{r}\right)^{\lambda_0+\alpha+2-\frac{d}{p}} \left\{\left(\frac{s}{r}\right)^{\mu_k}-\left(\frac{s}{r}\right)^{\mu_0}\right\}\right] \frac{1}{1-\left(\frac{s}{r}\right)^2}\chi_{+}\\
 &\quad -\left[C_k^\alpha\left(\frac{r}{s}\right)^{\lambda_0+\frac{d}{p}}\left\{\left(\frac{r}{s}\right)^{\mu_k}
 -\left(\frac{r}{s}\right)^{\mu_0}\right\}-C_k^\alpha\left(\frac{s}{r}\right)^{\lambda-\alpha+\frac{d}{p}}\left\{ \left(\frac{s}{r}\right)^{\nu_k}-\left(\frac{s}{r}\right)^{\nu_0} \right\}\right]\frac{1}{1-\left(\frac{r}{s}\right)^2}\chi_{-},
\end{align*}
and the associated operators
\begin{align*}
  \widetilde{\mathcal R}_2^{\alpha, \text{main, ap}} f(r, \omega )&=\sum_{k=0}^\infty \int_0^\infty
   \widetilde{K}^{\alpha, \text{main, ap}}_k(r,s) \mathbb {P}_k f(s, \omega) \frac{\dd s}{s};\\
  \widetilde{\mathcal R}_{\textup{err}} f(r, \omega )&=\sum_{k=0}^\infty \int_0^\infty
   \widetilde K_{\textup{err},k}^{\alpha, \text{main}}(r,s)  \mathbb {P}_k f(s, \omega) \frac{\dd s}{s}.
\end{align*}
For $ \widetilde{\mathcal R}_2^{\alpha, \text{main, ap}}$, we can rewrite it as
\begin{align*}
   & \widetilde{\mathcal R}_2^{\alpha, \text{main, ap}} f(r, \omega )\\
    =&\int_0^\infty\!\!  \left [\tfrac{\left\{\left(\frac{s}{r}\right)^{\nu_0+\lambda_0+2-\frac{d}{p}}- \left(\frac{s}{r}\right)^{\mu_0+\lambda_0+\alpha+2-\frac{d}{p}}\right\}\chi_{+}}{1-\left(\frac{s}{r}\right)^2}
    -\tfrac{\left\{\left(\frac{r}{s}\right)^{\mu_0-\lambda+\frac{d}{p}}- \left(\frac{r}{s}\right)^{\nu_0-\lambda-\alpha+\frac{d}{p}}\right\}\chi_{-}} {1-\left(\frac{r}{s}\right)^2} \right]\\
    &\times
  \sum_{k=0}^\infty C_k^\alpha \mathbb {P}_k f(s, \omega) \frac{\dd s}{s}.
\end{align*}
Using the boundedness of \CZ{} singular operators, we have
\begin{align*}
    \|\widetilde{\mathcal R}_2^{\alpha, \text{main, ap}} f(\cdot, \omega )\|_{L^p(\frac{\dd {r}}{r})} \lesssim \left\| \sum_{k=0}^\infty C_k^\alpha \mathbb {P}_k f(\cdot, \omega) \right\|_{L^p(\frac{\dd {r}}{r})}.
\end{align*}
For the coefficients $C_k^\alpha$, by Lemma \ref{ub for Riesz transform}, we have
\begin{align*}
    |C_k^\alpha| &\lesssim 1,\\
    \sup_{j}2^{j(N-1)} \sum_{k=2^j}^{2^{j+1}}|\fd ^N C_{k}^\alpha| &\lesssim 1.
\end{align*}
That is, $C_k^\alpha$ satisfies the conditions of \eqref{uniform bound} and \eqref{uniform difference bound}, therefore,
\begin{align}
   \|\widetilde{\mathcal R}_2^{\alpha, \text{main, ap}} f(r, \omega )\|_{L^p(\frac{\dd {r}}{r}\cdot \dd \omega)} \lesssim \Big\| \sum_{k=0}^\infty\mathbb {P}_k f\Big\|_{L^p(\frac{\dd {r}}{r}\cdot \dd \omega)} =\|f\|_{L^p(\frac{\dd {r}}{r}\cdot \dd \omega)}.
\end{align}

For the error term $\widetilde K_{\textup{err},k}^{\alpha, \text{main}}(r,s)$, using Lemma \ref{difference estimates for $s<r$} and Lemma \ref{difference estimates for $r<s<2r$},  the following estimates are valid:
\begin{align*}
    |\widetilde K_{\textup{err},k}^{\alpha, \text{main}}(r,s)| \lesssim 1; \\
    \sup_{j} 2^{j(N-1)}\sum_{k=2^j}^{2^{j+1}} |\fd ^N  \widetilde K_{\textup{err},k}^{\alpha, \text{main}}(r,s)| \lesssim 1.
\end{align*}
Hence,  $ \widetilde{\mathcal R}_{\textup{err}}$ is bounded in $L^p(\frac{\dd r}{r} \cdot  \dd \omega)$.
In conclusion, we have proved that $ \widetilde{\mathcal R}_2^{\alpha, \text{main}}$ is also bounded in $L^p(\frac{\dd r}{r} \cdot \dd \omega)$ for $1<p<\infty$.

\subsubsection{The estimate of the term  $\widetilde{\mathcal R}_2^{\alpha, \textup{rem}}$.} Recall that
\begin{align*}
   &\quad \widetilde{K}_{k, 2}^{\alpha, \textup{rem}} (r, s)\\
&= C_k^\alpha\left(\frac{s}{r}\right)^{\nu_k+{\lambda_0}+2-\frac{d}{p}}\underbrace{ \sum \limits_{n=0}^\infty E_{k,1n}^{\alpha}\left(\frac{s}{r}\right)^{2n} \chi_{+}}_{E_{k, 1}^{\alpha, +}(r, s)}-C_k^\alpha\left(\frac{s}{r}\right)^{\mu_k+{\lambda_0}+\alpha+2-\frac{d}{p}}\underbrace{\sum\limits_{n=0}^\infty  E_{k,2n}^{\alpha}\left(\frac{s}{r}\right)^{2n}\chi_{+}}_{E_{k, 2}^{\alpha, +}(r, s)}\\
  &\quad-C_k^\alpha\left(\frac{r}{s}\right)^{\nu_k-{\lambda_0}-\alpha+\frac{d}{p}}\underbrace{\sum\limits_{n=0}^\infty  E_{k,2n}^{\alpha}\left(\frac{r}{s}\right)^{2n} \chi_{-}}_{E_{k, 2}^{\alpha, -}(r, s)}+C_k^\alpha\left(\frac{r}{s}\right)^{\mu_k-{\lambda_0}+\frac{d}{p}} \underbrace{\sum \limits_{n=0}^\infty E_{k,1n}^{\alpha}\left(\frac{r}{s}\right)^{2n} \chi_{-}}_{_{E_{k, 1}^{\alpha, -}(r, s)}},\
\end{align*}
with
\begin{align}
 E_{k, 1}^{\alpha, -}(r, s)\coloneq   \sum \limits_{n=0}^\infty E_{k,1n}^{\alpha}\left(\frac{s}{r}\right)^{2n} \chi_{+}, \quad  E_{k, 2}^{\alpha, +}(r, s)\coloneq \sum\limits_{n=0}^\infty  E_{k,2n}^{\alpha}\left(\frac{s}{r}\right)^{2n}\chi_{+}, \\
  E_{k, 2}^{\alpha, -}(r, s)\coloneq \sum\limits_{n=0}^\infty  E_{k,2n}^{\alpha}\left(\frac{r}{s}\right)^{2n} \chi_{-}, \quad E_{k, 1}^{\alpha, -}(r, s)\coloneq \sum \limits_{n=0}^\infty E_{k,1n}^{\alpha}\left(\frac{r}{s}\right)^{2n} \chi_{-}.
\end{align}

 \begin{lemma}\label{derivative estimated for RT}
For $E_{k, 1n}^\alpha$, $E_{k, 2n}^\alpha$, $A_{k,1 n}^\alpha$ and  $A_{k,2 n}^\alpha$, and $N\geq 1$,  we have
\begin{align*}
     \bigg|\frac{\dd^N E_{k, 1n}^\alpha}{\dd k^N}\bigg| \lesssim \frac{1}{k^N} \frac{1}{n+1},\quad  \bigg|\frac{\dd^N E_{k,2n}^{\alpha}}{\dd k^N}\bigg|\lesssim \frac{1}{k^N} \frac{1}{n+1}.
\end{align*}
\end{lemma}
\begin{proof}
    See Appendix B.
\end{proof}
Using Lemma \ref{derivative estimated for RT}, a direct computation gives
\begin{align*}
     \bigg|\frac{\dd^N E_{k, 1}^{\alpha, +}}{\dd k^N} \bigg|,\  \bigg|\frac{\dd^N E_{k, 2}^{\alpha, +}}{\dd k^N} \bigg|&\lesssim \frac{1}{k^N} \left|\ln\!\Big({1-\Big(\frac{s}{r}\Big)^2}\Big)\right|\chi_+,\\
      \bigg|\frac{\dd^N E_{k, 1}^{\alpha, -}}{\dd k^N} \bigg|,\  \bigg|\frac{\dd^N E_{k, 2}^{\alpha, -}}{\dd k^N} \bigg| &\lesssim \frac{1}{k^N} \left|\ln\!\Big({1-\Big(\frac{s}{r}\Big)^2}\Big)\right|\chi_{-}.
\end{align*}

It follows from these estimates that $ \widetilde K_{k, 2}^{\alpha, \textup{rem}} (r, s)$ satisfies the condition \eqref{uniform bound} and \eqref{uniform difference bound} of Theorem \ref{multiplier theorem for spherical harmonic decomposition}. Therefore,   $\widetilde{\mathcal R}_2^{\alpha, \textup{rem}}$ is bounded in $L^p(\frac{\dd r}{r} \cdot \dd \omega)$. Thus we finished the proof of Proposition \ref{bdd for Riesz transform}.

In conclusion, if $\max \Big\{0, \frac{\sigma+\alpha}{d}\Big\}<\frac{1}{p}<\min \Big\{1, \frac{d-\sigma}{d},  \frac{d+\alpha}{d}\Big\}$, the operator $\mathcal R^{\alpha}$ is bounded in $L^p(\mathbb R^d)$.

%

\appendix
\section{%
}
\label{s:proof-of-lem-deriv-est-for-remainder-term}
\label{appendix:proof-of-lemma-derivative estimate for the reminder term}
\begin{proof}[Proof of Lemma \ref{derivative estimate for the reminder term}.]
We only give the details for $A_{k,n}^{+}$, as the estimate for $A_{k,n}^{-}$ is similar.

For $N=1$, using the definition of polygamma functions, we have
\begin{align*}
    \frac{d A_{k,n}^{+} }{\dd k }&=A_{k,n}^{+}\bigg[(\psi (a_k+n+1)-\psi(a_k+b_k+n+1)) \frac{da_k}{\dd k} \\
    & \qquad\qquad +(\psi(b_k+n+1)-\psi(a_k+b_k+n+1)\frac{db_k}{\dd k} \bigg]\\
    &\eqqcolon  A_{k,n}^{+} G_{n,k}.
\end{align*}
The Leibniz rule gives
\begin{align*}
    \frac{d^{N+1} A_{k,n}^{+} }{\dd k^{N+1} }&=\sum_{m=0}^N\binom{N}{m} \frac{d^{m} A_{k,n}^{+} }{\dd k^{m} }\frac{d^{N-m} G_{n,k} }{\dd k^{N-m} }
    =A_{k,n}^{+}\sum_{i=0}^{M(N)}C_{i, N} \prod_{j=0}^{N} (G_{n,k}^{(j)})^{\alpha_{i,j, N}},
\end{align*}
where $\alpha_{i,j,N}$ are integers depends on $i,j, N$ and satisfying $\sum \limits_{j=0}^N \alpha_{i,j,N}(1+j)=N+1$, and $C_{i, N}, M(N)$ are some constants which depends on $N$.

For $G_{n,k}$, we have the following estimate, with $z\in(a_k+n+1, a_k+b_k+n+1)$
and $\widetilde z\in (b_k+n+1, a_k+b_k+n+1)$,
\begin{align*}
   |G_{n,k}|\leq |\psi'(z) b_k|\cdot \big|\tfrac{da_k}{\dd k}\big|+| \psi'(\widetilde z) a_k|\cdot\big|\tfrac{db_k}{\dd k}\big|.
\end{align*}
To estimate the derivatives of $G_{n,k}$ with respect to $k$,  we denote
\begin{align*}
    \fd_b\psi(a_k+n+1)=\psi(a_k+n+1)-\psi(a_k+b_k+n+1),\\
    \fd_a\psi(b_k+n+1)=\psi(b_k+n+1)-\psi(a_k+b_k+n+1).
\end{align*}
Then
\begin{align}\label{formula for $G_{n,k}$}
    G_{n,k}=\fd_b \psi(a_k+n+1) \frac{da_k}{\dd k}+ \fd_a \psi(b_k+n+1) \frac{db_k}{\dd k}.
\end{align}
It is easy to check that
\begin{align*}
    | \fd_b\psi(a_k+n+1)|\bigg|\frac{da_k}{\dd k}\bigg|\leq \min \bigg\{\frac{1}{n+1}, \frac{1}{k}\bigg\}\frac{1}{k},\\
     | \fd_a\psi(b_k+n+1)|\bigg|\frac{db_k}{\dd k}\bigg|\leq  \frac{1}{n+1} \frac{1}{k}.
\end{align*}
By the \FDB{}, we have
\begin{align*}
    &\frac{\dd^N\fd_b\psi(a_k+n+1)}{\dd k^N}\\
    &=\sum_{\sum_{j=1}^N jm_j = N} C_{m_1,\dots, m_N} \psi^{(\sum_{j=1}^N m_j)} (a_k+n+1)\prod_{j=1}^{N}(a_k^{(j)})^{m_j}\\
    &\quad-\sum_{\sum_{j=1}^N jn_j = N}  C_{n_1,\dots, n_N} \psi^{(\sum_{j=1}^N n_j)} (a_k+b_k+n+1)\prod_{j=1}^{N}(a_k^{(j)}+b_k^{(j)})^{n_j}\\
    &= \hspace{-0.8em}\sum_{\quad {\sum_{j=1}^N jm_j = N}}\hspace{-1.5em} C_{m_1,\dots, m_N}\!\big[\psi^{(\sum_{j=1}^N m_j)}(a_k{+}n{+}1)-\psi^{(\sum_{j=1}^N m_j)} (a_k{+}b_k{+}n{+}1)\big]\!\prod_{j=1}^{N}(a_k^{(j)})^{m_j}\\
        &\quad-\hspace{-1.5em}\sum_{\quad {\sum_{j=1}^N jn_j = N}}\hspace{-1.3em}  C_{n_1,\dots, n_N} \psi^{(\sum_{j=1}^N n_j)} (a_k+b_k+n+1)\prod_{j=1}^{N}\sum_{\substack{\alpha_j=0}}^{n_{j-1}}\! C_{\alpha_j,n_j}(a_k^{(j)})^{\alpha_j}(b_k^{(j)})^{n_j-\alpha_j} .
    \end{align*}
    Thus, by \eqref{estimates for $a_k$} and \eqref{estimates for $b_k$},  we have the following estimate
  \begin{align}\label{estimate for difference of phi b}
     & \left|\frac{\dd^N\fd_b\psi(a_k+n+1)}{\dd k^N}\right|\nonumber\\
       \lesssim  & \sum_{\sum_{j=1}^N jm_j = N}\psi^{(\sum_{j=1}^N m_j)+1}(z)| b_k |\prod_{j=1}^{N}(a_k^{(j)})^{m_j}\nonumber\\
       & -\sum_{\sum_{j=1}^N jn_j = N}\frac{1}{n+1}\frac{1}{ k^{(\sum_{j=1}^N n_j)-1} }\sum_{\substack{\alpha_j=0}}^{n_{j-1}}\frac{1}{k^{\sum_{j=1}^N j \alpha_j +(j+1)(n_j-\alpha_j) -\alpha_1}}\nonumber\\
      \lesssim & \sum_{\sum_{j=1}^N jm_j = N} \frac{1}{n+1}\frac{1}{k^{(\sum_{j=1}^N m_j)+1}}\frac{1}{k^{\sum_{j=1}^N j m_j-m_1}}\nonumber\\
      & -\sum_{\substack{\sum_{j=1}^N jm_j = N},   \alpha_1< N}\frac{1}{n+1}\frac{1}{ k^{(\sum_{j=1}^N n_j)-1} }\frac{1}{k^{ N+N -\alpha_1}}\nonumber\\
    \lesssim & \frac{1}{n+1} \frac{1}{k^{N+1}}.
    \end{align}

Similarly, for $  \fd_a\psi(b_k+n+1)$, we have
\begin{align*}
    & \frac{\dd^N\fd_a\psi(b_k+n+1)}{\dd k^N}\\
    =& \sum_{\sum_{j=1}^N jm_j = N}C_{m_1,\dots, m_N} \psi^{(\sum_{j=1}^N m_j)} (b_k+n+1)\prod_{j=1}^{N}(b_k^{(j)})^{m_j}\\
    &  -\sum_{\sum_{j=1}^N jn_j = N}  C_{n_1,\dots, n_N} \psi^{(\sum_{j=1}^N n_j)} (a_k+b_k+n+1)\prod_{j=1}^{N}(a_k^{(j)}+b_k^{(j)})^{n_j}\\
  =&  \sum_{\sum_{j=1}^N jm_j = N}C_{m_1,\dots, m_N} \psi^{(\sum_{j=1}^N m_j)} (b_k+n+1)\prod_{j=1}^{N}(b_k^{(j)})^{m_j}\\
  & -\sum_{\sum_{j=1}^N j n_j=N}\hspace{-1em} C_{n_1\dots n_N} \psi^{(\sum_{j=1}^N n_j)} (a_k+b_k+n+1)\prod_{j=1}^{N}(\sum_{\substack{\alpha_j=0\\ }}^{n_j} C_{\alpha_j,n_j}a_k^{(j)(n_j-\alpha_j)}b_k^{(j)\alpha_j}).
    \end{align*}
Again, \eqref{estimates for $a_k$} and \eqref{estimates for $b_k$} yields
\begin{align} \label{estimate for difference of phi a}
    \left| \frac{\dd^N\fd_a\psi(b_k+n+1)}{\dd k^N}\right|&\lesssim \sum_{\substack{\sum_{j=1}^N jm_j=N}}\frac{1}{n+1} \frac{1}{k^{\sum_{j=1}^N (j+1)m_j}}\nonumber\\
   &\quad + \sum_{\substack{\sum_{j=1}^N jn_j=N\  \\ \alpha_1\leq  n_1 \leq N}}\hspace{-1em} \frac{1}{n+1} \frac{1}{k^{n_1+\dots +n_N-1}} \frac{1}{k^{\sum_{j=1}^N j n_j-(n_1-\alpha_1)}}\nonumber\\
   &\leq  \frac{1}{(n+1)k^{N-1}}.
    \end{align}
Using Leibniz rule for \eqref{formula for $G_{n,k}$}, we have
\begin{align*}
     G_{n,k}^{(N)}=\sum_{j=0}^{N}\!\binom{N}{j} \! \left( (\fd_b \psi(a_k+n+1))^{(j)} a_k^{(N-j+1)}\!+\!(\fd_a \psi(b_k+n+1))^{(j)} b_k^{(N-j+1)}\right).
\end{align*}
By \eqref{estimate for difference of phi b} and \eqref{estimate for difference of phi a}, we get
\begin{align}
   |  G_{n,k}^{(N)}|&\lesssim
\sum_{j=1}^{N}\frac{1}{n+1} \frac{1}{k^{j+1}}\frac{1}{k^{N-j+1}} +\sum_{j=0}^{N} \frac{1}{n+1} \frac{1}{k^{j-1}}\frac{1}{k^{N-j+2}}\nonumber\\
&\lesssim \frac{1}{n+1} \frac{1}{k^{N+1}}.
\end{align}

In view of the boundedness of $A_{k,n}^+$, we have
\[
    A_{k,n}^{+ (N+1)}  \lesssim \frac{1}{n+1} \frac{1}{k^{N+1}}.
\]
This finishes the proof.
\end{proof}

\section{%
}
\label{s:proof-of-lem-deriv-est-for-RT}
\begin{proof}[Proof of Lemma \ref{derivative estimated for RT}] Using the asymptotics of the Gamma function, we know that $E_{k, 1n}^\alpha, E_{k, 2n}^\alpha$ are uniformly bounded in $n$ and $k$. Hence,
\begin{align*}
    \frac{dE_{k, 1n}^\alpha}{dk}&=E_{k, 1n}^\alpha\Big[\left\{(\psi(a_k+n+1)-\psi(a_k-b_k+n+1))\right\}\frac{da_k}{dk}\\
    &+\left\{(\psi(a_k-b_k-\alpha/2+n+1)-\psi(a_k-\alpha/2+n+1))\right\}\frac{da_k}{dk}\\
    &+\left\{\psi(a_k-b_k+n+1)-\psi(-b_k+n+1) \right\}\frac{d b_k}{dk}\\
    &+\left\{\psi(-b_k-\alpha/2+n+1)-\psi(a_k-b_k-\alpha/2+n+1) \right\}\frac{d b_k}{dk}\Big]\\
    &\coloneq E_{k,1n}^\alpha G_{k,1n}^\alpha.
\end{align*}
By induction, we have
\begin{align}
    \frac{d^NE_{k, 1n}^\alpha}{dk^N} = E_{k, 1n}^\alpha \sum_{i=0}^{M(N)} C_{i, N} \prod_{j=0}^N \left(\frac{d^j G_{k, 1n}^\alpha}{dk^j}\right)^{\alpha_{i,j,N}}\end{align}
where $\alpha_{i,j,N}$ are integers that depend on $i,j, N$, satisfying $\sum \limits_{j=0}^N \alpha_{i,j,N}(1+j)=N+1$; and $C_{i, N}, M(N)$ are some constants which depend on $N$.
It is easy to check that
\begin{align}
    | G_{k, 1n}^\alpha| \lesssim \frac{1}{k} \frac{1}{n+1}.
\end{align}
Next, we compute the derivatives of $G_{k, 1n}^\alpha$.

  Using the \FDB, we have
 \begin{align*}
    &\quad\frac{\dd^N{\left\{(\psi(a_k+n+1)-\psi(a_k-b_k+n+1))\right\}}}{\dd k^N}\\
    &=\sum_{\sum_{j=1}^N jm_j = N} C_{m_1,\dots, m_N} \psi^{(\sum_{j=1}^N m_j)} (a_k+n+1)\prod_{j=1}^{N}(a_k^{(j)})^{m_j}\\
    &\quad-\sum_{\sum_{j=1}^N jn_j = N}  C_{n_1,\dots, n_N} \psi^{(\sum_{j=1}^N n_j)} (a_k-b_k+n+1)\prod_{j=1}^{N}(a_k^{(j)}-b_k^{(j)})^{n_j}\\
    &= \hspace{-0.8em}\sum_{\quad {\sum_{j=1}^N jm_j = N}}\hspace{-1.5em} C_{m_1,\dots, m_N}\!\big[\psi^{(\sum_{j=1}^N m_j)}(a_k{+}n{+}1)-\psi^{(\sum_{j=1}^N m_j)} (a_k{-}b_k{+}n{+}1)\big]\!\prod_{j=1}^{N}(a_k^{(j)})^{m_j}\\
        &\quad-\hspace{-1.5em}\sum_{\quad {\sum_{j=1}^N jn_j = N}}\hspace{-1.3em}  C_{n_1,\dots, n_N} \psi^{(\sum_{j=1}^N n_j)} (a_k-b_k+n+1)\prod_{j=1}^{N}\sum_{\substack{\alpha_j=0}}^{n_{j-1}}\! C_{\alpha_j,n_j}(a_k^{(j)})^{\alpha_j}(b_k^{(j)})^{n_j-\alpha_j} .
    \end{align*}
Hence,
   \begin{align*}
     & \frac{\dd^N{\left\{(\psi(a_k+n+1)-\psi(a_k-b_k+n+1))\right\}}}{\dd k^N}\\
       \lesssim  & \sum_{\sum_{j=1}^N jm_j = N}\psi^{(\sum_{j=1}^N m_j)+1}(z)| b_k |\prod_{j=1}^{N}(a_k^{(j)})^{m_j}\nonumber\\
       & -\sum_{\sum_{j=1}^N jn_j = N}\frac{1}{n+1}\frac{1}{ k^{(\sum_{j=1}^N n_j)-1} }\sum_{\substack{\alpha_j=0}}^{n_{j-1}}\frac{1}{k^{\sum_{j=1}^N j \alpha_j +(j+1)(n_j-\alpha_j) -\alpha_1}}\nonumber\\
      \lesssim & \sum_{\sum_{j=1}^N jm_j = N} \frac{1}{n+1}\frac{1}{k^{(\sum_{j=1}^N m_j)+1}}\frac{1}{k^{\sum_{j=1}^N j m_j-m_1}}\nonumber\\
      & -\sum_{\substack{\sum_{j=1}^N jn_j = N},   \alpha_1< N}\frac{1}{n+1}\frac{1}{ k^{(\sum_{j=1}^N n_j)-1} }\frac{1}{k^{ N+N -\alpha_1}}\nonumber\\
    \lesssim & \frac{1}{n+1} \frac{1}{k^{N+1}}.
    \end{align*}
An analogous calculation gives that
\begin{align*}
   \frac{\dd^N{\left\{(\psi(a_k-b_k-\alpha/2+n+1)-\psi(a_k-\alpha/2+n+1))\right\}}}{\dd k^N} \lesssim \frac{1}{n+1} \frac{1}{k^{ N+1}}.
\end{align*}
For $\psi(a_k-b_k+n+1)-\psi(-b_k+n+1)$, \FDB{} gives
\begin{align*}
   &\quad \frac{d^N\{\psi(a_k-b_k+n+1)-\psi(-b_k+n+1)\}}{dk^N} \\
  &=\sum_{\sum_{j=1}^N jm_j = N} C_{m_1,\dots, m_N} \psi^{(\sum_{j=1}^N m_j)} (a_k-b_k+n+1)\prod_{j=1}^{N}(a_k^{(j)}-b_k^{(j)})^{m_j}\\
    &\quad-\sum_{\sum_{j=1}^N jn_j = N}  C_{n_1,\dots, n_N} \psi^{(\sum_{j=1}^N n_j)} (-b_k+n+1)\prod_{j=1}^{N}(-b_k^{(j)})^{n_j}\\
    &= \hspace{-0.8em}\sum_{\quad {\sum_{j=1}^N jm_j = N}}\hspace{-1.5em} C_{m_1,\dots, m_N}\!\psi^{(\sum_{j=1}^N m_j)}(a_k{-}b_k+n{+}1)\!\prod_{j=1}^{N}\sum_{\alpha_j=0}^{m_j} C_{\alpha_j,m_j}(a_k^{(j)})^{\alpha_j}(b_k^{(j)})^{m_j-\alpha_j}\\
        &\quad-\hspace{-1.5em}\sum_{\quad {\sum_{j=1}^N jn_j = N}}\hspace{-1.3em}  C_{n_1,\dots, n_N} \psi^{(\sum_{j=1}^N n_j)} (-b_k+n+1)\prod_{j=1}^{N}(-b_k^{(j)})^{n_j} .
        \end{align*}
Hence,
\begin{align*}
    &\quad\left| \frac{d^N\{\psi(a_k-b_k+n+1)-\psi(-b_k+n+1)\}}{dk^N}\right|\\
     & \lesssim \sum_{\sum_{j=1}^N jm_j = N, \alpha_j\leq m_j} \frac{1}{n+1} \frac{1}{k^{\sum_{j=1}^N m_j-1}}\frac{1}{k^{\sum_{j=0}^Nj\alpha_j +(j+1)(m_j-\alpha_j)-\alpha_1}}\\
     &\quad +  \sum_{\sum_{j=1}^N jn_j = N}\frac{1}{n+1} \frac{1}{k^{\sum_{j=1}^N (j+1)n_j}}\\
     &\lesssim \frac{1}{n+1} \frac{1}{k^{N-1}}.
\end{align*}
Similarly, we have
\begin{align*}
 & \quad\left| \frac{d^N\{\psi(-b_k-\alpha/2+n+1)-\psi(a_k-b_k+n+1)\}}{dk^N}\right|\lesssim  \frac{1}{n+1} \frac{1}{k^{N-1}}.
\end{align*}
\begin{align}
    \left|\frac{d^N G_{k,1n}^\alpha}{dk^N}\right| \lesssim \frac{1}{n+1} \frac{1}{k^{N+1}}.
\end{align}
Hence, we have
\begin{align}
     \left|\frac{d^N E_{k,1n}^\alpha}{dk^N}\right| \lesssim \frac{1}{n+1} \frac{1}{k^{N+1}}.
\end{align}
The estimates for $E_{k,2n}^\alpha$ are similar, so we omit the details.
\end{proof}

\newpage
\begin{table}
\centering
\caption{Summary of notation used in this paper.}\label{table-notation}
\begin{tabular}{ll}
\toprule
Notation               & Description        \\
\hline
 $\fd ^N$                      &  $N$th order (forward) finite  difference                \\
 $\bessel_{\mu}$                      &  Bessel transform   of order $\mu$              \\
  $\hankel_{\nu}$                      &  Hankel transform   of order  $\nu$             \\
$\mathbb P_k$ &  Projection from $L^2(\mathbb S^{d-1})$ to the space of spherical harmonics of degree $k$\\
$Y_{k,l}(\omega)$ & The $l$th spherical harmonic of degree $k$\\
$\sh_{k,l}$ & $\sh_{k,l}=\{f(r)Y_{k,l}(\omega) : f(r) \in L^{2}(\mathbb  R_{+}; r^{ d-1}\dd{r}) \}$\\
$\mathbb H_{0; k,l}$ & $-\Delta $ restricted to $\sh_{k,l}$ \\
$\mathbb H_{ k,l}$ & $\la$ restricted to $\sh_{k,l}$\\
$\lambda_0$ & $\lambda_0=\frac{d-2}{2}$\\
$\mu_k$ & $\mu_k=\frac{d-2}{2}+k$\\[0.3em]
$\nu_k$ & $\nu_k=\sqrt{\mu_k^2+a}$\\
$e_k(r, \lambda)$&$e_k(r, \lambda)=(r\lambda)^{-\frac{d-2}{2}}J_{\mu_k}(r\lambda)$ \\
$\widetilde e_k(r, \lambda)$ &$\widetilde e_k(r, \lambda)=(r\lambda)^{-\frac{d-2}{2}}J_{\nu_k}(r\lambda)$ \\

$a_k$&$ a_k=\frac{\mu_k+\nu_k}{2}$ \\
 $b_k$ & $b_k=\frac{\mu_k-\nu_k}{2}$\\[0.3em]
$A_{k,n}^{+}$&$ A_{k,n}^{+}=\frac{\Gamma(a_{k}+n+1)\Gamma(b_{k}+n+1)}{\Gamma(a_{k}+b_{k}+n+1) \Gamma(n+1)}$\\[0.3em]
$A_{k,n}^{-}$ &$A_{k,n}^{-}= \frac{\Gamma(a_{k}+n+1)\Gamma(1-b_k+n)}{\Gamma(a_k-b_k+n+1) \Gamma(n+1)}$\\[0.3em]
$E_{k,n}^{+} $ & $E_{k,n}^{+}=A_{k,n}^{+}-1-\frac{a}{4(n+1)}$\\
$E_{k,n}^{-}$ & $E_{k,n}^{-}=A_{k,n}^{-}-1+\frac{a}{4(n+1)}$\\
$\chi_{\pm}$ &$\chi_{+}=\chi_{\{r/2<s<r\}}$ and $\chi_{-}=\chi_{\{r<s<2r\}}$ \\
$C_k^\alpha$ & $ C_k^\alpha = \frac{2 \sin( \pi \alpha/2) \sin( \pi (\mu_k-\nu_k)/2)}{\pi \sin( \pi (\mu_k-\nu_k+\alpha)/2)}(\alpha$ is not an even integer)\\
$ A_{k,1n}^{\alpha}$& $ A_{k,1n}^{\alpha}=  \frac{\Gamma(\frac{\mu_k+\nu_k}{2}+n+1)\Gamma(\nu_k-\frac{\alpha}{2}+n+1)\Gamma(\frac{\nu_k-\mu_k}{2}+n+1)\Gamma(-\frac{\alpha}{2}+n+1)}
 {\Gamma(\frac{\nu_k-\mu_k-\alpha}{2}+n+1)\Gamma(\nu_k+n+1)\Gamma(\frac{\nu_k+\mu_k-\alpha}{2}+n+1)\Gamma(n+1)}$\\
 $A_{k,2n}^\alpha$ & $ A_{k,2n}^\alpha=  \frac{\Gamma( \mu_k+\frac{\alpha}{2}+n+1)\Gamma(\frac{\nu_k+\mu_k}{2}+n+1)\Gamma(\frac{\alpha}{2}+n+1)\Gamma(\frac{\mu_k-\nu_k}{2}+n+1)}
 {\Gamma(\frac{\mu_k-\nu_k+\alpha}{2}+n+1)\Gamma(\frac{\nu_k+\mu_k+\alpha}{2}+n+1)\Gamma(\mu_k+n+1)\Gamma(n+1)}$\\
$E_{k, 1n}^\alpha$ &$E_{k, 1n}^\alpha=A_{k, 1n}^\alpha-1$\\
  $E_{k, 2n}^\alpha$  & $E_{k, 2n}^\alpha=A_{k, 2n}^\alpha-1$\\
\bottomrule
\end{tabular}
\end{table}

\begin{center}

\end{center}

\end{document}